\documentclass[preprint,12pt]{elsarticle}

\usepackage{amssymb}
\usepackage[T1]{fontenc}
\usepackage{amsthm}
\usepackage{amsmath}
\usepackage{mathrsfs}  
\usepackage{amsfonts}
\usepackage{mathtools}
\usepackage{graphics}
\usepackage{subfigure}
\usepackage{color}
\usepackage{hyperref}
\usepackage{mathabx}
\hypersetup{
	colorlinks=true,
	linkcolor=blue,
	filecolor=magenta,      
	urlcolor=cyan,
	pdftitle={Overleaf Example},
	pdfpagemode=FullScreen,
}
\usepackage[pagewise]{lineno}

\usepackage{algorithm}
\usepackage{algorithmic}
\makeatletter
\renewcommand*{\ALG@name}{Algorithm}
\makeatother

\newcommand{\sw}{s_{\mathrm{w}}}

\newcommand{\so}{s_{\mathrm{o}}}
\newcommand{\fw}{f_{\mathrm{w}}}

\newcommand{\fo}{f_{\mathrm{o}}}
\newcommand{\muw}{\mu_{\mathrm{w}}}
\newcommand{\mug}{\mu_{\mathrm{g}}}
\newcommand{\muo}{\mu_{\mathrm{o}}}
\newcommand{\um}{\mathcal{U}}
\newcommand{\Hm}{\mathcal{H}}
\newcommand{\nG}{\nu_\Gamma}
\newcommand{\wm}{\mathcal{W}}

\newcommand{\bmm}{\mathcal{B}}
\def\adj{\mathop{\mbox{Adj}}\nolimits}

\providecommand{\CS}{'\! S}

\newtheorem{theorem}{Theorem}[section]
\newtheorem{corollary}{Corollary}[section]
\newtheorem{claim}{Claim}[section]
\newtheorem{lemma}{Lemma}[section]
\newdefinition{definition}{Definition}[section]
\newdefinition{remark}{Remark}[section]

\begin{document}

\begin{frontmatter}



\title{Structure of undercompressive shock waves in three-phase flow in porous media}




\author[1]{Luis F. Lozano G.\corref{cor1}%
\fnref{fn1}}
\ead{luisfer99@gmail.com}
\author[2]{Ismael Ledoino}
\ead{ismael.sledoino@gmail.com}
\author[3]{Bradley J. Plohr }
\ead{bradley.j.plohr@gmail.com}
\author[4]{Dan Marchesin }
\ead{marchesi@impa.br}
\cortext[cor1]{Corresponding author}
\affiliation[1]{organization={Laboratory for Applied Mathematics (LAMAP), Federal University of Juiz de Fora},
            addressline={Rua J. L. Kelmer, s$/$n}, 
            postcode={36036-900},
            city={Juiz de Fora MG},
            country={Brazil}}
\affiliation[2]{organization={Laboratório Nacional de Computação Científica},
            addressline={Av. G. Vargas, 333},
            postcode={25651-076},
            city={Petrópolis, RJ},
            country={Brazil}}
\affiliation[3]{organization={},
            addressline={Los Alamos},
            postcode={ },
            city={New Mexico},
            country={USA}}
\affiliation[4]{organization={Instituto de Matem\'atica Pura e Aplicada},
            addressline={Estrada Dona Castorina 110},
            postcode={22460},
            city={Rio de Janeiro RJ},
            country={Brazil}}

\begin{abstract}
Undercompressive shocks are a special type of discontinuities that satisfy the viscous profile criterion rather than the Lax inequalities. These shocks can appear as a solution to systems of two or more conservation laws. This paper presents the construction of the undercompressive shock surface for two types of diffusion matrices. The first type is the identity matrix. The second one is the capillarity matrix associated with the proper modeling of the diffusive effects caused by capillary pressure. We show that the structure of the undercompressive surface for the different diffusion matrices is similar. We also show how the choice of the capillarity matrix influences the solutions to the Riemann problem. 
\end{abstract}




\end{frontmatter}

\section{Introduction}
The study of multi-phase flow in porous media is of great importance in various fields, ranging from Petroleum Engineering to Environmental Science. Understanding the complex dynamics and interactions of different fluid phases in such media is crucial for accurately predicting and optimizing the behavior of subsurface systems. One fundamental challenge in this area lies in solving the Riemann problem for three-phase flow, where intricate modeling difficulties and intricate nonlinear wave interactions arise \cite{Abreu2014,Abreu2006,Marchesin2001}.

The Riemann problem serves as a fundamental building block for analyzing the behavior of solutions to the governing system of partial differential equations (PDEs) in three-phase flow. In this context, the Riemann problem is an initial value problem, where the initial condition is specified as a pair of piecewise constant states. However, the system of PDEs for three-phase flow is not strictly hyperbolic, so non-classical waves can arise. 
Such waves can be discontinuous or continuous. Discontinuous waves are undercompressive shock waves, while continuous waves are transitional rarefaction waves \cite{L.1990}, leading to intricate wave interactions that are difficult to predict and capture accurately. 

We refer the reader to \cite{Marchesin2001,mehrabi2020solution} and references therein for a detailed review of the three-phase flow theory. 
Previous works have studied system of PDEs where there is loss of hyperbolicity. For example, the conservation equations for three-phase flow with relative mobility functions of the Corey type \cite{corey1956three} present an isolated point (called an umbilic point) where the strict hyperbolicity fails \cite{isaacson1988riemann,castaneda2012singular,L.1992a,Marchesin2014,Schaeffer1987}. In the case of equations for three-phase flow with relative mobility functions based on the Stone model, the conservation equations present an elliptic region \cite{bell1986conservation,fayers1989extension,holden1990strict}. Elliptic regions are obtained generically by approximating the Stone model flux functions by quadratic polynomial flux function, e.g., \cite{Holdem1987,azevedo1995multiple,Matos2008}.    

In \cite{L.1990,L.1992a}, Isaacson {\it et al.} identified undercompressive waves in the Riemann solution of systems of two conservation laws with general diffusion matrices. Undercompressive shocks are special in that, unlike classical shocks, the Rankine-Hugoniot constraints linking the shock speed and the states on either side of the shock interface are insufficient to determine the shock speed and exact amplitude of waves emerging along outgoing characteristics when the shock wave is perturbed. It was shown in \cite{Kulikov1968,Mailybaev2004,Dan2006} that the requirement of the existence of viscous profiles provides exactly the number of additional equations to resolve this indeterminacy.

In \cite{Kulikov1968,Mailybaev2004,Dan2006}, the authors explicitly introduced the additional equations resulting from the viscous profile requirement. They also developed a constructive method for sensitivity analysis under perturbations of problem parameters. In \cite{V.2002}, Azevedo, Marchesin, Plohr, and Zumbrun showed that, in the presence of nontrivial diffusion terms, such as those resulting from capillary pressure, it is not the elliptic region that plays the role of an instability region; rather, it is the region defined by Majda-Pego \cite{Majda1985}, which depends on the diffusion terms as well and contains the elliptic region. In \cite{Abreu2014,Abreu2006}, Abreu et al. performed two-dimensional simulations of flow in a heterogeneous porous medium involving non-classical waves. The front waves were remarkably stable. 

Undercompressive shocks form a two-parameter family. The focus of this paper is the construction of this family. For this purpose, we introduce the notion of wave manifold, which presents a geometric approach to the fundamental waves that appear in the solutions of Riemann problems for conservation law systems. The wave manifold is 3-dimensional, and undercompressive shocks form a two-dimensional surface. See \cite{issacson1992global,marchesin1994topology,azevedo2010topological}  and references therein for more details.

Undercompressive shocks are tied to the diffusive term (diffusion matrix) of the parabolic system associated with the system of conservation laws. An undercompressive shock represents an admissible discontinuity with viscous profile which is the connection between two saddle points. In ODE theory, this type of connection is structurally unstable, unlike connections between node and saddle that occur in classical Lax shocks, see \cite{L.1990,L.1992a,Guckenheimer1986}.

We consider two types of diffusion matrices: the first is the identity, and the second is the nonlinear diffusion matrix related to the capillary pressure effects. In the case of the identity matrix, the surface is constructed in analytic form. This construction is possible because undercompressive shocks and the associated orbits are  restricted along invariant lines, which can be computed analytically. When we consider the nonlinear matrix, such invariant lines do not exist. Numerical procedures are needed; see \ref{ApendiceA}.

This work is organized as follows. 
In Section~\ref{sec:MathModel}, we present the system of conservation laws that models three-phase flow in porous media under a few physical simplifications. We also introduce the Corey model with quadratic permeabilities for the fluid phases adopted in this work. Then we recall some basic facts of the theory for systems of conservation laws, bifurcation theory of Riemann solutions, and the wave curve method. Finally, we present the criterion of the viscous profile, which is used together with Lax's criteria to determine admissible discontinuities. Section \ref{sec:geomUSS} briefly introduces the necessary geometric framework to study the undercompressive wave surface within the context of the wave manifold. In particular, we define coordinate systems to visualize the undercompressive surface. We study in Section~\ref{section:reduced} the reduction of the original three-phase flow system to the scalar Buckley-Leverett conservation law. This reduction occurs along certain invariant lines in state space, which play a prominent role in constructing Riemann solutions. Then we present properties of a few exceptional points at the intersection of some bifurcation manifolds with the invariant lines. We also present a subdivision of the saturation triangle regarding the location of the umbilic point. In Section \ref{section:TransMap}, we compute the undercompressive map along the invariant lines. We also present some results for undercompressive rarefactions in this model. We construct the surface of undercompressive shocks in Section \ref{section:TransSurface} for the identity viscosity matrix. 
We construct the surface of undercompressive shocks for the realistic matrix in Section \ref{section:GeneralMatrix}. Surprisingly, this surface displays essentially the same topological structure found in the case of the identity viscosity matrix. Finally, in Section~\ref{sec:_conclus}, some conclusions are summarized.
The numerical experiments in this work used the specialized software ``ELI" \cite{ELI_web}, developed at the Laboratory of Fluid Dynamics of IMPA. This software package allowed us to obtain and explore: integral curves, Hugoniot curves, the main bifurcation loci, and phase portraits of dynamical systems and wave curves, which are all fundamental for the construction of Riemann solutions, for example \cite{V.2010,Azevedo2014,Lozano2018} and references therein. Numerical calculations in MATLAB were also performed.

\section{Mathematical Model}
\label{sec:MathModel}
In this section, we present the conservation law system that governs immiscible three-phase flow in porous media. The assumptions that we consider are $1D$ horizontal flow; incompressible fluids; negligible dispersion, gravitational and capillary effects; there are neither sources nor sinks; the fluids fill the entire pore rock space; the temperature is constant, and there is no mass interchange between phases. Let us consider the conservation of mass for each phase
\begin{equation}\label{eq:system1}
\frac{\partial\sw}{\partial t}+\frac{\partial \fw}{\partial x}=0, \qquad
\frac{\partial \so}{\partial t}+\frac{\partial \fo}{\partial x}=0,
\end{equation}
where $s_w(x,t)$ and $s_o(x,t)$, $x\in \mathbb{R}$ and $t>0$, are the saturations of water and oil, defined by $0\leq s_w\leq 1$, $0\leq s_o\leq 1$, and $0\leq s_g\leq 1$ (with $s_w+s_0+s_g=1$).
Let $f_w(s_w,s_o)$, $f_o(s_w,s_o)$, and $f_g(s_w,s_o)$ be the fractional flow functions for water, oil, and gas phases (with $1 = f_w+f_o+f_g$), which are defined by  
\begin{equation}\label{eq:flowfunct}
f_i= \lambda_i/\lambda_T,~~~~~\lambda_i =k\, k_{ri} /\mu_i,\qquad i = w,o,g,
\mbox{ and } \quad \lambda_T = \lambda_w+ \lambda_o + \lambda_g,
\end{equation}
where $\lambda_i$ is the relative mobility of phase $i$, $\lambda_T$ is the total mobility, $k$ is the permeability, $k_{ri}$ is the relative permeability of phase $i$, and $\mu_i$ is the viscosity of phase $i$, for $i=w,o,g$. As in \cite{V.2010,Azevedo2014,Andrade2017}, in this work, we consider quadratic relative permeabilities, {\it i.e.}, $k_{ri}=s_i^2,$ for $i=w,o,g$.
The system \eqref{eq:system1} can be written in vector form as
\begin{equation}
    \label{eq:sistem3vectorial}
    \dfrac{\partial U}{\partial t}+\dfrac{\partial F(U)}{\partial   x}=0.
\end{equation}
In this system, $U$ is the state $(\sw , \so )^T$ and the function $F:\Omega\subseteq \mathbb{R}^2\rightarrow\mathbb{R}^2$, $ F(U)$ is the flux $(\fw(U), \fo(U))^T $. The system \eqref{eq:sistem3vectorial} is called {\it hyperbolic} if the Jacobian matrix of the flux function $DF(U)$ has real eigenvalues, $\lambda_1(U)$ and $\lambda_2(U)$, $U\in\Omega.$ The eigenvalues are called {\it the characteristic speeds} of the system \eqref{eq:sistem3vectorial}. In the region of hyperbolicity, we have the natural ordering $\lambda_1(U) \leq \lambda_2(U)$.
We call $\lambda_s =\lambda_1$ the slow-family characteristic speed and $\lambda_f=\lambda_2$ the fast-family characteristic speed. A hyperbolic system for which  $\lambda_s\neq\lambda_f,\forall\, U\in\Omega$ is said to be strictly hyperbolic. In the case in which there exist points in $\Omega$ at which $\lambda_1(U)= \lambda_2(U)$, these points are known as umbilic points \cite{shearer1987} or coincidence points \cite{Medeiros1992}, and we say that the system is not strictly hyperbolic. The model considered in this work has an umbilic point in the interior of $\Omega$ \cite{Marchesin2014,Schaeffer1987}. 

We represent the saturation space $\Omega$ using barycentric coordinates in an equilateral triangle (called saturation triangle), see Fig.~\ref{fig:sattrian1}; here, we use the labels W, O, and G to designate the states with saturation $s_w = 1$, $s_o = 1$, and $s_g = 1$, respectively. The state $\um$ in the inner of saturation triangle (Fig.~\ref{fig:sattrian1}) is an umbilic point, and its coordinates are 
\begin{equation}\label{eq:CoordUm}
\mathcal{U} = \left(\mu_w,\mu_o,\mu_g\right)/(\mu_w+\mu_o+\mu_g).
\end{equation}
\begin{figure}[htbp]
   \centering
 \includegraphics[width=0.42\textwidth]{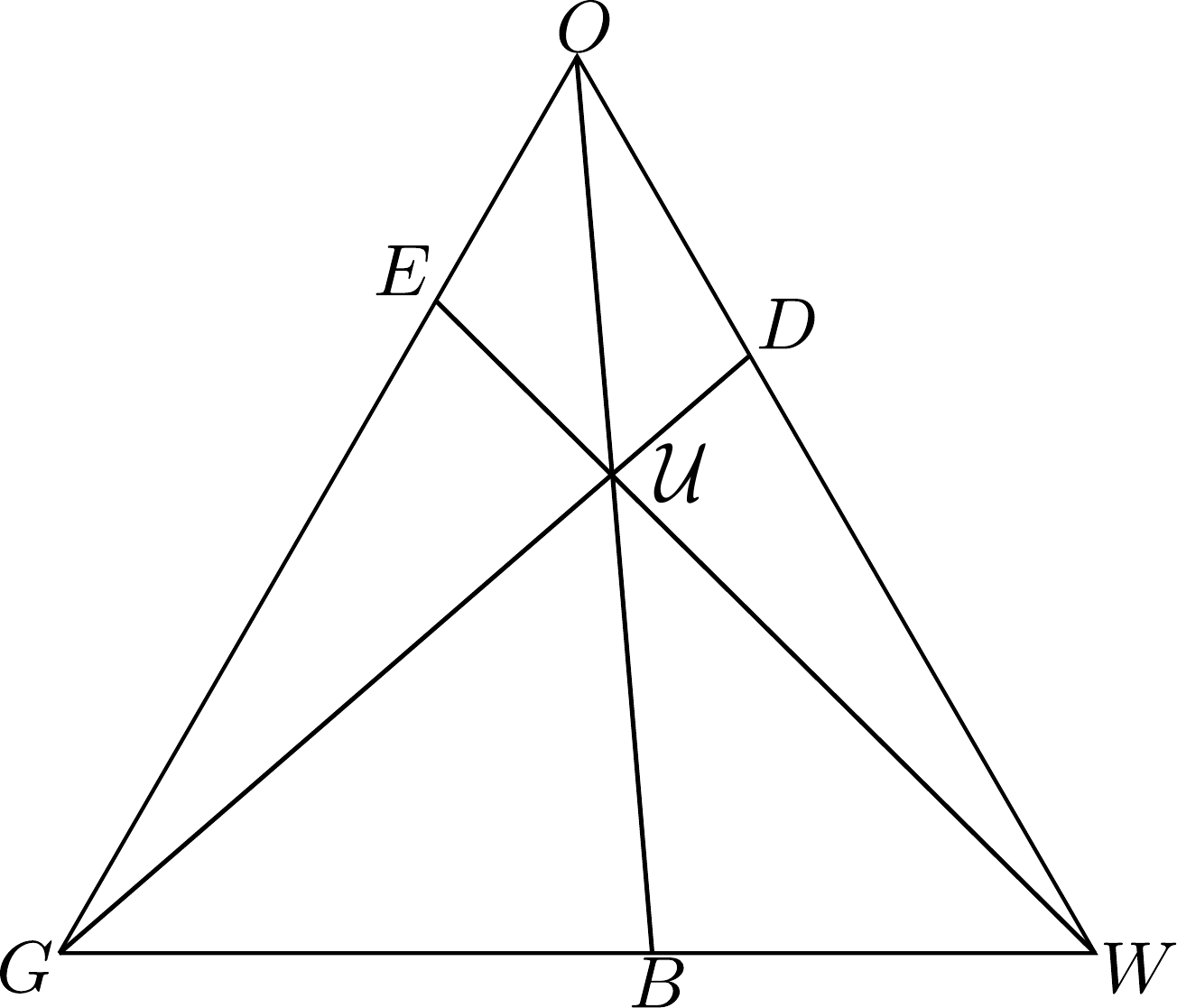}
  \caption{Saturation triangle $GWO$. The invariant lines [G, D], [W, E], and [O, B]. The state $\um$ represents the umbilic point of type $II_O$ in this work; see Corollary \ref{cor:NuUmbilic}. }
   \label{fig:sattrian1}
\end{figure}

We are interested in studying the Riemann problem for the system of partial differential equations (PDEs) \eqref{eq:system1}, {\it i.e.}, the system \eqref{eq:system1} subject to the initial conditions
\begin{equation}
\label{eq:initialCond_TP}
	U(x,0)=
	\begin{cases} 
	U_L, & x<0, \\
	U_R, & x>0, 
	\end{cases}  
\end{equation}
where $U_L$ and $U_R$ are constant. In some cases, we use $L$ and $R$. 

Scale invariance is a property of Riemann problems \eqref{eq:sistem3vectorial}, \eqref{eq:initialCond_TP}. It means that a coordinate change $(x, t) \rightarrow (c x, c t)$, with $c > 0$, changes neither the system of equations nor the initial conditions. Consequently, the solution of Riemann problems is expected to depend only on the ratio $\xi = x/t$. These invariant solutions are built from centered rarefaction waves, shock waves, composite waves, and constants states, e.g., \cite{Schaeffer1987,DeSouza1992,Lax1957}.
In general, the solution of a Riemann problem consists of a sequence of wave groups separated by constant states. A wave group refers to a contiguous sequence of waves, i.e., a sequence with no speed gaps between successive waves. 

\subsection{Rarefaction Waves}
Rarefaction waves are smooth self-similar solutions of system \eqref{eq:sistem3vectorial} represented by
\begin{equation}\label{eq:rarefact}
    U(x,t) = \widehat{U}(\xi),\qquad  \xi= x/t.
\end{equation}
Substituting \eqref{eq:rarefact} into the system \eqref{eq:sistem3vectorial}, we obtain the rarefaction curve by finding the solution to the eigenvalue problem
\begin{equation}\label{eq:systemEDO_rarefaction}
   (DF(\widehat{U}) - \lambda_i(\widehat{U}) I)r_i(\widehat{U}) = 0, 
\end{equation}
where $I$ is the identity matrix and $(\lambda_i(\widehat{U}),r_i(\widehat{U}))$, $i= s,\,f,$ is the eigenpair associated with the Jacobian matrix of the flux function $DF$. The eigenvector $r_i$ is parallel to $d\widehat{U} /d\xi$ with $\xi = \lambda_i,\,i= s,\,f$. Along such rarefaction curves, the corresponding eigenvalue (or characteristic speed) must be monotone, forcing these curves to stop at {\it inflection loci}, where the characteristic speed has an extremum \cite{Smoller1994}.

\subsection{Shock Waves}
Shock waves are bounded discontinuous self-similar solutions of the system \eqref{eq:sistem3vectorial} with a jump connecting the left state $U^-$ and the right state $U^+$ that satisfy the Rankine-Hugoniot condition:
 \begin{equation}\label{Si9}
F(U^+)-F(U^-)-\sigma\left( U^+ - U^-\right)=0,
\end{equation}
where $\sigma=\sigma(U^-;U^+)$ is the speed of propagation of the discontinuity. For a fixed $U^-$, the set of states $U$ such that the triple $(U^-,U,\sigma)$ satisfies the Rankine-Hugoniot condition \eqref{Si9} for some $\sigma$ is called the {\it Hugoniot locus}  for the state $U^-$ and is denoted by $\mathcal{H}(U^-)$.

Since by allowing discontinuous solutions, nonunique solutions appear in Riemann problems, the relation \eqref{Si9}  must be supplemented by an additional criterion that allows identifying unique and physically relevant solutions. Such criteria are known as entropy criteria or admissibility criteria (e.g., Lax \cite{Lax1957}, Liu \cite{Liu1975}, viscous profile \cite{Gelfand1959}, or vanishing adsorption \cite{Petrova2022}). In general, admissibility criteria are not equivalent \cite{Schecter1996, Gomes1989}. In this work, we construct solutions in which the shocks satisfy the viscous profile criterium.

\subsection{Composite Waves}
A composite wave of type rarefaction/shock (or shock/rarefaction) is a solution of the Riemann problem  \eqref{eq:sistem3vectorial}, \eqref{eq:initialCond_TP} consisting of a sequence of a rarefaction (or shock) wave followed by a shock wave (or rarefaction), with no segment of constant states separating them. We have an $i$-composite when the two waves are associated with the same $ith$-family, $i=s$ or $f$; otherwise, we have an undercompressive composite wave.

\subsection{Admissibility Criteria}
\subsubsection{Lax criterion}
 For conservation laws that are genuinely nonlinear and strictly hyperbolic, in \cite{Lax1957}, Lax introduced an eligibility criterion that associates discontinuous waves with family characteristics so that the characteristics of a family impinge on both sides of the discontinuity. In contrast, characteristics of the other family cross the discontinuity, experiencing a deviation.  
 \begin{definition}
 \label{def:Laxshocks}
 We use the following nomenclature for the discontinuities joining $U^-$ and  $U^+$ with characteristic speeds related to the propagation speed as below: 
    \begin{itemize}
\item Lax $s$-shock: $\lambda_s(U^+)<\sigma<\lambda_s(U^-) ~~\mbox{ and }~~\sigma<\lambda_f(U^+).$
\item Lax $f$-shock: $	\lambda_f(U^+)<\sigma<\lambda_f(U^-) ~~\mbox{ and }~~\lambda_s(U^-)<\sigma.$
\item  crossing: $ \lambda_s(U^-)<\sigma<\lambda_f(U^-) ~~\mbox{ and }~~\lambda_s(U^+)<\sigma<\lambda_f(U^+).$
\item overcompressive: $\lambda_f(U^+)<\sigma<\lambda_s(U^-).$
\end{itemize}
\end{definition}
Lax used the nomenclature $1$- and $2$-shock for slow and fast shocks, while we adopted the nomenclature $s$- and $f$-shocks. For more general conservation laws, characteristics can be tangent to discontinuity. In certain cases, we allow some equalities in the previous Lax configurations, giving way to a generalized Lax criterion \cite{Keyfitz1980}. If the crossing discontinuity has a viscous profile, we call it an {\it undercompressive shock} (also called {\it transitional shock} \cite{L.1990, L.1992a, Lozano2018, Aparecido1998}).
 \subsubsection{Viscous profile criterion}
 \label{sec:viscousprofile}
Consider the parabolic system of PDEs associated with the system \eqref{eq:sistem3vectorial}
\begin{equation}\label{eq:sistem2}
\dfrac{\partial U}{\partial t}+\dfrac{\partial F(U)}{\partial   x}=\epsilon\dfrac{\partial}{\partial x}\left(\bmm(U)\dfrac{\partial U}{\partial x}\right),
\end{equation}
the coefficient $\epsilon$ of the diffusive term in \eqref{eq:sistem2} is often rather small. The viscosity matrix $\bmm(U)$ depends on the relative permeabilities and viscosities of the three fluids and the partial derivatives of capillary pressures. Following \cite{V.2002} we define $\bmm(U) = Q(U)P'(U)$, where $Q(U)$ and $P'(U)$ are given by
  \begin{equation}\label{system6}
Q(U)= \left( \begin{array}{cc}
\lambda_w (1-f_w) & -\lambda_w f_o\\
-\lambda_o f_w & \lambda_o (1-f_o) \end{array} \right)
~\mbox{ and }~
P'(U)= \left( \begin{array}{cc}
\frac{\partial p_{wg}}{\partial s_w} & \frac{\partial p_{wg}}{\partial s_o}  \\
\frac{\partial p_{og}}{\partial s_w} & \frac{\partial p_{og}}{\partial s_o}  \end{array} \right).
\end{equation}
In \cite{V.2002}, it is shown that $\bmm(U)$ is symmetric and positive definite in the interior of the saturation triangle $\Omega$ and $\det \bmm(U)=0$ on $\partial \Omega$. We adopt the model from \cite{Abreu2014,V.2002} for the capillary pressure differences, then $P'(U)$ is given by
\begin{equation}\label{PUprima}
P'=
\left( \begin{array}{cc}
 \varsigma+\tau &  \tau \\
 \tau &  \tau \end{array} \right),\quad
 \varsigma =  \frac{c_{ow}}{2}  \frac{(1+s_w)}{s_w^{3/2}} ~\mbox{ and }~ \tau = \frac{c_{og}}{2} \frac{(2-s_w-s_o)}{(1-s_w-s_o)^{3/2}},
\end{equation}
where $c_{ow}$ and $c_{og}$ are positive constants. 
\begin{definition}
    A solution $U(x,t)\in\Omega$ of \eqref{eq:sistem2} is a {\it traveling wave} if there exist states $U^-$ and $U^+$ $\in\Omega$, a speed $\sigma\in\mathbb{R}$ and a vector function $\widetilde{U}(\xi)$ such that 
    \begin{equation}\label{V4A}
    U(x,t) = \widetilde{U}(\xi), \qquad \xi = (x-\sigma t)/\epsilon,
    \end{equation}
    and satisfies the following boundary condition
\begin{equation}\label{V4}
\lim_{\xi\rightarrow-\infty}\widetilde{U}(\xi)={ U^-} ~~\mbox{   and   }~~\lim_{\xi\rightarrow +\infty}\widetilde{U}(\xi)={ U^+}.
\end{equation}
\end{definition}

Substituting \eqref{V4A} and \eqref{V4} into \eqref{eq:sistem2}, we obtain (omitting tildes)
\begin{equation}\label{V8}
\mathcal{B}(U(\xi))\, \dfrac{d U}{d \xi }= F(U(\xi))-F( U^-)-\sigma \left(U(\xi)-{ U^-}\right),
\end{equation}
A traveling wave solution corresponds to a connection from $U^-$ to $U^+$ and is called a {\it viscous profile} of the shock wave.

Because $\bmm(U)$ is singular on the boundary of $\Omega$, we introduce the rescaled variable $\eta$, defined by
\begin{equation}\label{V7}
\frac{d \xi}{d\eta } = \det(\bmm(U)).
\end{equation}  
Using (\ref{V8}) in (\ref{V7}), we obtain the ODE system
\begin{equation}\label{V9}
 \dfrac{d U}{d \eta }=\adj(\bmm(U))\left[ F(U)-F( U^-)-\sigma \left(U-{ U^-}\right)\right],
 \end{equation}
where $\adj(\bmm(U))$  is the {\it adjugate matrix} (transpose of the matrix of cofactors). 

One standard procedure to find traveling wave solutions is analyzing the state space \cite{Volpert1994}. To do this, we locate and classify all equilibria of the system \eqref{V9}. Notice that $U^-$ and $U^+$ are equilibria of the vector field of \eqref{V9} and that all equilibria lie on the Hugoniot locus with $\sigma = \sigma(U^-;U^+)$ satisfying \eqref{Si9}.

\begin{definition}\label{def:typeEqui}
 	We have the following combinations between the type of orbit connecting $U^-$ and $U^+$ and the type of these equilibria:
\begin{itemize}
	\item Lax $s$-shock, $S_s~$, $U^-$ is a repeller and $U^+$ is a saddle.
	\item Lax $f$-shock, $S_f~$, $U^-$ is a saddle and $U^+$ is an attractor.
    \item Undercompressive shock, $S_u~$, $U^-$ and $U^+$ are saddles.
    \item Overcompressive shock, $S_O~$, $U^-$ is a repeller and $U^+$ is an attractor.
\end{itemize}
\end{definition}

\subsection{Characteristic discontinuities}
We use the notation found in \cite{V.2015} to describe the Riemann solution.
Discontinuities with or without viscous profiles are classified in Definition \ref{def:Laxshocks}. 
There are eight types of characteristic discontinuities, {\it i.e.}, with propagation velocity equal to a characteristic speed. When these discontinuities satisfy the viscous profile, we call them characteristic shocks. In this work, the following characteristic shocks arise: 
\begin{definition}\label{def:charequilibium}
 Characteristic shocks are limit cases of $i$-shock $i\in\{s,f,u\}$.
			\begin{enumerate}
				\item Left-char.$\,s$-shock:  $U^-$ is a repeller-saddle and $U^+$ is a saddle, and  satisfies $$\lambda_s(U^+)<\sigma=\lambda_s(U^-) ~\mbox{ and }~\sigma<\lambda_f(U^+).$$ 
				\item Right-char.$\,f$-shock: $U^-$ is a saddle and $U^+$ is a saddle-attractor and  satisfies $$ \lambda_f(U^+)=\sigma<\lambda_f(U^-) ~\mbox{ and }~\lambda_s(U^-)<\sigma.$$
                \item Left-char.$\,u$-shock:  $U^-$ is a saddle-attractor and $U^+$ is a saddle, and  satisfies  $$ \lambda_s(U^-)<\sigma=\lambda_f(U^-) ~\mbox{ and }~\lambda_s(U^+)<\sigma<\lambda_f(U^+).$$
		\end{enumerate}
\end{definition}
A shock of type $X \in \{s,f,u,O\}$ from states $A$ to $B$ is written as $A\xrightarrow{S_X} B$. A {\it prime} on the left or the right of $S$ indicates left- or right-characteristic shocks, respectively. For instance, $M\xrightarrow{\CS_u} N$ indicates an undercompressive shock from $M$ to $N$ with $\sigma(M;N)=\lambda_s(M)$ ({\it i.e.}, a left-char.$\,u$-shock).  On the other hand, for $i = \{s,f\}$, $A\xrightarrow{R_i} B$ means that there exists an $i$-rarefaction from $A$ to $B$ with speed equal to $\lambda_i$. In the wave sequence $A\xrightarrow{a} B\xrightarrow{b} C $, we use $v_f^a$ to represent the final velocity of the $a$-wave and $v_i^b$ to represent the initial velocity of the $b$-wave. The wave sequence is considered \emph{compatible} if and only if
\begin{eqnarray}\label{eq:Cond_Compat}
v_f^a \le v_i^b.
\end{eqnarray}
\subsection{Bifurcation loci}
This section defines subsets of $\Omega$ that play a fundamental role in constructing Riemann solutions that include undercompressive shocks; see \cite{L.1992a, V.2015}.
\begin{definition}\label{def:seconBif}
A state $U$ belongs to the {\it secondary bifurcation loci} for the family $i\in\{s,f\}$  if there exists a state $U'\neq U$ such that
\begin{equation}
U' \in \mathcal{H}(U) \mbox{  with  } \lambda_i(U') = \sigma(U;U') \mbox{  and  } l_i(U')(U'-U)=0,  
\end{equation}
where $l_i(U')$ is the left eigenvector of the Jacobian matrix $DF(U')$ associated with $\lambda_i(U')$.
\end{definition}  
 \begin{definition}\label{def;DobleContact}
A state $U$ belongs to the $(i,j)-${\it Double Contact loci} if there is a state $U'$ such that
\begin{equation}
U'\in \mathcal{H}(U) \mbox{ with } \lambda_i(U) = \sigma(U;U')=\lambda_j(U'),
\end{equation}
where the families $i$ and $j$ may be the same or different.
\end{definition}
In this work, we call {\it fast double contact} the $(f,f)$-double contact loci and {\it mixed double contact} the $(s,f)$ (or $(f,s)$)-double contact loci. 
\begin{definition}\label{def:iLRExt}
	Let $C$ be a set in state space. For $i \in\{s,f\}$,  we define the sets:
	 \begin{itemize}
     \item $i$-left-extension of $C$:\\
     $ \{U^+\in\Omega:\exists\, U^- \in C \,\mbox{ such that }\, U^+\in\mathcal{H}(U^-),\,\sigma(U^-;U^+)=\lambda_i(U^-)\}.$
     \item $i$-right-extension of $C$:\\
     $ \{U^+\in\Omega:\exists\, U^- \in C\,\mbox{ such that }\, U^+\in\mathcal{H}(U^-),\,\sigma(U^-;U^+)=\lambda_i(U^+)\}.$
 \end{itemize}
\end{definition}
\subsection{Wave Curves}\label{sec:wavecurve}
A {\it wave curve} is a parameterized curve in the state space containing waves of the same group. In the literature, the algorithm used to construct wave curves is usually the \emph{local continuation algorithm}, which Liu developed in \cite{Liu1975}, as a generalization of the Ole\u{\i}nik construction for the Riemann solution of scalar equations \cite{Oleinik1957}. However, in cases where the system is not strictly hyperbolic, the traditional approach fails, and it is necessary to consider nonlocal branches of the Hugoniot locus \cite{castaneda2016}. In this work, we use the \emph{global continuation algorithm} for constructing wave curves; see \cite{Lozano2018} for details. The basis of this algorithm lies in a list of compatible wave groups (see \cite{Schecter1996,V.2015}), which must satisfy the criterion for admissible viscous profiles. 

A wave curve is determined by its initial state, family, and its sense (forward or backward). A forward wave curve of the family $i$, starting at the state $L$, is a parametrization of the states $U$ in state space, with increasing speed, that are reached by $L$, on its right, by an $i$-wave group. We name it the {\it forward $i$-wave curve}, $\wm_i^+(L)$. Similarly, we name the {\it backward $i$-wave curve}, $\wm_i^-(R)$, starting at the state $R$, the parametrization of states $U$ in the state space with decreasing speed that reach the state $R$ on its left. 
\section{The geometry of the undercompressive shock surface}
\label{sec:geomUSS}
In this section, we briefly introduce the necessary geometric framework to study the undercompressive wave surface within the context of the wave manifold. We refer the reader to \cite{issacson1992global,marchesin1994topology,azevedo2010topological} for a more comprehensive treatment of the wave manifold.

We shall use the following notation for the quantities describing a shock wave: $U^- =(s_w^-,s_o^-)^T$ is the left state, $U^+ = (s_w^+,s_o^+)^T$ is the right state, and $\sigma$ is the shock speed. The triple  $(U^-,U^+,\sigma)\in\mathcal{P}:=\Omega\times\Omega\times \mathbb{R}$ must satisfy the Rankine-Hugoniot condition \eqref{Si9}. Then, the Rankine-Hugoniot condition constitutes two equations for the five variables $U^+$, $U^+$, and $\sigma$.

The solution set of the Rankine-Hugoniot condition contains not only the shock points, for which $U^+ \neq U^-$, but also the trivial solutions with $U^+ =  U^-$ and arbitrary $\sigma$. Through a blow-up procedure, the singularity in $U^+ =  U^-$ is removed; see \cite{issacson1992global,marchesin1994topology,azevedo2010topological} for details. In this way, it is possible to define {\it the wave manifold} $W\subset \mathcal{P}$ as a 3-dimensional submanifold comprising the shock and rarefaction points. 

As a $3$-dimensional manifold, $W$ can be visualized (except near certain exceptional points) using the 3-tuples $\left(U^-,\sigma\right)$ or $\left(U^+,\sigma\right)$ as local coordinates.
Three projection maps are useful for working with the wave manifold:
{\it left projection} $\pi^-: W \rightarrow \Omega$ is defined by
\begin{equation}\label{eq:projPim}
    \pi^-\left(U^-,U^+,\sigma\right):=U^-,
\end{equation}
which projects the $U^-$-coordinate; similarly {\it the right projection} $\pi^+$ projects the $U^+$-coordinate; and {\it the speed projection} $\sigma$ projects the $\sigma$-coordinate.

\subsection{Undercompressive map}
Under suitable assumptions about the viscous conservation law, undercompressive shock waves form an open codimension-1 submanifold  $\mathcal{T}$ of $W$, called the {\it undercompressive manifold}. On the other hand, the undercompressive shock surface $\mathcal{T}$ with boundary can also be visualized as a 2-dimensional manifold embedded in $W$. 

Consider the sets $\mathcal{D}_{\mathcal{T}}, \, \mathcal{D}_{\mathcal{T}}' \subset \Omega$ as the images of $\mathcal{T}$ under $\pi^-$ and $\pi^+$:
\begin{equation}
    \mathcal{D}_{\mathcal{T}} := \pi^-[\mathcal{T}], \quad \mbox{ and }\quad \mathcal{D}_{\mathcal{T}}':= \pi^+[\mathcal{T}].
\end{equation}
Suppose that $\pi^-|_{\mathcal{T}}$ (meaning the restriction of $\pi^-$ to the submanifold $\mathcal{T}$) is invertible. In this case, the inverse map $(\pi^-|_{\mathcal{T}})^{-1}$ maps $ \mathcal{D}_{\mathcal{T}}$ to  $\mathcal{T}$, and we can form the {\it undercompressive map} 
\begin{equation}
    \label{eq:tranMap}
    T:= \pi^+\circ (\pi^-|_{\mathcal{T}})^{-1},
\end{equation}
which maps $\mathcal{D}_{\mathcal{T}} \rightarrow \mathcal{D}'_{\mathcal{T}}: U^- \mapsto U^+$. We call the sets $\mathcal{D}_{\mathcal{T}}$ and $\mathcal{D}_{\mathcal{T}}'$ the  {\it undercompressive domain} and {\it codomain}, respectively. Likewise, we can define the {\it undercompressive speed map}
\begin{equation}
    \label{eq:tranMapsspeed}
    \sigma_{\mathcal{T}}:= \sigma \circ (\pi^-|_{\mathcal{T}})^{-1},
\end{equation}
which maps $\mathcal{D}_{\mathcal{T}} \rightarrow \mathbb{R}: U^- \mapsto \sigma$. The undercompressive submanifold $\mathcal{T}$ is a graph in $W$ lying over the undercompressive domain $\mathcal{D}_{\mathcal{T}}$ :
\begin{equation}
    \label{eq:setUngraphi}
\{(U^-,T(U^-),\sigma_{\mathcal{T}}(U^-)): U^-\in \mathcal{D}_{\mathcal{T}} \}.
\end{equation}
\begin{remark}
    The inverse undercompressive map $T^{-1}$ can be advantageous in certain situations. We can define it like how we define $T$ in \eqref{eq:tranMap}, as follows:
    \begin{equation}
    \label{eq:tranMapinv}
    T^{-1}:= \pi^-\circ (\pi^+|_{\mathcal{T}})^{-1},
\end{equation}
which maps $\mathcal{D}'_{\mathcal{T}} \rightarrow \mathcal{D}_{\mathcal{T}}: U^+ \mapsto U^-$. 
\end{remark}
\begin{remark}\label{rem:twowings}
    Often $\mathcal{T}$ is drawn in $\Omega$ (or $\Omega\times\mathbb{R}$) as two wedges: first when $(s_w, s_o) = (s_w^-, s_o^-)$ (or $(s_w, s_o,\sigma) = (s_w^-, s_o^-,\sigma))$ and second when $(s_w, s_o) = (s_w^+, s_o^+)$  (or $(s_w, s_o,\sigma) = (s_w^+, s_o^+,\sigma))$, giving the misleading impression that  $\mathcal{T}$ has two “palm leaves”. But this is only for better visualization of the saddle-saddle connections in the context of the dynamical systems.
\end{remark}

\section{The role of reduced two-phase flow}\label{section:reduced}
In this section, we study the reduction of our system given in \eqref{eq:system1} to the scalar Buckley-Leverett conservation law for two-phase flow in porous media. As we see in \cite{Azevedo2014} and Remark 3.3.8 in \cite{Lozano2018}, this reduction follows straightforwardly along each of the edges of the saturation triangle. However, for Corey\textquotesingle s model with quadratic relative permeability functions with $\mathcal{B}=I$, it also occurs along the straight lines $[G, D],[W, E]$, and $[O, B]$ as well, making them into invariant lines of our PDE \eqref{eq:system1}, see Fig.~\ref{fig:sattrian1}. Note that these invariant lines coincide with the secondary bifurcation locus. According to \cite{L.1992a} and \cite{DeSouza1992}, for viscosity matrix equal to the identity, the only admissible non-Lax shocks arise in the associated traveling wave ODEs \eqref{V9} as heteroclinic orbits lying on the invariant lines. These undercompressive shocks are essential in solving the general Riemann problem under the viscous profile criterion. Because of this reason, we are interested in characterizing these invariant lines and studying their properties. 
\subsection{Parameters and coordinates}
Points of the state space along the line $[G,D]$ (see Fig.~\ref{fig:sattrian1}) satisfy the identity
\begin{equation}\label{surface1}
\frac{s_w}{\mu_w} =  \frac{s_o}{\mu_o} =  \frac{s_w+s_o}{\mu_w + \mu_o}.
\end{equation}
Let us define along $[G,D]$ the {\it effective saturation} and {\it effective viscosity} respectively as
\begin{equation}\label{surface2}
s = s_w+s_o , ~\mbox{ and }~ \mu_{wo}= \mu_w+\mu_o.
\end{equation} 
Then, the line segment $[G,D]$ is parametrized by $s$ as
\begin{equation}\label{surface3}
s_w = \frac{\mu_w}{\mu_{wo}}s  ~\mbox{ and }~  s_o = \frac{\mu_o}{\mu_{wo}}s,
\end{equation}
where $0 \leq s \leq 1$. The system of PDEs \eqref{eq:sistem2} can be written as
\begin{equation}\label{sysSurface1}
\displaystyle\frac{\partial}{\partial t}\left(
\begin{array}{c}
s_w  \\
s_o
\end{array}
\right)
+\displaystyle\frac{\partial}{\partial x}\left(
\begin{array}{c}
F_w(s_w,s_o) \\
F_o(s_w,s_o)
\end{array}
\right)
= \displaystyle\frac{\partial^2}{\partial x^2}\left(
\begin{array}{c}
s_w \\
s_o
\end{array}
\right).
\end{equation}
Thus, we see that system \eqref{sysSurface1} reduces to the scalar Buckley-Leverett equation 
\begin{equation}\label{surface6}
\frac{\partial s}{\partial t} +\frac{\partial}{\partial x}f(s,\nu)=\frac{\partial^2 s}{\partial x^2},
\end{equation}
where the flux function $f$ and the viscosity ratio $\nu$ are
\begin{equation}\label{surface7}
f(s,\nu)  = \frac{s^2}{s^2+\nu (1-s)^2}, ~~~\nu = \frac{\mu_{w}+\mu_{o}}{\mu_g}. 
\end{equation}
The Rankine-Hugoniot condition \eqref{Si9} for a shock with speed $\sigma$ between $M$ and $N$ reduces to
\begin{equation}\label{RankHugParameter}
-\sigma(s_N -s_M) -f(s_N,\nu)-f(s_M,\nu)=0.
\end{equation}

\begin{remark}\label{re:paramS}
	The relationships \eqref{surface6}, \eqref{surface7} are valid over any invariant line $[G,D]$, $[O,B]$, and $[W,E]$, which are identified by the vertex where it begins, $\Gamma \in \{G,W,O\}$. Sometimes it is useful to express the final point of the invariant line, $\mathbb{B}\in\{D,E,B\}$, see Fig.~\ref{fig:sattrian1}. The choice of vertex specifies a relation between phase saturations: $\gamma \in \{g,w,o\}$ corresponds to the phase at the vertex $\Gamma$ and $\alpha$, $\beta$ represent the two interchangeable phases. Then, the quantity $s$, the viscosity ratio, and the effective viscosity are, respectively,  
	\begin{equation}\label{paramS}
	s = s_{\alpha}+s_{\beta},~~\nu_\Gamma =\mu_{\alpha\beta}/\mu_{\gamma},~~ \mbox { with }~~ \mu_{\alpha\beta} = \mu_{\alpha}+\mu_{\beta}.
	\end{equation}
To return to the original variables of the saturation triangle, we use
 	\begin{equation}\label{paramS2}
	s_{\alpha} = s\, \mu_{\alpha}/\mu_{\alpha\beta},~~ s_{\beta} = s\, \mu_{\beta}/\mu_{\alpha\beta}.
	\end{equation}
 With this parametrization, the vertices $G, W$, and $O$ always have value $s$ equal to zero, and states $D, E$, and $B$ on the opposite edge have value $s$ equal to one. 
\end{remark}
\begin{remark}
	We introduce further notation for states over any invariant line in ``effective" quantity. For example, let $M = (M_w, M_o)$ be a state in the saturation triangle belonging to $[G,D]$. Then, we write $s_M \in (0,1)$ to indicate that the state is specified by one effective saturation parameter.  
\end{remark}
Finally, we provide an identity for effective shocks. Given a speed and a left (or right) state in an invariant line, it returns the right (or left) state of an admissible shock. Let $M$ and $N$ be two states on an invariant line. Then substituting \eqref{surface7} into \eqref{RankHugParameter} we obtain
\begin{equation}\label{identity0}
\frac{\sigma}{\nu_\Gamma}\left[s_M^2+\nu_\Gamma (1-s_M)^2\right]\left[s_N^2+\nu_{\Gamma} (1-s_N)^2\right] = s_N+s_M-2\,s_N s_M.
\end{equation}
\subsection{Analytic expressions for certain states over invariant lines}
Considering that over the invariant lines, the flux function reduces to \eqref{surface7}, we derive a few expressions for states over these lines that help characterize the Riemann solutions for states in the saturation triangle. First, we fix the parameters $\mu_w, \mu_o $ and $\mu_g$, choose an invariant line by a specific $\Gamma$  and compute the value for $\nu_\Gamma$ using \eqref{paramS}. We write $f(s,\nu_\Gamma)=f(s)$ to represent the effective flux on this line. Then we calculate the characteristic speeds $\lambda_a$ and $\lambda_b$ for states over the invariant line. Hence the expressions:    
\begin{equation}\label{eigenoverbif}
\lambda_a(s)=\frac{2 s}{\mu_{\alpha \beta}D(s)}, ~
\lambda_b(s)=\frac{2 s (1-s)}{\mu_{\alpha \beta} \mu_\gamma D(s)^2},~
\mbox{   with   }~
D(s)=\frac{ s^2}{\mu_{\alpha \beta}}+\frac{ (1-s)^2}{\mu_\gamma}.
\end{equation}

\begin{remark}\label{rem:charspeed}
	It is clear that $\lambda_b$ can be obtained as $f'(s)$, the derivative of flux function \eqref{surface7}. Moreover, if the reduced state $s$ lies between the vertex and the umbilic point $~\um$, we have $\lambda_a(s)<\lambda_b(s)$. Conversely, if the state $s$ lies between the umbilic point and the edge opposite to the vertex, then $\lambda_b(s)<\lambda_a(s)$. We have   
	\begin{equation}\label{eq:charadef}
	\lambda_s(s)=\min \{\lambda_a(s),\lambda_b(s)\} ~~\mbox{ and }~~\lambda_f(s)=\max \{\lambda_a(s),\lambda_b(s)\},
	\end{equation}
	the slow and fast characteristic speeds for a state $s$ along any invariant line (Fig.~\ref{fig:autovsateffec}).
\end{remark}
\begin{figure}[ht]
	\begin{center}
		\centering
		\includegraphics[width=0.5\textwidth]{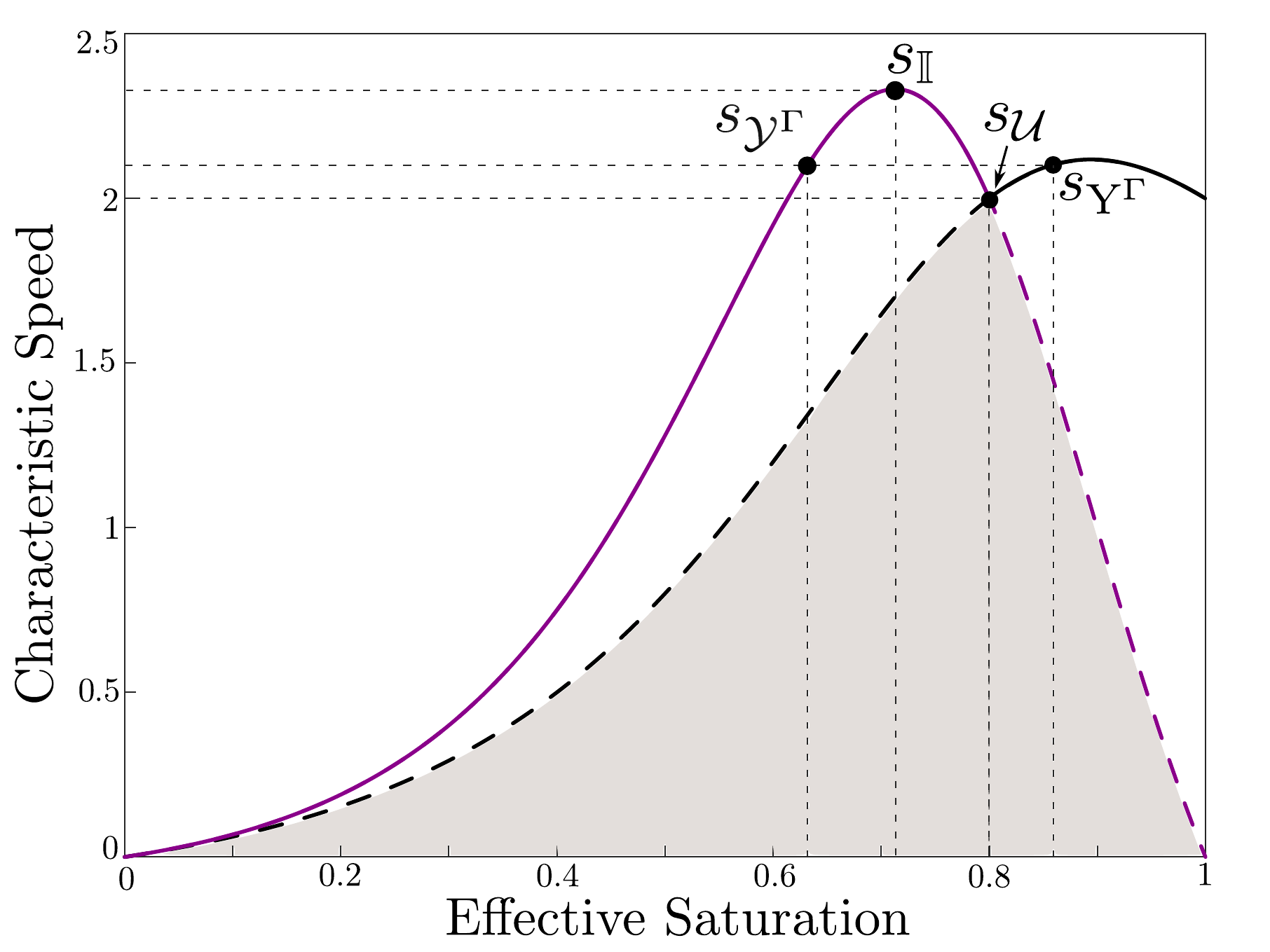}
		\caption{Graphs for expressions $\lambda_a$ (black curve) and $\lambda_b$ (purple curve) above the white region. The dashed (resp. solid) line above the shaded region is the slow (resp. fast) characteristic speed $\lambda_s$ (resp. $\lambda_f$) defined in \eqref{eq:charadef}. The horizontal axis corresponds to a parametrization of an invariant line in terms of effective saturation $s$, and the vertical axis is characteristic speed. The states $s_\um$, $s_{\mathcal{Y}^{\Gamma}}$ and $s_{\text{Y}^{\Gamma}}$ are given in \eqref{eq:umbilic}, \eqref{eq:DoubleContactExp}(a), and (b). The definition of $s_\mathbb{I}$ is given in Remark \ref{rem:inflection}.}
		\label{fig:autovsateffec}
	\end{center}
\end{figure}

From the definition of the umbilic point, we have $\lambda_a(s_\mathcal{U})=\lambda_b(s_\mathcal{U})$. Then from the expressions \eqref{eigenoverbif} we obtain
\begin{equation}\label{eq:umbilic}
s_{\mathcal{U}}=\nu_\Gamma/(1+\nu_\Gamma).
\end{equation}

\begin{remark}\label{rem:DoubleContact}
	It is possible to derive the expressions for two distinguished states over each invariant line. These states are denoted by $\mathcal{Y}^{\Gamma}$ and $\text{Y}^{\Gamma}$ ($\Gamma \in \{G,W,O\}$) and belong to the fast double contact locus. We assume that the umbilic point lies between $\mathcal{Y}^{\Gamma}$ and $\text{Y}^{\Gamma}$ ($s_{\mathcal{Y}^{\Gamma}}\in(0,s_\mathcal{U})$, $s_{\text{Y}^{\Gamma}}\in(s_\mathcal{U},1)$), then
	\begin{equation} \label{Doubledef}
	\text{Y}^{\Gamma}\in \mathcal{H}(\mathcal{Y}^{\Gamma})~\mbox{  and  }~\sigma(\mathcal{Y}^{\Gamma};\text{Y}^{\Gamma}) = \lambda_f(\mathcal{Y}^{\Gamma}) = \lambda_f(\text{Y}^{\Gamma}),
	\end{equation}
	the value $\sigma(\mathcal{Y}^{\Gamma};\text{Y}^{\Gamma})$ is the speed of the shock between $\mathcal{Y}^{\Gamma}$ and $\text{Y}^{\Gamma}$. From the expressions \eqref{eigenoverbif} and definitions \eqref{eq:charadef} and \eqref{Doubledef} we have
	\begin{equation}\label{eqDoubleC}
	\lambda_b(s_{\mathcal{Y}^{\Gamma}}) =\sigma(s_{\mathcal{Y}^{\Gamma}};s_{\text{Y}^{\Gamma}})=\lambda_a(s_{\text{Y}^{\Gamma}}) .
	\end{equation}

From the second and the last expression of \eqref{eqDoubleC} we obtain
\begin{equation}\label{eq:douiblecon1}
	s_{\text{Y}^{\Gamma}}=\frac{s_{\mathcal{Y}^{\Gamma}}\nu_\Gamma}{2 s_{\mathcal{Y}^{\Gamma}}^2 (\nu_\Gamma+1)-2\nu_\Gamma s_{\mathcal{Y}^{\Gamma}}+\nu_\Gamma}.
 \end{equation}
 Notice that the denominator of \eqref{eq:douiblecon1} is nonzero, and that by substituting \eqref{eq:douiblecon1} in the first and second equation of \eqref{eqDoubleC} we obtain	\begin{equation}\label{eqDoubleC3}
	2\left(s_{\mathcal{Y}^{\Gamma}}(\nu_\Gamma+1)-\nu_\Gamma\right)\left(2s^2_{\mathcal{Y}^{\Gamma}}(\nu_\Gamma +1)-\nu_\Gamma\right)s^2_{\mathcal{Y}^{\Gamma}}=0.
	\end{equation}
 Because we seek a solution different from the umbilic point, we arrive at the expressions in increasing order 
	\begin{equation}\label{eq:DoubleContactExp}
	s_{\mathcal{Y}^{\Gamma}} = \frac{1}{2}\sqrt{\frac{2 \nu_\Gamma}{\nu_\Gamma  + 1}},~~~~~
	s_{\text{Y}^{\Gamma}} =\frac{\nu_\Gamma +\sqrt{2\nu_\Gamma(\nu_\Gamma+1)}}{2(\nu_\Gamma+2)}.
	\end{equation}
\end{remark}

Now, consider the intersection point between the invariant line and the triangle edge transversal to that line. The value of $s$ at this point is one (see Remark \ref{re:paramS}). Let us compute $\mathbb{B}_1$ and $\mathbb{B}_2$, the slow and fast left-extensions of this intersection point, respectively. These states lie on the invariant line $[\Gamma , \mathbb{B}]$ with $\Gamma \in\{G,W,O\} $, $\mathbb{B}\in\{D,E,B\}$, and satisfy
\begin{eqnarray}\label{Ext1}
\lambda_s(s_{\mathbb{B}_1})=\sigma(s_{\mathbb{B}_1};1),\\ \label{Ext2}
\lambda_f(s_{\mathbb{B}_2})=\sigma(s_{\mathbb{B}_2};1).
\end{eqnarray}
Substituting equations \eqref{Ext1} and \eqref{Ext2} in \eqref{identity0}, we have the expressions in increasing order 
\begin{equation}\label{eq:extenD}
s_{\mathbb{B}_2}=\frac{\sqrt{\nu_\Gamma+1}-1}{\sqrt{\nu_\Gamma+1}}, ~~~~~s_{\mathbb{B}_1}=\frac{\nu_\Gamma}{2+\nu_\Gamma}.
\end{equation}
Finally, we define the extension of the umbilic point, {\it i.e.}, the state $\mathbb{B}_0$ on the invariant line $[\Gamma,\mathbb{B}] $ such that
\begin{equation}\label{eq:D0}
\lambda_{\mathcal{U}}:=\lambda_s(s_{\mathcal{U}})=\lambda_f(s_\mathcal{U})=\sigma(s_{\mathbb{B}_0};s_\mathcal{U}).
\end{equation}
Substituting \eqref{eq:umbilic} into \eqref{eigenoverbif} we obtain $\lambda_{\mathcal{U}}=2.$ Then using \eqref{identity0} and \eqref{eq:D0} we have
\begin{equation}\label{eq:D01}
(2 s_{\mathbb{B}_0}-1)(s_{\mathbb{B}_0}\nu_\Gamma+s_{\mathbb{B}_0}-\nu_\Gamma)=0,~~~~\mbox{  therefore   }~~~~ s_{\mathbb{B}_0} = 1/2.
\end{equation}
Figure \ref{fig:limitesBoundaries} shows the localization of states $D_0$, $D_1$ and $D_2$ along of $[G,D].$
\begin{figure}
		\begin{center}
			\includegraphics[width=0.42\textwidth]{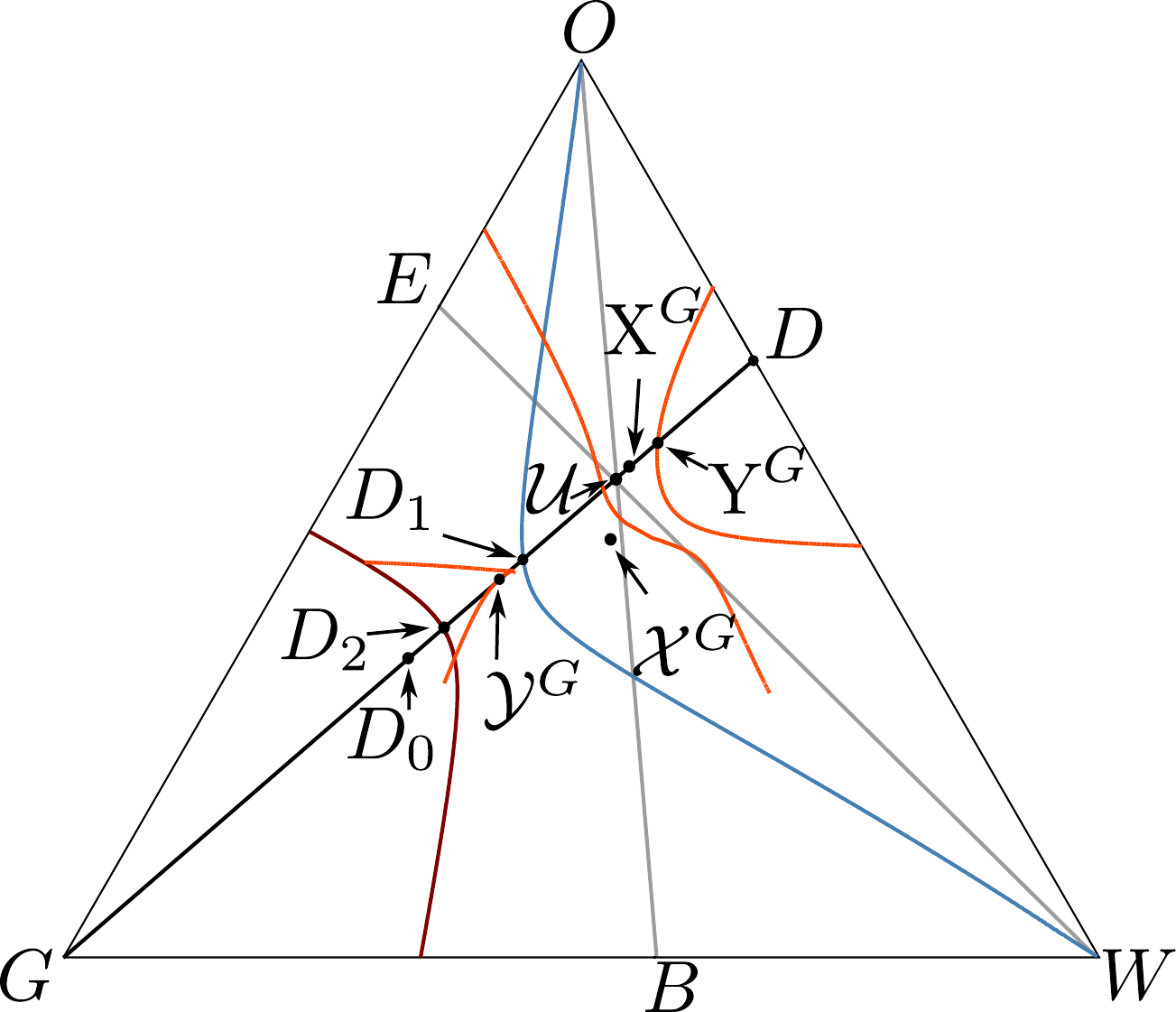}
		\end{center}
		\caption{Bifurcation curves in state space that contain the states $D_2=\mathbb{B}_2$ and $D_1=\mathbb{B}_1$  defined respectively in \eqref{eq:extenD}(a) and (b) for ${[G,D]}$; the states  $\mathcal{Y}^{G}$ and $\text{Y}^{G}$ defined respectively in \eqref{eq:DoubleContactExp} (a) and (b); and the pair of states of mixed double contact $\mathcal{X}^{G}$ and $\text{X}^{G}$ with $\text{X}^{G}\in[G,D]$. The blue (resp. red) curve is the left slow (fast) boundary extension of edge $[O, W]$. The orange curves are the double fast contact, and the gray ones are the invariant lines $[G, D]$, $[O, B]$, and $[W, E]$.}
		\label{fig:limitesBoundaries}
	\end{figure}

\begin{remark}\label{rem:D0E0B0}
	We denoted $D_0$, $E_0$, and $B_0$ to the extensions of the umbilic point that satisfy \eqref{eq:D0} along the invariant lines $[G,D]$, $[W,E]$ and $[O,B]$ respectively. From \eqref{eq:D01}, we have $s_{D_0}=1/2$, $s_{E_0}=1/2$ and $s_{B_0}=1/2$.   
\end{remark}
\subsection{Dependence of distinguished two-phase states on fluid viscosities}\label{section:dependence}
In this section, we study the location of certain important points, such as the umbilic point and the fast double and mixed contact points, when we vary the parameters $\mu_w,\mu_o$, and $\mu_g$. Moreover, it is possible to characterize the position of contact points in terms of the position of the umbilic point in the saturation triangle. All of these states lie on an invariant line, so it suffices to explain the variation of these states along a single invariant line; the other cases are similar. 

We recall the classification of the umbilic point given in \cite{Schaeffer1987} as type $I$ or $II$. In \cite{Asakura} and \cite{Marchesin2014}, the position of the umbilic point was characterized by a general family of models, of which the Corey model is a particular case.
\begin{definition}\label{def:traingleI}
	For the Corey Quad model, we define the triangle $\Omega_\mathcal{U}$ with vertices
	$(0,1/2),(1/2,0)$, and $(1/2,1/2)$, see shaded triangle in Fig.~\ref{fig:DivisionViscosidadesA2}(a).
\end{definition}

\begin{theorem}[\cite{Marchesin2014,Schaeffer1987}]\label{Thm:MatostypeI/II}
	For the Corey Quad model, the umbilic point is classified as:
	\begin{enumerate}
		\item type $I$ if ~~$\mathcal{U}$ lies inside $\Omega_\mathcal{U}$;
		\item type $II$ if ~~$\mathcal{U}$ lies outside $\Omega_\mathcal{U}$;
		\item border-type  $I/II$ if ~ $\mathcal{U}$ lies on the edges of $\Omega_\mathcal{U}$.
	\end{enumerate}
\end{theorem}

\begin{figure}
	\centering
	\subfigure[Regions where $\mathcal{U}$ is type $I$ and $II$.]{\includegraphics[width=0.3125\textwidth]{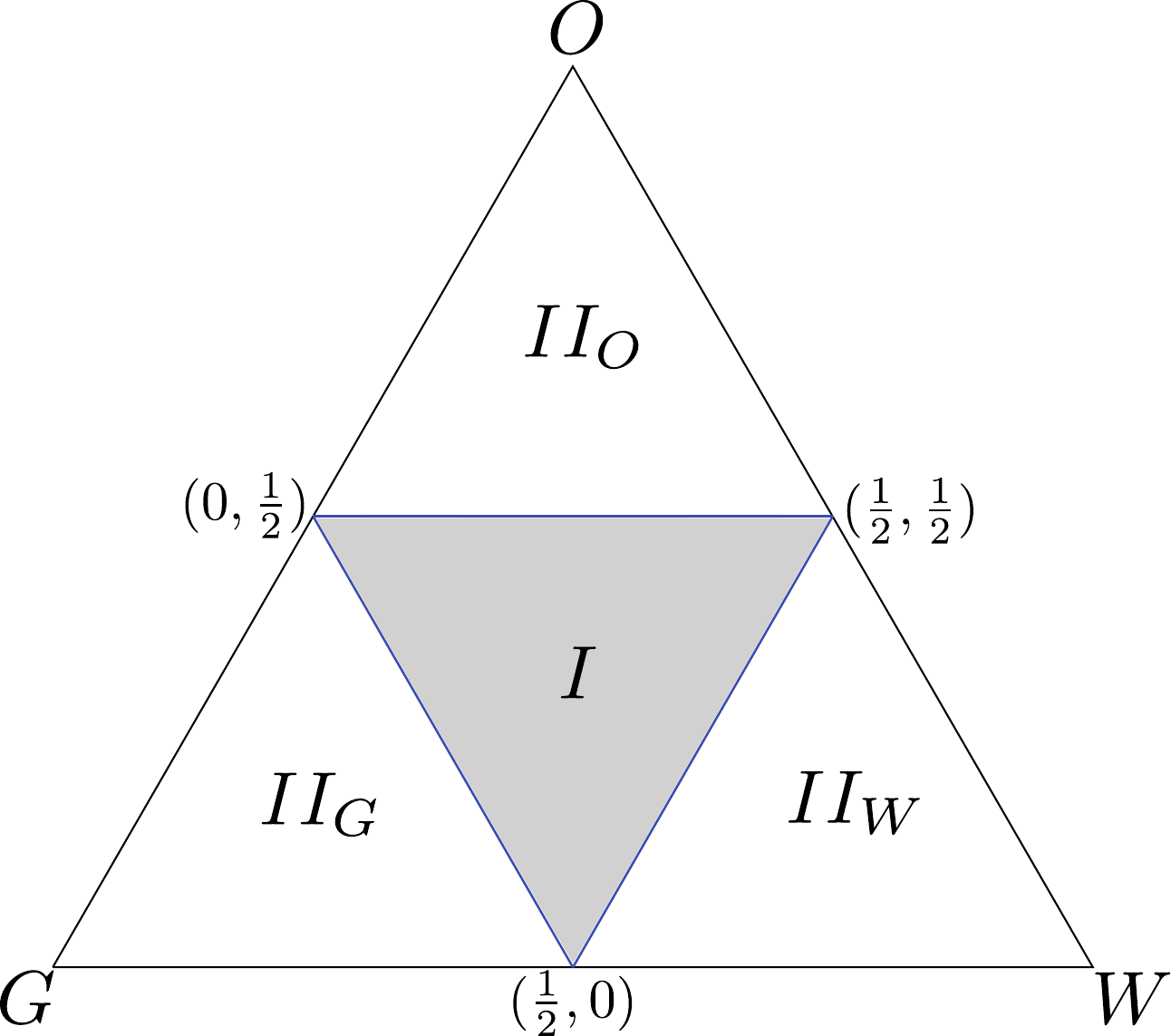}}  \hspace{1.2mm}
	\subfigure[Region for $\um$ makes $\text{Y}^{G}$ relevant for the invariant lines associated with vertex $G$, see Remark \ref{rem:UmY2}. ]{\includegraphics[width=0.3125\textwidth]{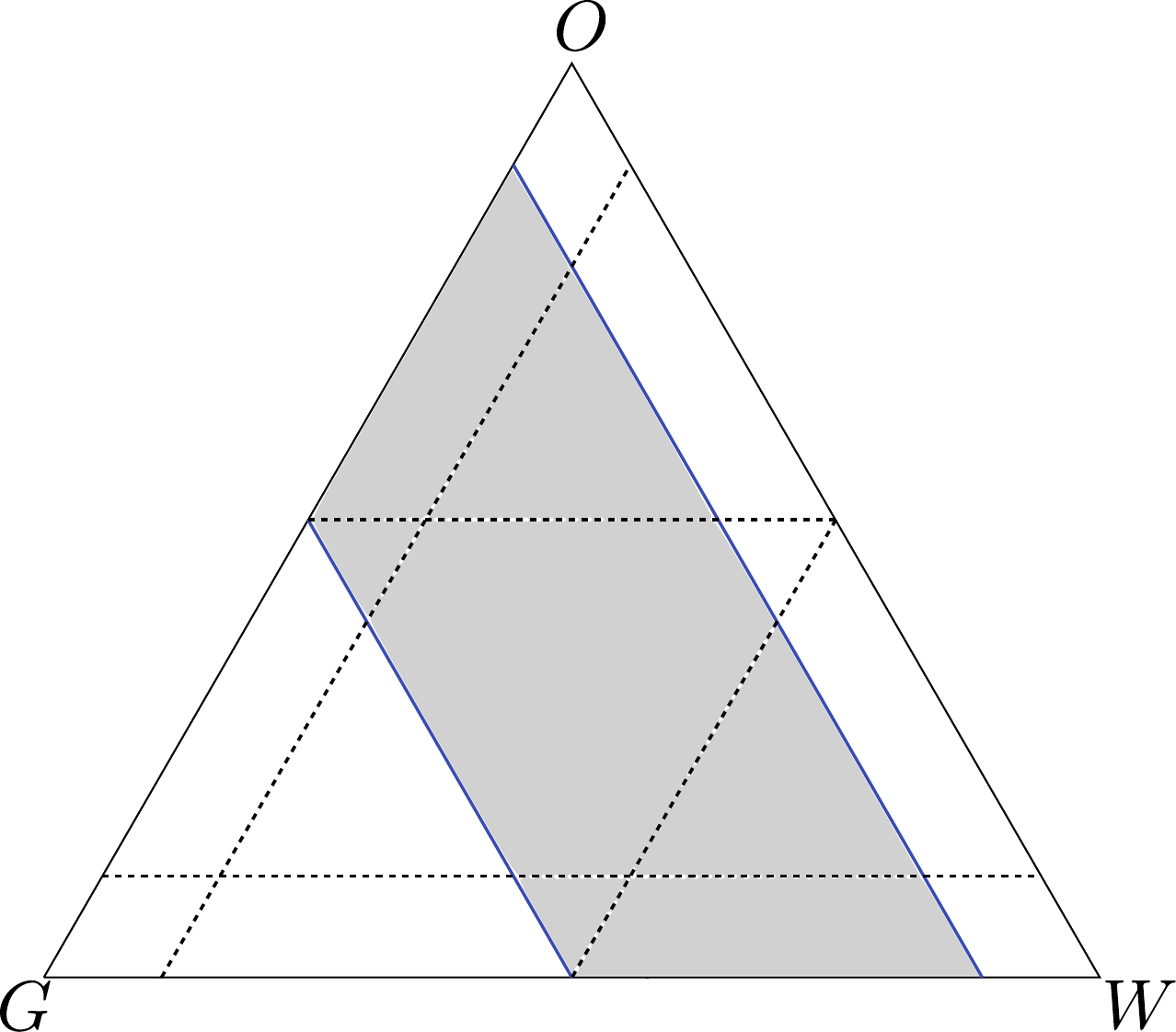}}  
 \hspace{1.2mm}
    \subfigure[All 31 regions for $\um$ determined by different conditions on $\nu_W$, $\nu_O$ and $\nu_G$.]{\includegraphics[width=0.3125\textwidth]{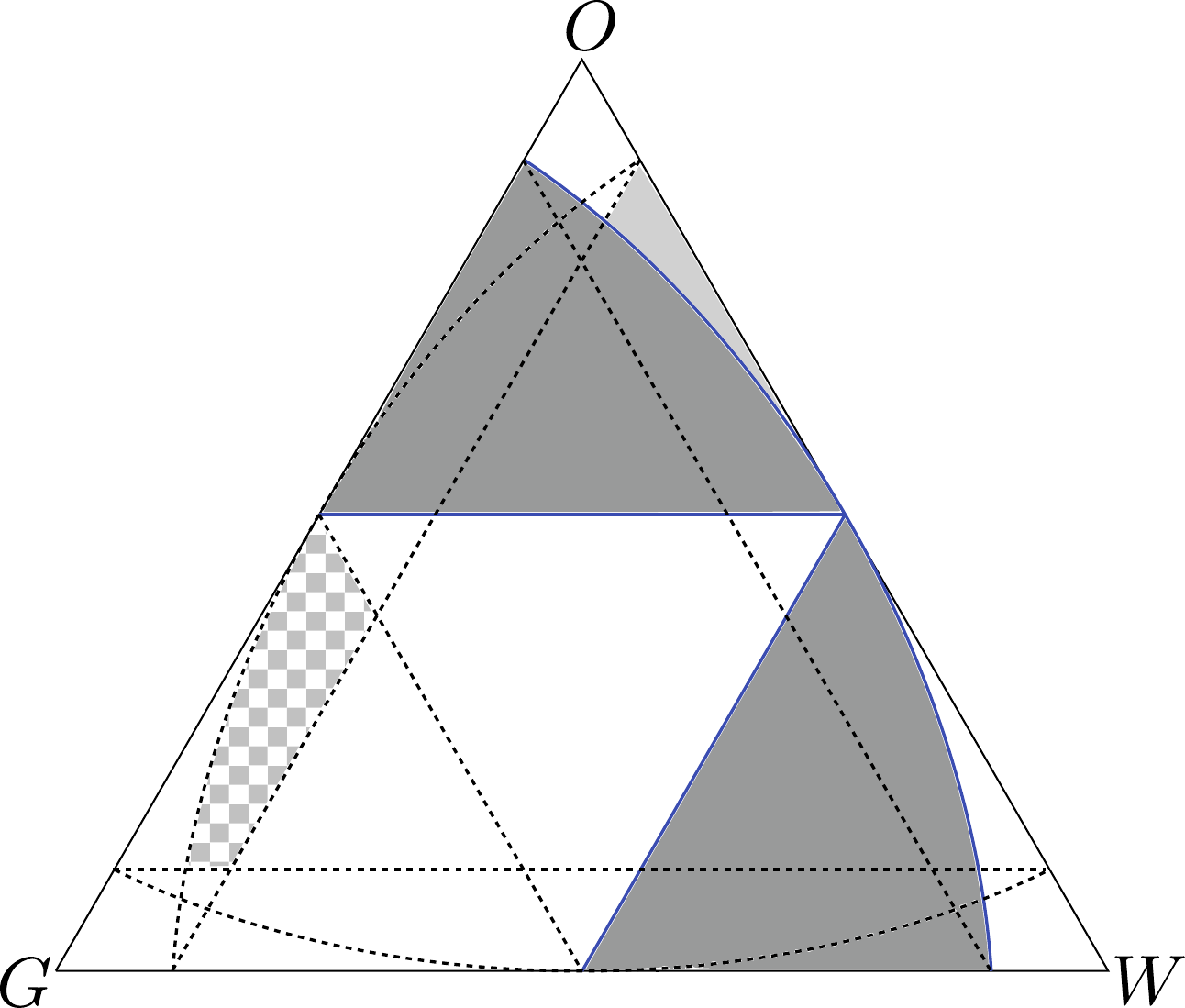}}
	\caption{Position of the umbilic point $\mathcal{U}$. (a) Classification of $\mathcal{U}$ as type $I$ or $II$ given in Corollary \ref{cor:NuUmbilic}. (b) Region associated with vertex $G$ where the location of  $~\mathcal{U}$ imply the existence of the reduced states $s_{\mathcal{Y}^{G}}$ and $s_{\text{Y}^{G}}$. (c) Region for $\um$ in which $\text{X}^{G}$ is relevant for the invariant line associated with vertex $G$ (solid dark gray).  }
	\label{fig:DivisionViscosidadesA2}
\end{figure}
The following result characterizes the viscosity ratios in terms of the location of the umbilic point, and it follows from Theorem \ref{Thm:MatostypeI/II}.
\begin{corollary}\label{cor:NuUmbilic} 
	For the Corey Quad model, if the umbilic point $\mathcal{U}$ is type II, then two of the ratios $\nu_{O}$, $\nu_{G}$ and $\nu_{W}$ are greater than one while the other is less than one. Conversely, if the umbilic point $\mathcal{U}$ is type I, all ratios $\nu_{O}$, $\nu_{G}$ and $\nu_{W}$ defined in \eqref{paramS}(b), are greater than one. To be more precise:
	\begin{enumerate}
		\item $\mathcal{U}\in II_O~~ \Leftrightarrow
		~~\nu_{O}<1$ , $\nu_{G}>1$ and $\nu_{W}>1;$
		\item $\mathcal{U}\in II_W~~ \Leftrightarrow ~~\nu_{W}<1$ , $\nu_{G}>1$ and $\nu_{O}>1;$
		\item $\mathcal{U}\in II_G~~ \Leftrightarrow ~~\nu_{G}<1$ , $\nu_{O}>1$ and $\nu_{W}>1;$
		\item $\mathcal{U}\in \Omega_\mathcal{U}~~ \Leftrightarrow~~\nu_{O}$ , $\nu_{G}$ and ~$\nu_{W}$ are greater than one.
	\end{enumerate}
\end{corollary}
\begin{remark}\label{rem:inflection}
	Notice that the derivative of the effective flux function \eqref{surface7} with $\nu=\nu_G$ along $[G, D]$ is continuous and has a maximum at $s_\mathbb{I}$. This state represents the inflection point of the effective flux function. The meaning of $s_\mathbb{I}$ over any invariant line depends on its position concerning the umbilic point. If $s_\mathbb{I}$ is to the left of $s_{\mathcal{U}}$, it is a fast inflection point. Otherwise, it is a slow inflection point (see equations \eqref{eigenoverbif} and Fig.~\ref{fig:autovsateffec}). Let us study the inflection point $s_\mathbb{I}$. Differentiating the function $f(s)$ twice and setting it equal to zero, we obtain
	\begin{equation}\label{eq:InflectionS}
2{s_\mathbb{I}}^3-3{s_\mathbb{I}}^2+s_{\mathcal{U}}=0.
	\end{equation}
	The discriminant of the cubic equation \eqref{eq:InflectionS} is $ \lozenge=108 s_\um (1-s_\um)$. Then $\lozenge \geq 0 $ if and only if $0\leq s_\um\leq 1.$  Therefore,  when $s_\um = 0$ or $s_\um = 1$ there is only the trivial solution $s_\mathbb{I} = s_\um$ for $s_\mathbb{I}\in[0,1]$. Now, let $P_{s_\um}(s_\mathbb{I})= 2{s_\mathbb{I}}^3-3{s_\mathbb{I}}^2+s_{\mathcal{U}}$ be the function associated with \eqref{eq:InflectionS}. Notice that $P_{s_\um}(0)=s_\um\geq 0$ and $P_{s_\um}(1)=s_\um -1 \leq 0$  for all $s_\um\in[0,1]$. Additionally $P'_{s_\um}(s_\mathbb{I})=6s_\mathbb{I}(s_\mathbb{I}-1)\leq 0$ for all $s_\mathbb{I}\in[0,1] $ and $s_\um\in[0,1]$. Therefore, there is only one solution of \eqref{eq:InflectionS} in $[0,1]$ for all $s_\um\in[0,1]$. From the implicit relation between $s_\um$ and $s_\mathbb{I}$ given by \eqref{eq:InflectionS} we conclude that 
	\begin{eqnarray} \label{eq:InflectionS1}
	\mbox{if  }&\nu_{G}\leq 1 ~~ \Leftrightarrow ~~ s_\mathcal{U}=\frac{\nu_{G}}{\nu_{G} + 1}\leq \frac{1}{2} ~~ \Leftrightarrow ~~ s_\um\leq s_\mathbb{I} \leq \frac{1}{2}; \\  \label{eq:InflectionS2}
	\mbox{if  }& \nu_{G} > 1 ~~ \Leftrightarrow ~~ s_\mathcal{U}=\frac{\nu_{G}}{\nu_{G} + 1}> \frac{1}{2} ~~ \Leftrightarrow ~~\frac{1}{2} < s_\mathbb{I} < s_\um.
	\end{eqnarray}
\end{remark}

The following result, stated in \cite{Andrade2017}, shows the relationship between the viscosity ratio $\nu_G$ and the existence of states $\mathcal{Y}^{G}$ and $\text{Y}^{G}$ along $[G, D]$ of fast double contact.  
\begin{lemma}\label{lemm:DoubleContact}
	Consider the invariant line $[G, D]$. There are two states of fast double contact over $[G, D]$ if and only if $1 <  \nu_{G} \leq 8$.
\end{lemma}
\begin{remark}\label{rem:UmY2}
	Notice that in terms of the position of the umbilic point in the saturation triangle, the two states of fast double contact over $[G,D]$ exist and are relevant if and only if the umbilic point, $\um=(\um_w,\um_o)$ is in the region bounded by the straight lines $\um_w+\um_o=1/2$ and $\um_w+\um_o=8/9$. (See shaded region in Fig.~\ref{fig:DivisionViscosidadesA2} (b)).
\end{remark}
\begin{proof}[Proof of Lemma \ref{lemm:DoubleContact}]
Let $r = \sqrt{2 \nu_{G}/(\nu_{G}  + 1)}$. We rewrite $s_{\text{Y}^{G}}$ given in \eqref{eq:DoubleContactExp} as $s_{\text{Y}^{G}}= -r/2(r -2).$ Otherwise, for $\nu_{G} >0$, we have $r>0$ and
	\begin{equation}\label{ineqDoubleC1}
	    	r\left(r-1\right)^2>0  \Leftrightarrow r >  r^2 \left(2 - r\right).
	\end{equation}
	As $2-r>0$, from \eqref{ineqDoubleC1}, we obtain $ s_{\text{Y}^{G}} > s_\mathcal{U}$. Notice that  when $\nu_{G}=1$ we have $s_{\mathcal{Y}^{G}}=s_{\text{Y}^{G}}=s_\mathcal{U}=1/2$ by \eqref{eq:DoubleContactExp}. Otherwise, $2\nu_{G}/(\nu_{G}+1)> 1 \Leftrightarrow \nu_{G}>1$. Thus 
	\begin{eqnarray}
	\nu_{G}>1
	\Leftrightarrow \nu_{G}/(\nu_{G}+1) ~>~\frac{1}{2}\sqrt{\frac{2\nu_{G}}{\nu_{G}+1}}.
	\end{eqnarray}
	Therefore, we have that $s_\mathcal{U}>s_{\mathcal{Y}^{G}}$ if $\nu_{G}>1$. On the other hand, consider the case in which $ \nu_{G}< 1$. It implies that $2\nu_{G}/(\nu_{G}+1)<1.$ Thus $\nu_{G}/(\nu_{G}+1) >\frac{1}{2}\sqrt{2\nu_{G}/(\nu_{G}+1)}.$ Therefore $s_\mathcal{U} < s_{\mathcal{Y}^{G}}$. This contradicts our assumption given in Remark \ref{rem:DoubleContact}  that the  umbilic point lies between $ s_{\mathcal{Y}^{G}}$ and $s_{\text{Y}^{G}}$. Thus, we conclude that when $\nu_{G} < 1$, the expressions given in \eqref{eq:extenD} for $s_{\mathcal{Y}^{G}}$ and $s_{\text{Y}^{G}}$ do not represent a pair of states of fast double contact along the invariant line $[G, D]$.
 
	Let us prove that over any invariant line, there cannot be states of fast double contact on the same side of the umbilic point. Assume that there are two states $s_M, s_N $ in $(s_\mathcal{U},1]$, $s_M\neq s_N$, such that $\sigma(s_M;s_N)=\lambda_f(s_M)=\lambda_f(s_N).$ 
	Substituting  the expressions \eqref{eqDoubleC} in \eqref{identity0}, we obtain 
	\begin{multline}\label{eqDoubleC5}
	\displaystyle \frac{2s_{M}({s_{N}}^2+\nu_{G}(1-s_{N})^2)}{\nu_{G}}=s_{M}+s_{N}-2s_{M} s_{N} \\=\displaystyle \frac{2s_{N}\left({s_{M}}^2+\nu_{G}(1-s_{M})^2\right)}{\nu_{G}}.
	\end{multline}
	Equating the left and right expressions of \eqref{eqDoubleC5} we have
	\begin{equation}\label{eqDoubleC6}
	(s_M-s_N)(s_M s_N (\nu_G+1)-\nu_G)=0  ~~ \Leftrightarrow ~~  s_{M}=\displaystyle \frac{\nu_{G}}{s_{N}(\nu_{G}+1)}.
	\end{equation}
	Substituting \eqref{eqDoubleC6}(b) in the equality between the left and center expressions of \eqref{eqDoubleC5} we obtain
	\begin{equation}\label{eqDoubleC7}
	s_N^2\nu_{G}+s_N^2-2s_N\nu_{G}+\nu_{G}=0.
	\end{equation}
	The discriminant of \eqref{eqDoubleC7} is $4\nu_{G}^2-4(1+\nu_{G})\nu_{G}=-4\nu_{G}< 0.$ Therefore we have a contradiction. When considering the case  $s_M , s_N $ in $(0,s_\mathcal{U})$ we obtain another contradiction. Thus we conclude that $\nu_{G}>0$ implies $s_{\mathcal{U}}<s_{\text{Y}^{G}}$. 
	
	Let us assume that $\text{Y}^{G}$ coincides with $D$ ($E \mbox{ or } B $ depending on which invariant line). Then $s_{\text{Y}^{G}} =s_D = 1 $. Substituting in expression \eqref{eq:DoubleContactExp} we obtain
	\begin{equation}
	-\frac{1}{2}\frac{\sqrt{2 \nu_{G}(\nu_{G}  + 1)}}{\sqrt{2 \nu_{G}(\nu_{G}  + 1)} -2(\nu_{G}+1)} = 1 ~~ \Leftrightarrow ~~ -7\nu_{G}^2+48 \nu_{G}+64=0.
	\end{equation}
	As $\nu_{G}\geq 0$, then  $\nu_{G}=8.$
	Now assume that the umbilic point lies inside the shaded region of Fig.~\ref{fig:DivisionViscosidadesA2}(b). This means that
	\begin{equation} \label{Eq:umbilicthm}
	1/2<~\mathcal{U}_w+\mathcal{U}_o\leq8/9 ~~ \Leftrightarrow 	1/2<\frac{\nu_{G}}{\nu_{G}+1}\leq8/9.
	\end{equation}
	Then, the last expression of \eqref{Eq:umbilicthm} implies $1< \nu_{G} \leq 8$. 
\end{proof}

\begin{remark}\label{rem:compsISy2}
	Assume that $\nu_{G}>1.$ Then from \eqref{eq:InflectionS2} we have $\frac{1}{2} < s_\mathbb{I} < s_\um$. Following  Remark \ref{rem:inflection}, notice that by \eqref{eq:InflectionS} and \eqref{eq:DoubleContactExp}(a) $P_{s_\um}(s_{\mathcal{Y}^{G}}) = 2 s_{\mathcal{Y}^{G}}^3-3 s_{\mathcal{Y}^{G}}^2 + s_\um = s_\um/2\left(\sqrt{2 ~s_\um} -1\right) >0$. Therefore we have
	\begin{equation}\label{eq:compsISy2}
	\nu_{G} > 1 ~~ \Leftrightarrow ~~ s_\mathcal{U}=\frac{\nu_{G}}{\nu_{G} + 1}>1/2~~ \Leftrightarrow ~~s_{\mathcal{Y}^{G}}< s_\mathbb{I} < s_\um.
	\end{equation}  
\end{remark}

Let us discuss the mixed double contact points below. We are interested in the state $\text{X}^{G}$, the intersection between the mixed double contact locus and the invariant line $[G,D]$. To see the behavior of the mixed contact locus when the $\muw,$ $\muo$, and $\mug$ vary, see Section 4.4 in \cite{Lozano2018}.

Unlike fast double contact states, $\mathcal{Y}^{G}$ and $\text{Y}^{G}$, the state $\mathcal{X}^{G}$ and its correspondent state $\text{X}^{G}$ are not in the same invariant line $[G,D]$ (see Fig.~\ref{fig:limitesBoundaries}). Moreover, $\mathcal{X}^{G}$ is not part of any invariant line. Because of this reason, we cannot characterize them using the reduced saturation parameter (additionally, the shock between $\mathcal{X}^{G}$ and $\text{X}^{G}$ is not admissible). However, the location of $\text{X}^{G}$, like that of $\text{Y}^{G}$, has an important role in our analysis as they define regions where certain wave curves bifurcate (see \cite{Lozano2018}).

Numerical solid evidence shows that the mixed double contact locus exists only when the umbilic point lies in one of the regions of type II. Even so, from the knowledge that $\text{X}^{G}$ is in $[G, D]$ (for example, when $\text{X}^{G}$ lies on the boundary of the saturation triangle), we can extract meaningful information about conditions for $\text{X}^{G}$ to lie inside of saturation triangle (See Fig.~\ref{fig:HuigoniotbyD}).
\begin{figure}
	\centering 
	\includegraphics[width=0.42\textwidth]{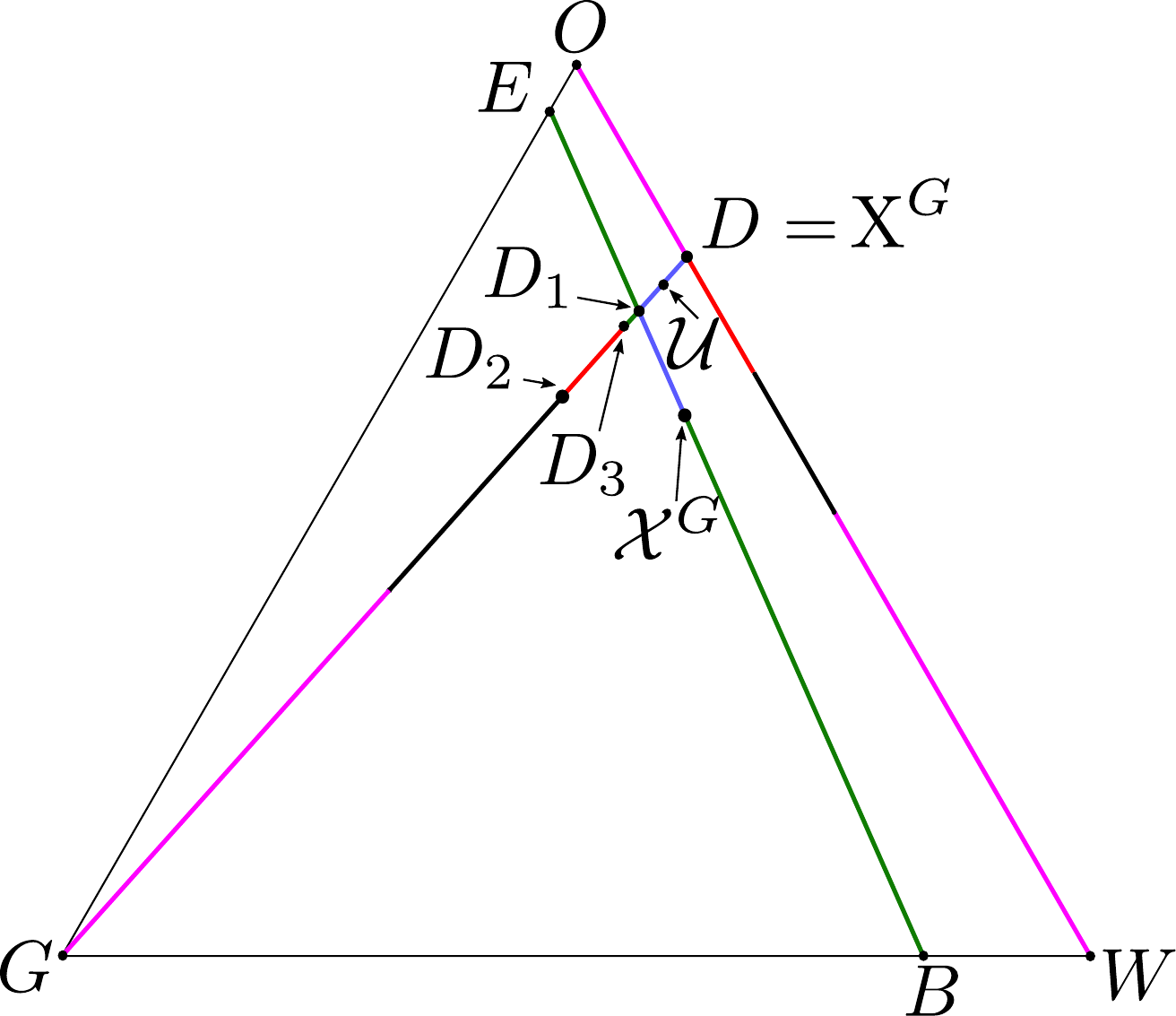}
	\caption{Typical example of a Hugoniot locus from $U^-=D$. In this case $\mu_w,\mu_o$ and $\mu_g$ satisfies $\frac{(\mu_w-\mu_o)^2}{\mu_g(\mu_w+\mu_o)}=8.$ The states $D_2$ and $D_1$ were defined in \eqref{eq:extenD}(a) and (b), respectively; the state $D_3$ satisfy $\sigma(\text{X}^{G};D_3)=\sigma(\text{X}^{G};\mathcal{X}^{G})$, where $\text{X}^{G}$  and $\mathcal{X}^{G}$ are states of mixed double contact.  }
	\label{fig:HuigoniotbyD}
\end{figure}

Now, we present a result that shows the relationship between the viscosity ratio $\nu_G$ and the existence of mixed double contact state $\text{X}^{G}$ along the invariant segment $[G,D]$. This result was first shown in \cite{Andrade2017}. 

\begin{definition}
    Following Remark \ref{re:paramS} we define 
	\begin{equation} 
	\nu^{-}_\Gamma = \mu_{\alpha\beta}^{-}/\mu_{\gamma},~~ \mbox { with }~~ \mu_{\alpha\beta}^{-}=\mu_\alpha - \mu_\beta.
	\end{equation}
\end{definition}
\begin{lemma}\label{lemm:MixedContact}
	Consider the invariant line $[G, D]$ and the umbilic point in  regions $II_O$ or $II_W$. Then the state $\text{X}^{G}$ of mixed double contact over $[G,D]$ lies inside the saturation triangle if and only if $({\nu^-_{G}})^2/\nu_{G} \leq 8$.
\end{lemma}

\begin{remark}
	Notice that in terms of the position of the umbilic point in the saturation triangle, the state $\text{X}^{\Gamma}$ is inside the saturation triangle if and only if the  umbilic point $\um=(\um_w,\um_o)$ is in the region bounded by the straight lines $\um_w+\um_o=1/2$, $\um_o=1/2$ and the ellipse $9\, \um_w^2+14\, \um_w\,\um_o+9\, \um_o^2-8 \,\um_w-8\,\um_o=0$ (see the solid dark gray region in Fig.~\ref{fig:DivisionViscosidadesA2} (c)).
\end{remark}	

\begin{proof}[Proof of Lemma \ref{lemm:MixedContact}] Assume that $\mathcal{U}$ is in region $II_O$ or $II_W$ and  $\text{X}^{G}$ coincides with $D$. Then we have that	
	\begin{equation}\label{eq:mixedcontact0}
	\text{X}^{G} = D=(D_w,D_o)^T = \left(\frac{\mu_w}{\mu_w+\mu_o},\frac{\mu_o}{\mu_w+\mu_o}\right)^T  \mbox{ and } ~~~s_D=1 = s_{\text{X}^{G}}.
	\end{equation}  
	Using \eqref{eigenoverbif}, we have that $2=\lambda_f(s_D)= \lambda_f(s_{\text{X}^{G}})$. Consider the Hugoniot curve for $D$. From Lemma 4.8 of \cite{Azevedo2014}, the $\mathcal{H}(D)$ consisting the three straight lines $[O,W]$, $[G,D]$ and $[E,B]$, see Fig.~\ref{fig:HuigoniotbyD}. Since for all states $M$ over the branch $[O,W]$, the slow characteristic speed is zero, the state $\mathcal{X}^{G}$ corresponding to $\text{X}^{G}$ is not in this branch.
    Let $\mathcal{X}^{G}=\left(\mathcal{X}_{w}^{G}, \mathcal{X}_{o}^{G}\right)^T$ be the corresponding state to $\text{X}^{G}$; from the definition of the mixed double, contact we have
	\begin{equation}\label{eq:mixedcontact1}
	\sigma(\mathcal{X}^{G};\text{X}^{G})=\lambda_s(\mathcal{X}^{G})=\lambda_f(\text{X}^{G}) = 2.
	\end{equation}
	From the Rankine-Hugoniot condition \eqref{Si9}, we have
	\begin{eqnarray}\label{eq:mixedcontact2}
	f_w(\mathcal{X}_{w}^{G}, \mathcal{X}_{o}^{G})-f_w(D_w,D_o)&=2\left( \mathcal{X}_{w}^{G} - D_w\right),\\ \label{eq:mixedcontact3}
	f_o(\mathcal{X}_{w}^{G}, \mathcal{X}_{o}^{G})-f_o(D_w,D_o)& =2\left( \mathcal{X}_{o}^{G} - D_o\right) ~.
	\end{eqnarray}
	From \eqref{eq:flowfunct}, we have
	\begin{equation}\label{eq:mixedcontact4}
	f_w(D_w,D_o) =\mu_w/(\mu_w+\mu_o) ~~\mbox{ and } ~~ f_o(D_w,D_o) =\mu_o/(\mu_w+\mu_o).
	\end{equation}
	Therefore, by substituting \eqref{eq:mixedcontact4} into \eqref{eq:mixedcontact2} and \eqref{eq:mixedcontact3} and dividing the resulting expressions, we obtain
	\begin{equation}\label{eq:mixedcontact7}
	\frac{\left(\mathcal{X}_{w}^{G}\right)^2}{\mu_w}\left(2\mathcal{X}_{o}^{G}-\frac{\mu_o}{\mu_w+\mu_o}\right)- \frac{\left(\mathcal{X}_{o}^{G}\right)^2}{\mu_o}\left(2\mathcal{X}_{w}^{G}-\frac{\mu_w}{\mu_w+\mu_o}\right) = 0.
	\end{equation}
	On the other hand, taking into account Lemma 4.8 of \cite{Azevedo2014}, where the analytic expression of the Rankine-Hugoniot locus is displayed, the nonlocal branch of $\mathcal{H}(D)$ that joins the states $E$ and $B$ can be parametrized as
	\begin{equation}\label{eq:mixedcontact9}
	\mathcal{X}_{w}^{G}= M \mu_w/(\mu_w+\mu_g),\quad \mathcal{X}_{o}^{G}= (1 - M) \mu_o/(\mu_o+\mu_g), ~~~~~~ M\in[0,1].
	\end{equation}
	Therefore, substituting \eqref{eq:mixedcontact9} into \eqref{eq:mixedcontact7} we obtain
	\begin{multline}\label{eq:mixedcontact11}
	    	\mu_w\mu_o\left[(2\mu_g+\mu_o+\mu_w) M-\mu_g-\mu_w\right]\\ \left[2(\mu_w +\mu_o) M^2-(3\mu_w+\mu_o) M+\mu_w+\mu_g\right]=0. 
	\end{multline}

	Then 
	\begin{eqnarray}\label{eq:mixedcontact12}
	(2\mu_g+\mu_o+\mu_w) M-\mu_g-\mu_w=0 ~~~~\mbox{or}\\\label{eq:mixedcontact13}
	2(\mu_w +\mu_o) M^2-(3\mu_w+\mu_o) M+\mu_w+\mu_g=0.
	\end{eqnarray}
	Notice that if \eqref{eq:mixedcontact12} is true, then $\mathcal{X}_{w}^{G}=\mu_w/(2\,\mu_g+\mu_o+\mu_w)$ and $\mathcal{X}_{o}^{G}=\mu_o/(2\,\mu_g+\mu_o+\mu_w)$. Notice that $\mathcal{X}_{w}^{G}/\mu_w = \mathcal{X}_{o}^{G}/\mu_o$, thus the state $\mathcal{X}^{G}$ would be on the invariant line $[G,D]$. Then $s_{\mathcal{X}^{G}} = (\mu_w+\mu_o) /(2\,\mu_g+\mu_o+\mu_w) = \nu_{G}/(2+\nu_{G})$. But $\sigma(s_{\mathcal{X}^{G}};s_{\text{X}^{G}}) = 2(2+\nu_{G})/(\nu_{G}+4) \neq 2 , ~~ \forall ~~\nu_{G}\in (0,1).$ This contradicts \eqref{eq:mixedcontact1}. Therefore \eqref{eq:mixedcontact13} holds. 
 
	Let us analyze the discriminant of \eqref{eq:mixedcontact13}. Notice that
	\begin{equation}
	\triangle = (\mu_w-\mu_o)^2 -8\mu_g(\mu_w+\mu_o).
	\end{equation}
	Given that we seek a unique solution of \eqref{eq:mixedcontact13}, we need that $\triangle = 0$. Therefore  $\text{X}^{G} = D$ if and only if 
	\begin{equation}\label{eq:mixedcontact15}
	(\mu_w-\mu_o)^2 /\mu_g(\mu_w+\mu_o)=({\nu_{G}^-})^2/\nu_{G}=8. 
	\end{equation}
\end{proof}
Figure \ref{fig:DivisionViscosidadesA2}(c) shows two regions for $\um$ where choosing the distinguish invariant line implies that it satisfies different properties. For example, in the plaid light gray region, we have that $~\um$ is type $II_G$, $\nu_W>8,~\nu_G<1,$  $1<\nu_O<8$, ${(\nu_W^-)}^2/\nu_W<8$ and ${(\nu_O^-)}^2/\nu_O<8$. However, in the solid light gray region, $~\um$ is type  $II_O$, $\nu_G>8,~\nu_O<1,$  $1<\nu_W<8$, ${(\nu_G^-)}^2/\nu_G>8$ and ${(\nu_W^-)}^2/\nu_W<8$.

\begin{remark}
	Lemmas \ref{lemm:DoubleContact} and \ref{lemm:MixedContact} still hold if we replace the invariant line $[G, D]$ by $[W, E]$ or $[O, B]$ using \eqref{paramS}. Overall there are 31 different regions where the umbilic point could be; they are shown in Fig.~\ref{fig:DivisionViscosidadesA2}(c). 
\end{remark}
\section{Non classical waves in the reduced two-phase flow model}\label{section:TransMap}
{\it Undercompressive shock} and {\it transitional rarefaction waves} were introduced in \cite{L.1990,L.1992a} to solve Riemann problems for non-strictly hyperbolic systems of conservation laws. Undercompressive shock waves for a generic viscosity matrix $\mathcal{B}$ are characterized by the undercompressive map $T:\mathcal{D}_{\mathcal{T}}\xrightarrow{}\mathcal{D}'_{\mathcal{T}}$  given in \eqref{eq:tranMap}, such that an admissible crossing discontinuity connects $U$ and $U'$ with shock speed $\sigma=\sigma(U;U')$. However, when $\mathcal{B} = I $, the maps $\pi^-$ and $\pi^+$, restricted to $\mathcal{T}$, are not invertible, so the undercompressive shock waves are not characterized by a map.  Instead, ${\mathcal{T}}$ projects into state space as line segments on two sides of $~\um$. Regarding $\left(U^-,\sigma\right)$- and $\left(U^+,\sigma\right)$-coordinates, $\mathcal{T}$ is a planar wedge-shape figure; see Fig.~\ref{fig:SuperficeTraninicio0} and Section \ref{section:TransSurface}.

For $\mathcal{B} = I $, we take advantage of the formulation of the reduced model (Section \ref{section:reduced}) and study, in an explicit way, the part of the undercompressive surface of shocks for which the left and right states lie on that particular invariant line. That is, for each state $ M $ on the segment $ [\um, \mathbb{B}],$ $\mathbb{B}\in\{E,D,B\}$, there are states $ \mathfrak{S}, \mathcal{F} $  along the segment $[\Gamma,\um]$ with $\Gamma\in\{G,W,O\}$ such that any state $ U $ between $ [\mathcal{F}, \mathfrak{S}] $ has an admissible undercompressive shock to $M$. 
This procedure is analogous to choosing a state $ U $ along $[\Gamma, \um] $ and calculating the interval defined by $\mathfrak{S}$ and $\mathcal{F}$ along $[\um,\mathbb{B}]$; see Fig.~\ref{fig:SuperficeTraninicio0} and Section \ref{section:TransSurface}.

The states $\mathfrak{S}$ and $\mathcal{F}$ are characteristic shocks of the suitable families. For this reason, we use expressions \eqref{identity0} and \eqref{eigenoverbif}(a)-(b) for reduced states to find explicitly the end states of the undercompressive interval associated with $ M$.

Notice that when $ M = \mathbb{B} $ we have straightforward end points $\mathfrak{S} =\mathbb{B}_1$ and $\mathcal{F} = \mathbb{B}_2$ for the undercompressive interval. They corresponding effective saturations $s_{\mathbb{B}_2}$ and $s_{\mathbb{B}_1}$ are given in  \eqref{eq:extenD}.

\begin{theorem}\label{thm:x1x2}
	Let $M$ be a state along the secondary bifurcation line associated with the vertex $\Gamma \in \{G,W,O\}$ on the invariant segment $(\um,\mathbb{B}]$, $\mathbb{B}\in\{D,E,B\}$  with $1<\nu_\Gamma \leq 8$. Then, there are two states $\mathfrak{S}$, $\mathcal{F}$  along the segment $[\Gamma,\um)$ such that for any state $U$, with $s_{U}\in(s_{\mathcal{F}},s_{\mathfrak{S}})$, we have an undercompressive shock between $U$ and $M$ with undercompressive speed $\sigma(U;M)$. The state $\mathfrak{S}$ satisfies
	\begin{equation}\label{eq:x1x2}
	\sigma(\mathfrak{S};M)=\lambda_s(\mathfrak{S}).
	\end{equation}
	For the state $\mathcal{F}$ we have two cases depending on $\text{Y}^{\Gamma}$ (Remark \ref{rem:DoubleContact}): 
	\begin{enumerate}
		\item[a)] If $s_{M}\in(s_\um,s_{\text{Y}^{\Gamma}}]$
		\begin{equation}\label{eq:x1x2A}
		\sigma(\mathcal{F};M)=\lambda_f(M).
		\end{equation}
		\item[b)] If $s_{M}\in(s_{\text{Y}^{\Gamma}},1]$
		\begin{equation}\label{eq:x1x2B}
		\sigma(\mathcal{F};M)=\lambda_f(\mathcal{F}).
		\end{equation}
	\end{enumerate}
\end{theorem} 

\begin{proof} We first compute $\mathfrak{S}$. Substituting \eqref{eq:x1x2} in \eqref{eq:charadef}(a) and using \eqref{identity0} we obtain 
	\begin{equation}\label{eq:x1x20}
	s_{\mathfrak{S}} = \frac{\nu_\Gamma s_{M}}{2(1+\nu_\Gamma)s_{M}^2+ \nu_\Gamma(1-2s_{M})}. 
	\end{equation}
	Notice that the denominator of \eqref{eq:x1x20} is nonzero and for $M$ in the edge  opposite to the vertex $\Gamma$ ($s_{M}=1$ ), $s_{\mathfrak{S}} =\nu_\Gamma / (2+\nu_\Gamma)=s_{\mathbb{B}_1}$. Given that $\nu_\Gamma > 1$ then $s_{M}>s_\um >1/2$ then $1-2s_{M}<0$. Thus
	\begin{eqnarray}\label{eq:x1x21}
	s_{M}(1-2s_{M})<s_\um (1-2s_{M})& \Leftrightarrow s_{M}<2s_{M}^2+s_\um (1-2s_{M}),\\ \label{eq:x1x21A}
	& \displaystyle\Leftrightarrow \frac{s_{M}s_\um}{2s_{M}^2+s_\um (1-2s_{M})}<s_\um.
	\end{eqnarray}
	Hence, substituting \eqref{eq:umbilic} in \eqref{eq:x1x21A} we obtain
	\begin{equation}
	    s_{\mathfrak{S}}=\displaystyle \frac{s_{M}\nu_\Gamma}{2(1+\nu_\Gamma)s_{M}^2+\nu_\Gamma (1-2s_{M})}<s_\um.
	\end{equation}
	Therefore, we conclude that for $s_{M}\in(s_\um,1]$,  $s_{\mathbb{B}_1}\leq s_{\mathfrak{S}} <s_{\um}$. Now, from Lemma \ref{lemm:DoubleContact} there are $\mathcal{Y}^{\Gamma}$ and $\text{Y}^{\Gamma}$ along the invariant line defined by vertex $\Gamma$, with $s_{\mathcal{Y}^{\Gamma}}\in(0,s_\um]$ and $s_{\text{Y}^{\Gamma}}\in[s_\um,1]$. Notice that these points naturally define a bifurcation over their corresponding segments because they are the only points along the invariant line that satisfy \eqref{Doubledef}. Next, we study the cases for $\mathcal{F}$: 
	\begin{enumerate}
		\item[a)] Case $s_{M}\in(s_\mathcal{U},s_{\text{Y}^{\Gamma}}]$. We seek a state $\mathcal{F}$ that satisfies \eqref{eq:x1x2A}. Then from \eqref{identity0} and \eqref{eq:charadef}(b) we obtain the quadratic equation
		\begin{equation}\label{eq:x1x22A} 
		A s^2_{\mathcal{F}}+Bs_{\mathcal{F}}+C=0,
		\end{equation}
		where
		\begin{equation}\label{eq:x1x22} 
		A = 2 s_{M} (1+\nu_\Gamma),~~~~~B =-(2s_{M}+1)\nu_\Gamma,~~~~~\mbox{ and }~~~~C= \nu_\Gamma s_{M}.
		\end{equation}
		Notice that the discriminant of quadratic equation \eqref{eq:x1x22A} is 
		\begin{equation}\label{eq:x1x23}
		B^2-4AC  = -\nu_\Gamma\left[(4\nu_\Gamma+8)s_{M}^2-4\nu_\Gamma s_{M}-\nu_\Gamma\right].
		\end{equation}
		Then,
		\begin{multline}\label{eq:x1x24}
		B^2-4AC\geq 0 \displaystyle\Leftrightarrow \frac{\nu_\Gamma -\sqrt{2\nu_\Gamma(\nu_\Gamma+1)}}{2(\nu_\Gamma+2)} < \frac{\nu_\Gamma}{\nu_\Gamma + 1}\\<s_{M}\leq \frac{\nu_\Gamma +\sqrt{2\nu_\Gamma(\nu_\Gamma+1)}}{2(\nu_\Gamma+2)}=s_{\text{Y}^{\Gamma}}. 
		\end{multline} 
		Moreover, when $M = \text{Y}^{\Gamma}$, we have $B^2-4AC = 0$. Then,
		\begin{equation}
		\displaystyle s_{\mathcal{F}} = \frac{(2 s_{\text{Y}^{\Gamma}} +1)\nu_\Gamma}{4 s_{\text{Y}^{\Gamma}}(1+\nu_\Gamma)}=\frac{1}{2}\sqrt{\frac{2\nu_\Gamma }{\nu_\Gamma+1}}=s_{\mathcal{Y}^{\Gamma}}. 
		\end{equation}

        Therefore,  $s_{\mathcal{F}} \in [s_{\mathcal{Y}^{\Gamma}},s_\um)$. Otherwise, since $s_{M}<s_{\text{Y}^{\Gamma}}$, $s_{\mathcal{F}}$ is chosen such that it satisfies the largest root of the quadratic equation  \eqref{eq:x1x22A}, as the other root implies $s_{\mathcal{F}}<s_{\mathcal{Y}^{\Gamma}}$, which contradicts the given assumptions.
        
		\item[b)] Case $s_{M}\in(s_{\text{Y}^{\Gamma}},1].$ We seek a state $\mathcal{F}$ that satisfies \eqref{eq:x1x2B}. From \eqref{identity0} and \eqref{eq:charadef}(b), we obtain the quadratic equation 
		\begin{equation}\label{eq:x1x25A}
		A s^2_{\mathcal{F}}+Bs_{\mathcal{F}}+C=0,
		\end{equation}
		where 
		\begin{equation}\label{eq:x1x25}
		A = \left(2 s_{M}-1\right) \left(\nu_\Gamma +1\right),~~~B = -2s_{M}(\nu_\Gamma+1)~~\mbox{ and }~~
		C= \nu_\Gamma.
		\end{equation}
		Notice that the discriminant of equation \eqref{eq:x1x25A} is
		\begin{equation}\label{eq:discX2}
		B^2-4AC = (\nG+1)\left[(\nG+1)s_{M}^2-2\nu_\Gamma s_{M}+\nu_\Gamma\right],
		\end{equation}
		which is positive for all $s_{M}\in(s_{\text{Y}^{\Gamma}},1]$. Additionally, if we consider $s_{M} = 1,$ we obtain $s_{\mathcal{F}} = (\sqrt{\nu_\Gamma+1}\pm 1)(\sqrt{\nG+1})$. Notice that the largest of these two values for $s_{\mathcal{F}}$ is greater than one, which is invalid. Then for $s_{M} = 1 $, we obtain $s_{\mathcal{F}} = (\sqrt{\nu_\Gamma+1}- 1)/(\sqrt{\nG+1})$ which coincides with $s_{D_2}$, see equation \eqref{eq:extenD}(a). On the other hand, as the pair of states $s_{\text{Y}^{\Gamma}}$ and $s_{\mathcal{Y}^{\Gamma}}$ also satisfies \eqref{eq:x1x2B}, the solution of quadratic equation \eqref{eq:x1x25A} for $s_{M}=s_{\text{Y}^{\Gamma}}$ is $s_{\mathcal{F}}=s_{\mathcal{Y}^{\Gamma}}$, hence $s_{\mathcal{F}} \in [s_{\mathbb{B}_2},s_{\mathcal{Y}^{\Gamma}})$. We conclude that $s_{\mathcal{F}}$ is the lowest solution of the quadratic equation defined by the coefficients \eqref{eq:x1x25}, because if we choose the other one $s_{\mathcal{F}}>1$ which contradicts the given assumptions. 
	\end{enumerate}
\end{proof}

\begin{corollary}\label{cor:x1x2r8}
	For $\nu_\Gamma >8 ~~(\mbox{or}~~ {(\nG^-)}^2/\nG>8)$ in the context of Theorem \ref{thm:x1x2}, the state $\mathfrak{S}$ satisfies \eqref{eq:x1x2}. The case a) holds for $\mathcal{F}$ with $s_{M}\in(s_\um,1]$ and $s_{\mathcal{F}}\in[s_{\mathbb{B}_2^*},s_\um)$, where $\mathbb{B}_2^*$ satisfies $\sigma(\mathbb{B}_2^*;\mathbb{B})=\lambda_f(\mathbb{B})$, for $\mathbb{B}\in\{D,E,B\}$. 
\end{corollary} 

\begin{proof}
	As seen in Lemma \ref{lemm:DoubleContact} (or Lemma \ref{lemm:MixedContact}), when $\nu_\Gamma >8$ (or ${(\nG^-)}^2/\nG>8$) 
	\begin{equation}
	s_{\text{Y}^{\Gamma}}=\frac{\nG+\sqrt{2 \nG(\nG+1)}}{2(\nG+2)}>1.	
	\end{equation}
	Then, state $\text{Y}^{\Gamma}$ is not inside of the saturation triangle and $M$ must be chosen in $(\um,\mathbb{B}]$ ({\it i.e.,} $s_{M}\in(s_{\um},1])$. The state $\mathfrak{S}$, which is the lower extremum of undercompressive interval, is calculated as in \eqref{eq:x1x20} and satisfies $s_{\mathfrak{S}}\in [s_{\mathbb{B}_1},s_\um)$. The upper extremum $\mathcal{F}$ is found as in the proof of Theorem \ref{thm:x1x2} with $s_{M}$ satisfying quadratic equation \eqref{eq:x1x22A} with coefficients \eqref{eq:x1x22}. Notice that for $s_{M}=1$, we obtain from the equation \eqref{eq:x1x22A} 
	\begin{equation}
	s_{\mathcal{F}}=\frac{3\nG\pm\sqrt{\nG(\nG-8)}}{4(\nG+1)}<\frac{\nG}{\nG+1}=s_\um.	
	\end{equation}
	We name $\mathbb{B}_2^*\in[\Gamma,\um]$ the state associated with $\mathbb{B}$, $\mathbb{B}\in\{D,E,B\}$ such that $\sigma(\mathbb{B};\mathbb{B}_2^*)=\lambda_f(\mathbb{B})$ and $s_{\mathbb{B}_2^*}=\frac{3\nG+\sqrt{\nG(\nG-8)}}{4(\nG+1)}$. Now, consider $s_{M}\in(s_\um,1)$. Dividing the coefficients \eqref{eq:x1x22} by $\nG+1$, the solutions of \eqref{eq:x1x22A}  are   
	\begin{equation}
	s_{\mathcal{F}}=\frac{s_\um(2s_{M}+1)\pm\sqrt{s_\um^2 (2 s_{M}+1)^2-8 s_{M}^2 s_\um}}{4 s_{M}},
	\end{equation}
	among which we consider the  largest value of $s_{\mathcal{F}}$ that satisfy $s_{\mathbb{B}_2^*}<s_{\mathcal{F}}<s_\um$. 
\end{proof}

\begin{theorem}\label{thm:x1x2x32}
	For $\nu_\Gamma \leq 1$ in the setting of Theorem \ref{thm:x1x2}, but $M \in [\mathbb{B}_0,1]$, with $\mathbb{B}_0$ defined in \eqref{eq:D0}. Then $\mathfrak{S}$ satisfies \eqref{eq:x1x2} and case b) holds for $\mathcal{F}$.
\end{theorem} 

\begin{proof}
	Notice that the state $\mathbb{B}_0$ and  $\um$ satisfies both equations \eqref{eq:x1x2} and \eqref{eq:x1x2B}. Furthermore $s_\um \leq 1/2 = s_{\mathbb{B}_0}$. We have two options for the choice of $M$: (i) $M\in[\mathbb{B}_0,\mathbb{B}]$   {\it i.e.,} $1/2\leq s_{M}\leq 1$; or (ii) $M\in(\um,\mathbb{B}_0)$, {\it i.e.,} $s_{\um}<s_{M}<1/2$. 
	
	In the first possibility, the equation \eqref{eq:x1x20} for $s_{\mathfrak{S}}$ is computed as in proof of Theorem \ref{thm:x1x2} and satisfies $s_{\mathfrak{S}}\in[s_{\mathbb{B}_1},s_{\um}]$. For the upper extreme $\mathcal{F}$, we have the part  b) of Theorem \ref{thm:x1x2} with $s_{\mathcal{F}}$ satisfying the quadratic equation \eqref{eq:x1x25A} with coefficients \eqref{eq:x1x25} and $s_{M}\in[1/2,1]$. Notice that when $s_{M}=1/2$ we obtain $s_{\mathcal{F}}=s_\um$ while when $s_{M}=1$, $s_{\mathcal{F}}=s_{\mathbb{B}_2}$. Then, we have that for $s_{M}\in(s_{\mathbb{B}_0},1]$ the upper extreme $\mathcal{F}$ satisfy $s_{\mathbb{B}_2}<s_{\mathcal{F}}\leq s_\um$. 
	
	In the second possibility, the equation \eqref{eq:x1x20} for $s_{\mathfrak{S}}$ is also computed as in the proof of Theorem \ref{thm:x1x2} and satisfies $s_{\mathfrak{S}}\in[s_{\um},1/2]$, {\it i.e.,} in this case both states $M$ and $\mathfrak{S}$ lie in the same side of $\um$, which contradicts our assumption of the undercompressive interval must be lie in $[\Gamma,\um]$. Therefore, there are undercompressive shocks joining states $M\in[\mathbb{B}_0,\mathbb{B}]$ and $U\in[\mathcal{F},\mathfrak{S}]\subset [\Gamma,\um]$.     
\end{proof}


\subsection{Transitional Rarefaction waves}\label{section:TransiRarefaction}
Transitional rarefactions were described in \cite{L.1990} for models of two conservation laws. Such waves arise when an integral curve of the fast family is followed by an integral curve of the slow family (in the direction of increasing characteristic speed). The two characteristic speeds must coincide at the point where these curves join. In our model, these waves arise when the umbilic point is type $II_\Gamma$ with $\Gamma \in \{G,W,O\}$, and when considering solutions of Riemann problems that involve states along the invariant line with viscosity ratio $\nG < 1$. As seen in Remark \ref{rem:inflection}, when $\nG<1$, the parameter $s_\mathbb{I}$ represents the intersection between the slow inflection locus and the invariant line defined by vertex $\Gamma$ and satisfies $s_\um<s_\mathbb{I}<s_{\mathbb{B}_0}=1/2$. We follow the same procedure presented for undercompressive shocks in the previous section. We consider a state $M$ along $[s_\um,s_{\mathbb{B}_0}]$ and show how to build transitional rarefaction waves. 
\begin{theorem}\label{thm:transRaref}
	Consider $\nu_\Gamma < 1$ in the setting of Theorem \ref{thm:x1x2}, but for $s_{M} \in [s_\um,s_{\mathbb{B}_0}]$ with $\mathbb{B}_0$ defined in \eqref{eq:D0}. Let $U$ be a state along $[0,s_\um]$, then there is a transitional rarefaction wave between $U$ and $M$ with a transition state between the fast and slow rarefaction at the umbilic point. We have two cases for this transitional rarefaction :
	\begin{enumerate}
		\item[a.] If $s_{M} \in (s_\um ,s_\mathbb{I}]$, then the transitional rarefaction consists of the following sequence of waves
		\begin{equation}\label{sec:transRarefaction}
		U \xrightarrow{R_f} \um  \xrightarrow{R_s} M.
		\end{equation}
		\item[b.] If $s_{M} \in (s_\mathbb{I},s_{\mathbb{B}_0}]$, then there is a state $M_3$, $s_{M_3}\in[s_\um,s_\mathbb{I})$ with $\sigma(M_3;M)=\lambda_s(M_3)$, such that the transitional rarefaction consists of the follows sequence of waves
		\begin{equation}\label{sec:transRarefactionM3}
		U \xrightarrow{R_f} \um  \xrightarrow{R_s} M_3 \xrightarrow{'S_s} M .
		\end{equation} 
	\end{enumerate}
\end{theorem} 
\begin{proof}
	As we see in Remark \ref{rem:charspeed} $\lambda_b(s) = f'(s)$ and satisfy $\lambda_b(s) = \lambda_f(s)$ for $s\in[0,s_\um]$ and $\lambda_b(s) = \lambda_s(s)$ for  $s\in[s_\um,1]$. Moreover 
 	\begin{equation}\label{eq:derivatelambdab}
	\lambda_b'(s)=\dfrac{2s_\um(2 s^3-3s^2+s_\um)}{(\nu_G+1)(s^2-2s_\um s+s_\um)^3}.
	\end{equation}
	From Remark \ref{rem:inflection} $\lambda_b'(s_\mathbb{I}) = 0$ and $~s_\um<s_\mathbb{I}<1/2$. By \eqref{eq:derivatelambdab} we see that $\lambda_b(s)$ is an increasing function in $(0,s_\mathbb{I})$. Then, we can put in $~s_\um$ a slow rarefaction preceded by a fast rarefaction satisfying the speed compatibility criteria \eqref{eq:Cond_Compat}. Therefore, if $s_{M} \in (s_\um,s_\mathbb{I}]$ we have the sequence given in \eqref{sec:transRarefaction} for any state $s_{U} \in [0,s_\um)$. Now, consider $s_{M} \in (s_\mathbb{I},1/2]$ and we seek a state $M_3$, $s_{M_3}\in[s_\um,s_\mathbb{I})$ such that satisfy $\sigma(M_3;M)=\lambda_s(M_3)$ . Therefore from \eqref{identity0} and \eqref{eq:charadef}(a) we obtain the quadratic equation
	\begin{equation}\label{eq:TransRareM3A}
	A s^2_{M_3}+Bs_{M_3}+C=0, 
	\end{equation}
	where 
	\begin{equation}\label{eq:TransRareM3}
	A = (2s_{M}-1)(\nG+1), ~~~B =-2s_{M} (\nG+1)~~\mbox{ and }~~	C= \nG.
	\end{equation}
	Notice that the discriminant of \eqref{eq:TransRareM3A} is positive for all $s_{M}$ with $s_{M}<1/2$. Thus, the solutions of \eqref{eq:TransRareM3A} have opposite signs. If $ M$ is equal to $\mathbb{B}_0 $, $A$ is null in \eqref{eq:TransRareM3} then we have a linear equation $-(\nG +1)s_{M_3} + \nG = 0.$ Hence $s_{M_3} = s_\um$. Moreover, notice that
	\begin{eqnarray*}
	s_\um (s_\um-1)(2s_{M}-1)^2\leq 0 \!\!\!\!\!\!& \Leftrightarrow (-2s_{M}s_\um+s_{M}+s_\um)^2 \leq s_{M}^2-2s_{M}s_\um+s_\um\\
	& \Leftrightarrow -2s_{M}s_\um+s_{M}+s_\um \leq \sqrt{s_{M}^2-2s_{M}s_\um+s_\um} \\
	& ~~\Leftrightarrow s_\um \leq \displaystyle\frac{s_{M} - \sqrt{s_{M}^2-2s_{M}s_\um+s_\um}}{2s_{M}-1} = s_{M_3}.~~~~~~~~~
	\end{eqnarray*}   
 
	We next prove that $s_{M_3}<s_\mathbb{I}$. Let $P_{s_\um}(s)=2s^3 - 3s^2 +s_\um$ be the function associated with \eqref{eq:InflectionS}. From Remark \ref{rem:inflection} $P_{s_\um}(s_\mathbb{I})=0$, $P_{s_\um}(s)>0$ for $s\in[0,s_\mathbb{I})$ and  $P_{s_\um}(s)<0$ for $s\in(s_\mathbb{I},1]$ for all $0\leq  s_\um \leq1$. On the other hand, dividing the coefficients \eqref{eq:TransRareM3} by $(\nG+1)$, we obtain 
	\begin{equation}\label{eq:TransRarefaUltimo}
	(2s_{M}-1){s_{M_3}}^2-2 s_{M} s_{M_3} +s_{\um}=0.
	\end{equation}
	From \eqref{eq:TransRarefaUltimo} we have $s_\um = 2s_{M} s_{M_3} -(2s_{M}-1)s^2_{M_3}$. Notice that
	\begin{eqnarray}
	P_{s_\um}(s_{M_3})=2s^3_{M_3} - 3s^2_{M_3} +s_\um = &2s_{M_3}(1-s_{M_3})(s_{M}-s_{M_3})> 0. 
	\end{eqnarray}
	Hence, $s_\um \leq s_{M_3} < s_\mathbb{I}$.  Therefore, if $s_{M} \in (s_\mathbb{I},s_{\mathbb{B}_0}]$ there is a state $M_3$, with $s_{M_3} \in [s_\um,s_\mathbb{I})$, such that the transitional rarefaction is given by the sequence  \eqref{sec:transRarefactionM3} for any state $s_{U} \in [0,s_\um)$. 
	
\end{proof}
\section{The surface of undercompressive shocks, \texorpdfstring{$\bmm(U)=I$}{B U =  I}  } \label{section:TransSurface}
In Section \ref{section:TransMap}, the undercompressive shocks appear along invariant lines (secondary bifurcation loci, see Definition \ref{def:seconBif}) into the saturation triangle. This representation is simply a projection of a more complex structure called the surface of undercompressive shocks. Therefore, presenting undercompressive shocks constructed in a 3-dimensional state space is convenient. Besides two coordinates in the saturation triangle, we add the shock speed (or characteristic speed, depending on the wave group) as a third coordinate. In this setting, we define the surface of undercompressive shocks $\mathcal{T}$ when $\bmm(U)$ is the identity matrix and study its construction.
\subsection{Undercompressive boundaries}
In this section, we use $(U^-,\sigma)-$ or $(U^+,\sigma)-$ coordinates for the visualization of the undercompressive surface in 3-dimensional space,  $\Omega\times\mathbb{R}^+$. We write a generic state as $P=(s_w,s_o,v)$, where $(s_w,s_o)$ is a state in the saturation triangle $\Omega$ and $v\in\mathbb{R}^+$ is a speed; see Fig.~\ref{fig:SuperficeTraninicio0}(a). In this context, a state  $M\in\Omega$ can be written as $M=(M_w,M_o,0)$.
 \begin{figure}[ht]
 	\centering
 	\subfigure[3-dimensional state space $\Omega\times\mathbb{R}^+$ and orthogonal plane associated with invariant line ${[G,D]}$.]{\includegraphics[scale=0.35]{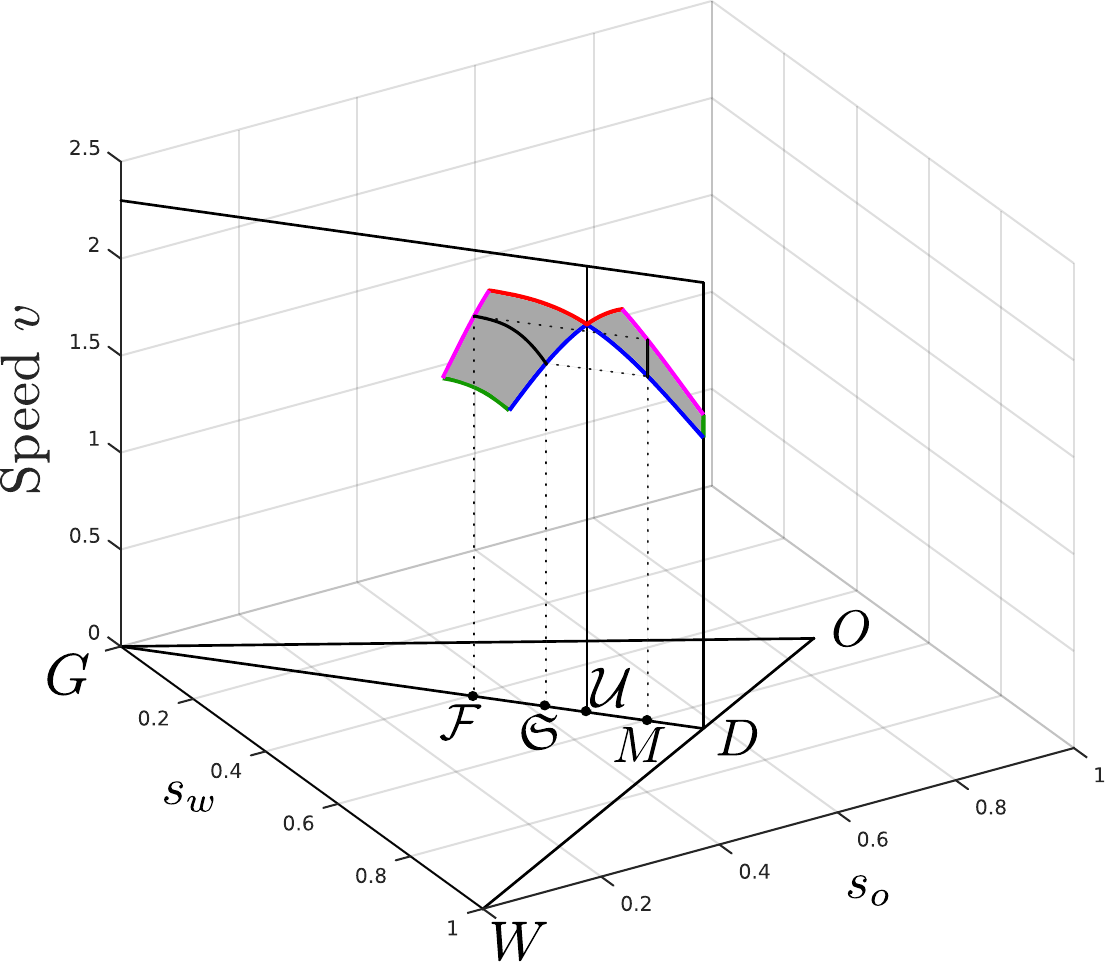}}\hspace{1.5mm}
 	 	\subfigure[Boundaries of the surface of undercompressive shocks for the case $1<\nu_G\leq8$.]{\includegraphics[scale=0.3]{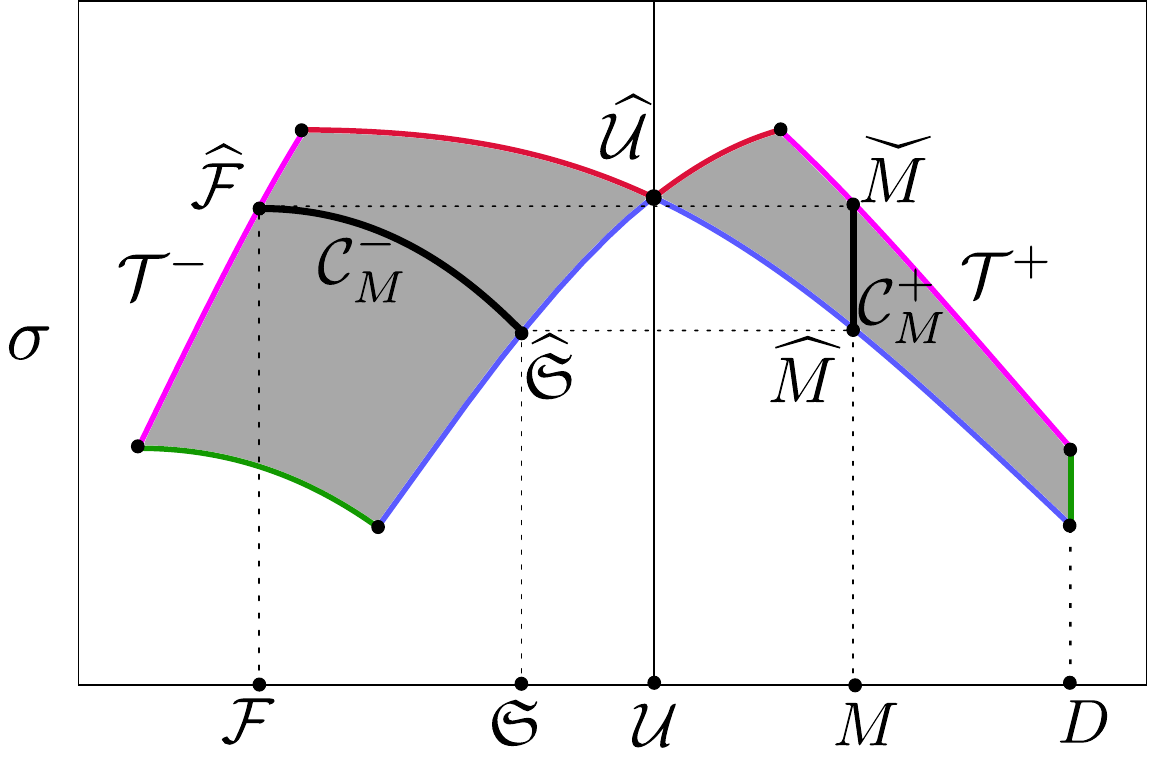}}
 	\caption{The surface of undercompressive shocks for the case $\bmm(U)=I$. (a) Representation of the invariant plane ${[G,D]}\times\mathbb{R}^+ $. (b) Procedure to construct the undercompressive interval for $M\in{(\um,D]}$, and ${[\mathcal{F}(M),\mathfrak{S}(M)]}\subset  {[G,\um)}$. The horizontal axis corresponds to ${[G,D]}$, and the vertical axis is the shock speed.}
 	\label{fig:SuperficeTraninicio0}
 \end{figure}
In Section \ref{section:TransMap}, we presented the construction of the undercompressive interval corresponding to the effective saturation on each invariant line. For a given state $ M $ along the invariant segment $ (\um, D] $, we compute the undercompressive segment $ [\mathcal{F}(M), \mathfrak{S}(M)] $, which includes all states $ U $ that form an undercompressive shock with $ M $, satisfying $ \sigma_{\mathfrak{S}} \leq \sigma_U \leq \sigma_{\mathcal{F}} $, where $ \sigma_{\mathfrak{S}} = \sigma(\mathfrak{S}(M); M) $, $ \sigma_U = \sigma(U; M) $, and $ \sigma_{\mathcal{F}} = \sigma(\mathcal{F}(M); M) $ ($ \mathfrak{S} $ and $ \mathcal{F} $ as defined in Theorem \ref{thm:x1x2}, Theorem \ref{thm:x1x2x32}, or Corollary \ref{cor:x1x2r8}). 

The procedure outlined constructs a pair of curves in $ \Omega \times \mathbb{R}^+ $:
\begin{eqnarray}
         \mathcal{C}_{M}^-=&\{(U_w,U_o,\sigma_U): (U_w,U_o)\in [\mathcal{F}(M),\mathfrak{S}(M)], \, \sigma_{\mathfrak{S}}\leq\sigma_U\leq\sigma_{\mathcal{F}} \},\\
       \mathcal{C}_{M}^+=&\{(U_w,U_o,\sigma_U): U_w=M_w,\, U_o= M_o, \, \sigma_{\mathfrak{S}}\leq\sigma_U\leq\sigma_{\mathcal{F}} \},
\end{eqnarray}
with $ M \in (\um, D]$. Additionally, note that the curve $ \mathcal{C}_{M}^- $ connects the points $ \widehat{\mathfrak{S}} = (\mathfrak{S}_w, \mathfrak{S}_o, \sigma_{\mathfrak{S}}) $ and $ \widehat{\mathcal{F}} = (\mathcal{F}_w, \mathcal{F}_o, \sigma_{\mathcal{F}}) $. Similarly, the curve $ \mathcal{C}_{M}^+ $ connects the points $ \widehat{M} = (M_w, M_o, \sigma_{\mathfrak{S}}) $ and $ \widecheck{M} = (M_w, M_o, \sigma_{\mathcal{F}}) $; see Fig.~\ref{fig:SuperficeTraninicio0}(b). 

We define the surface of undercompressive shocks $\mathcal{T}$, from its representations $\mathcal{T}^-$ in $(U^-,\sigma)$-coordinates and $\mathcal{T}^+$ in $(U^+,\sigma)$-coordinates. Thus, $\mathcal{T}^-$ can be define as the union of curves $\mathcal{C}_{M}^-$, for all $M \in(\um,D]$. Likewise, we can define $\mathcal{T}^+$ from the union of curves $\mathcal{C}_{M}^+$, for all $M \in(\um,D]$; see Fig.~\ref{fig:SuperficeTraninicio0}(b). We proceed analogously for the other invariant segments $[W, E]$ and $[O, B]$.  By varying $M$ and using Theorems \ref{thm:x1x2}, and \ref{thm:x1x2x32}, and Corollary \ref{cor:x1x2r8}, we construct the {\it surface of undercompressive shocks} $\mathcal{T}$.  

Again, the procedure outlined above constructs the boundaries of $\mathcal{T}$, which are associated with the endpoints $ \widehat{\mathfrak{S}},$ $ \widehat{\mathcal{F}},$ $ \widehat{M},$ and $ \widecheck{M} $ of the curves $ \mathcal{C}_{M}^- $ and $ \mathcal{C}_{M}^+$, defined for a state $M$ in an invariant segment $[\mathcal{U}, D]$; see Section 5.
Similar to $\mathcal{T}^-$ and $\mathcal{T}^+$, each boundary can be defined from its representation in $(U^-,\sigma)$-coordinates and $(U^+,\sigma)$-coordinates as follows:  
\begin{definition} \label{Def:TranBoun} 
	Consider the surface of undercompressive shocks associated with the invariant line $[G, D]$. We define
	\begin{enumerate} 
		\item the \textit{slow characteristic boundary} $ (SCB) $ as the set of pairs of states $\widehat{\mathfrak{S}}$ and $\widehat{M}$ in  $\Omega\times\mathbb{R}^+$ such that there is a left-char.$\,s$-shock between $\mathfrak{S}$ and $M$, with $\sigma_{\widehat{\mathfrak{S}}}=\sigma_{\widehat{M}}=\sigma(\mathfrak{S};M)=\lambda_s(\mathfrak{S})$ where $\mathfrak{S}\in[G,~\um)$ and  $M\in(\um,D]$. The states $\widehat{\mathfrak{S}}\in {SCB}^-\subset \mathcal{T}^-$ and $\widehat{M}\in {SCB}^+\subset \mathcal{T}^+$; see blue curves in Figs.~\ref{fig:Transiplane1e8} and \ref{fig:WCNonLocalLA}.
		\item the \textit{fast characteristic boundary} $ (FCB) $ as the set of pairs of states $\widehat{\mathcal{F}}$ and $\widecheck{M}$ in $\Omega\times\mathbb{R}^+$ such that there is a right-char.$\,f$-shock between $\mathcal{F}$ and $M$, with  $\sigma_{\widehat{\mathcal{F}}}=\sigma_{\widecheck{M}}=\sigma(\mathcal{F};M)=\lambda_f(M)$  where $\mathcal{F}\in[G
        ,~\um)$ and $M\in(\um,D]$. The states $\widehat{\mathcal{F}}\in {FCB}^-\subset \mathcal{T}^-$ and $\widecheck{M}\in {FCB}^+\subset \mathcal{T}^+$; see red curves in Figs.~\ref{fig:Transiplane1e8} and \ref{fig:WCNonLocalLA}(b).
		\item the \textit{undercompressive characteristic boundary} $ (UCB) $ as the set of pairs of states $\widehat{\mathcal{F}}$ and $\widecheck{M}$ in $\Omega\times\mathbb{R}^+$ such that there is a left-char.$\,u$-shock between $\mathcal{F}$ and $M$, with $\sigma_{\widehat{\mathcal{F}}}=\sigma_{\widecheck{M}}=\sigma(\mathcal{F};M)=\lambda_f(\mathcal{F})$ where $\mathcal{F}\in[G,~\um)$ and $M\in(\um,D]$. The states $\widehat{\mathcal{F}}\in {UCB}^-\subset \mathcal{T}^-$ and $\widecheck{M}\in {UCB}^+\subset \mathcal{T}^+$; see magenta curves in Figs.~\ref{fig:Transiplane1e8}(a) and \ref{fig:WCNonLocalLA}(a).
		\item the \textit{genuine undercompressive boundary} $ (GUB) $ as the set of pairs of states $\widehat{U}$ and $\widebar{D}$ in $\Omega\times\mathbb{R}^+$ such that there is an $u$-shock between $U$ and $D$, where $\sigma_{\widehat{U}}=\sigma_{\widebar{D}}=\sigma(U;D)$ and $U\in(\mathcal{F}(D),\mathfrak{S}(D))$. The states $\widehat{U}\in {GUB}^-\subset \mathcal{T}^-$ and $\widebar{D}\in {GUB}^+\subset \mathcal{T}^+$; see green curves in Figs.~\ref{fig:Transiplane1e8} and \ref{fig:WCNonLocalLA}.
 	\end{enumerate}
\end{definition}
\begin{remark}
    Definition \ref{Def:TranBoun} was presented considering the invariant line $[G, D]$. However, it can be generalized using the invariant line $[\Gamma, \mathbb{B}]$, with $\Gamma \in \{G, W, O\}$ and $\mathbb{B} \in \{D, E, B\}$.  
\end{remark}

In the following lemmas, we show the different types of undercompressive boundaries, which depend on the values of the viscosity ratio $\nu_{\Gamma}$ defined in \eqref{paramS}. We describe the procedure for the invariant line $[G, D]$ and the viscosity ratio $\nu_{G}$ since the process for the other invariant lines is analogous.
\begin{lemma} \label{lem:transnu4bondaries}
	Consider the invariant line $[G,D]$ with a viscosity ratio satisfying $1 < \nu_{G} \leq 8$. The surface of undercompressive shocks $\mathcal{T}$, associated with the invariant plane corresponding to $[G,D]$, is bounded by the following curves ({\it see Fig.~\ref{fig:Transiplane1e8}(a)}):
	\begin{enumerate}
		\item The curves $SCB^- \subset \mathcal{T}^-$ and $SCB^+ \subset \mathcal{T}^+$, connecting the points $\widehat{D}_1$ and $\widehat{\um}$, and $\widehat{\um}$ and $\widehat{D}$, respectively. 
		\item The curves $FCB^- \subset \mathcal{T}^-$ and $FCB^+ \subset \mathcal{T}^+$, connecting the points $\widehat{\mathcal{Y}}^G$ and $\widehat{\um}$, and $\widehat{\um}$ and $\widehat{Y}^G$, respectively. 
		\item The curves $UCB^- \subset \mathcal{T}^-$ and $UCB^+ \subset \mathcal{T}^+$, connecting the points $\widehat{D}_2$ and $\widehat{\mathcal{Y}}^G$, and $\widecheck{D}$ and $\widehat{Y}^G$, respectively. 
		\item The curves $GUB^- \subset \mathcal{T}^-$ and $GUB^+ \subset \mathcal{T}^+$, connecting the points $\widehat{D}_2$ and $\widehat{D}_1$, and $\widecheck{D}$ and $\widehat{D}$, respectively.
	\end{enumerate}
\end{lemma}

\begin{figure}[ht]
	\centering
	\subfigure[Surface of undercompressive shocks $\mathcal{T}$ for ${[G,D]}$, when $1<\nu_G<8$.]{\includegraphics[scale=0.315]{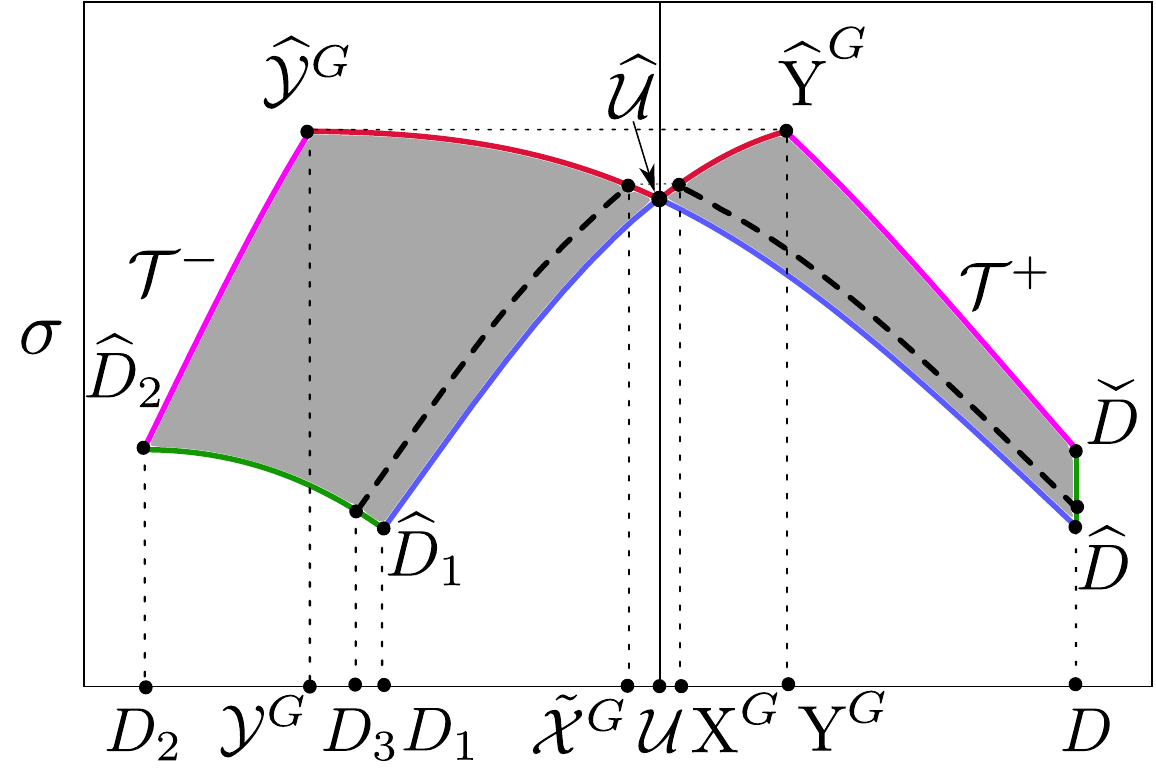}}\hspace{1.65mm}
	\subfigure[Surface of undercompressive shocks $\mathcal{T}$ for ${[G,D]}$, when $\nu_G>8$ and ${(\nu_G^-)}^2/\nu_G\leq8$.]{\includegraphics[scale=0.315]{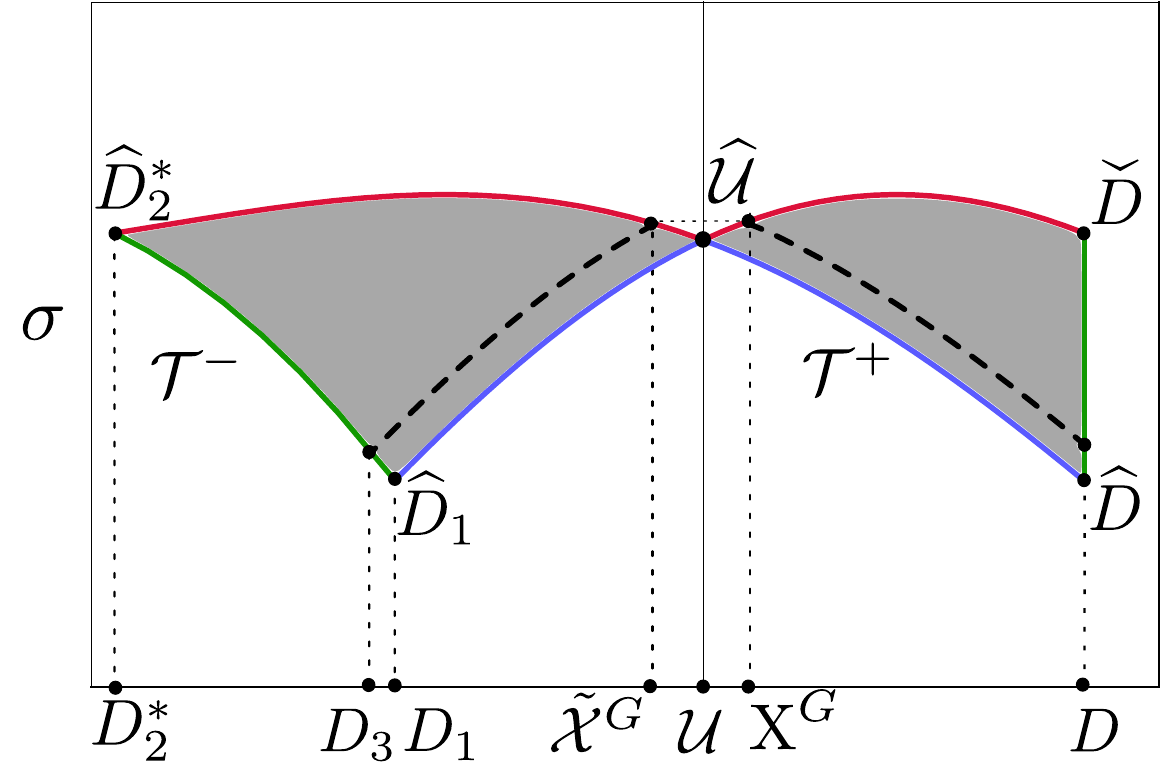}}
	\caption{Undercompressive boundaries in Definition \ref{Def:TranBoun}: The blue curves are $SCB$, red curves are $FCB$, magenta curves are $UCB$, and green curves are $GUB$. The horizontal axis corresponds to $[G, D]$, and the vertical axis is the shock speed; see Remark \ref{rem:twowings}. The dashed curve is associated with the boundary for loss of compatibility of characteristic shocks; see Remark \ref{rem:transitionalBY1} \protect\cite{Lozano2018}.
	}
	\label{fig:Transiplane1e8}
\end{figure}
\begin{proof}
This characterization of the undercompressive boundaries follows directly from Definition \ref{Def:TranBoun} and Theorem \ref{thm:x1x2}.
\end{proof}

\begin{lemma}\label{lem:transnu3bondariesmayor8}
		For $\nu_G>8$, let $D_2^*$ be the right $f$-extension of $D$ on $[G,D]$ such that $\sigma(D_2^*;D)=\lambda_f(D)$, see Corollary \ref{cor:x1x2r8}. Then, the surface of undercompressive shocks $\mathcal{T}$, associated with the invariant plane corresponding to $[G,D]$, comprises the region bounded by the following curves ({\it refers to Fig.~\ref{fig:Transiplane1e8}(b)}):
\begin{enumerate}
			\item The curves $SCB^- \subset \mathcal{T}^-$ and $SCB^+ \subset \mathcal{T}^+$, as in Lemma \ref{lem:transnu4bondaries} (1).
			\item The curves $FCB^- \subset \mathcal{T}^-$ and $FCB^+ \subset \mathcal{T}^+$, connecting the points $\widehat{D}_2^*$ $\widehat{\um}$, and $\widecheck{D}$  and $\widehat{\um}$ respectively.
			\item The curves $GUB^- \subset \mathcal{T}^-$ and $GUB^+ \subset \mathcal{T}^+$ connecting the points $\widehat{D}_2^*$ and $\widehat{D}1$, and $\widehat{D} $ and $\widehat{\um}$ respectively. 
\end{enumerate}
\end{lemma}
\begin{proof}
The justification follows directly from Definition \ref{Def:TranBoun} and from Corollary \ref{cor:x1x2r8}. Notice that there is no boundary of type $UCB$. 
\end{proof}
\begin{lemma} \label{lem:transnu3bondariesmenor1}
	For $0<\nu_G\leq1$, let $D_0$ be the left $f$-extension of $\,\um$ as in the Theorem \ref{thm:x1x2x32}. Then, the surface of undercompressive shocks $\mathcal{T}$, associated with the invariant plane corresponding to $[G,D]$, comprises the region bounded by the following curves ({\it refers to Fig.~\ref{fig:WCNonLocalLA}(a)}):
\begin{enumerate}	
			\item The curves $SCB^- \subset \mathcal{T}^-$ and $SCB^+ \subset \mathcal{T}^+$, connecting the points $\widehat{D}_1$ and $\widehat{\um}$, and $widehat{D}$ and $\widehat{D}_0$, respectively. 
            \item The curves $UCB^- \subset \mathcal{T}^-$ and $UCB^+ \subset \mathcal{T}^+$, connecting the points $\widehat{D}_2$ and $\widehat{\um}$, and $\widecheck{D}$ and $ \widehat{D}_0$, respectively.
			\item The curves $GUB^- \subset \mathcal{T}^-$ and $GUB^+ \subset \mathcal{T}^+$ as in Lemma \ref{lem:transnu4bondaries} (4). 	
	\end{enumerate}
\end{lemma}
\begin{figure}[ht]
	\centering
	\subfigure[Surface of undercompressive shocks $\mathcal{T}$ for ${[G,D]}$, when $0<\nu_G\leq1$.]{\includegraphics[scale=0.3]{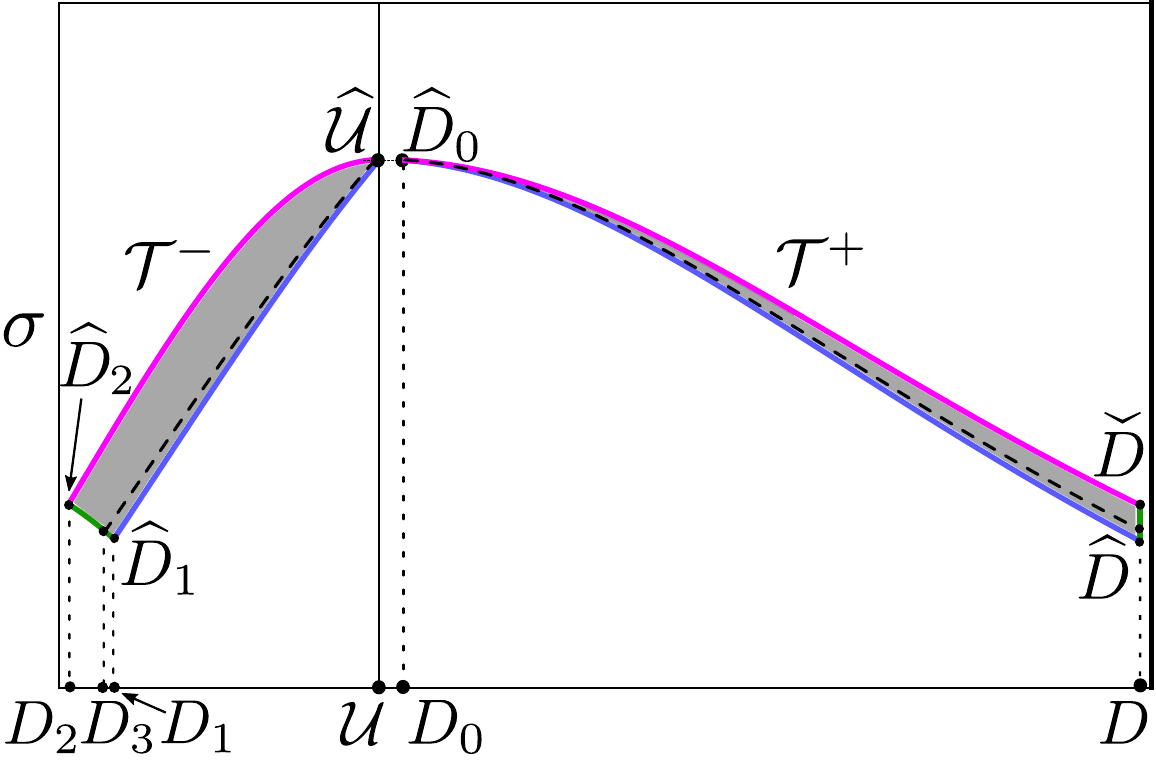}}\hspace{1.65mm}
	\subfigure[Surface of undercompressive shocks $\mathcal{T}$ for ${[G,D]}$, when ${(\nG^-)}^2/\nG>8$.]{\includegraphics[scale=0.3]{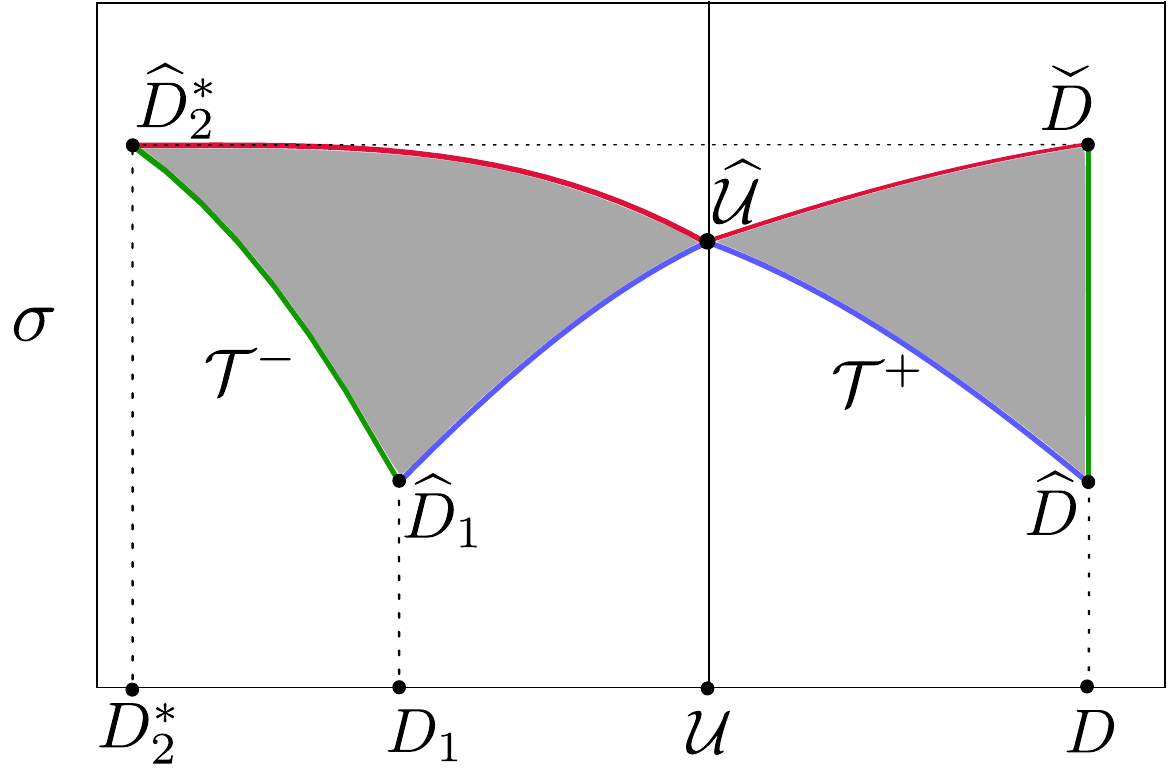}} 
	\caption{Undercompressive boundaries in Definition \ref{Def:TranBoun}: The blue curves are $SCB$, red curves are $FCB$, magenta curves are $UCB$, and green curves are $GUB$. The horizontal axis corresponds to ${[G,D]}$, and the vertical axis is the shock speed, see Remark \ref{rem:twowings}. The dashed curve is associated with the boundary for loss of compatibility of characteristic shocks; see Remark \ref{rem:transitionalBY1} and \protect\cite{Lozano2018}.}
	\label{fig:WCNonLocalLA}
\end{figure}
\begin{proof}
The justification follows directly from Definition \ref{Def:TranBoun} and from Theorem \ref{thm:x1x2x32}. Notice that there is no boundary of type $FCB$.
\end{proof}

\begin{remark}
	Notice in Lemma \ref{lem:transnu3bondariesmenor1} that there is a gap between states $\mathcal{U}$ and $D_0$, as shown in Fig.~\ref{fig:WCNonLocalLA}(a). This means that, in this scenario, the amplitude of undercompressive shocks is bounded away from zero. This fact has been noted before; see \cite{Dan2006}. 
    The solutions encompassing the segment between states $\um$ and $D_0$ use transitional rarefactions, as discussed in Section \ref{section:TransiRarefaction}.
\end{remark}

\begin{claim}
For $0<\nu_G\leq1$, let $D_0$ be the left $f$-extension of $\um$ as in Theorem \ref{thm:x1x2x32}. Then, for any right state $R$ in the segment $(\um,D_0]$, there is no left state $L$ along the invariant line $[G,D)$ such that an undercompressive shock can join $L$ and $R$ (see Fig.~\ref{fig:ComparateSpeedD0}).
\end{claim}

\begin{figure}[ht]
	\centering
	\subfigure[Speed diagram: Primary branch {$[G,D]$}, for $\mathcal{H}(R)$ with $R\in{(\um,D_0]}$, when $0<\nu_G<1$.]{\includegraphics[scale=0.425]{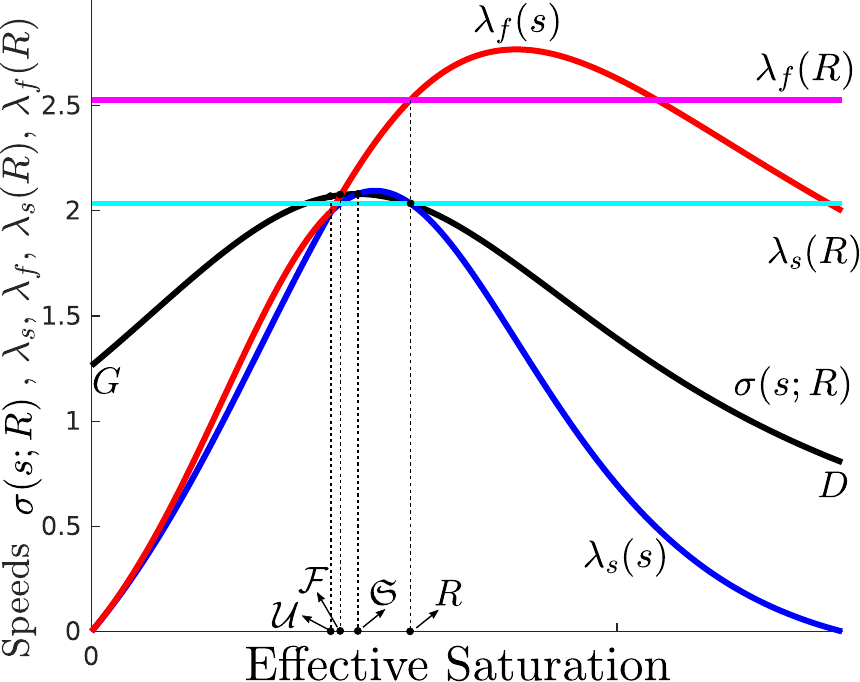}}\hspace{1.8mm}
		\subfigure[Speed diagram: zoomed segment {$[\mathcal{F},\mathfrak{S}]$} for $\mathcal{H}(R)$ with $R\in{(\um,D_0]}$.]{\includegraphics[scale=0.425]{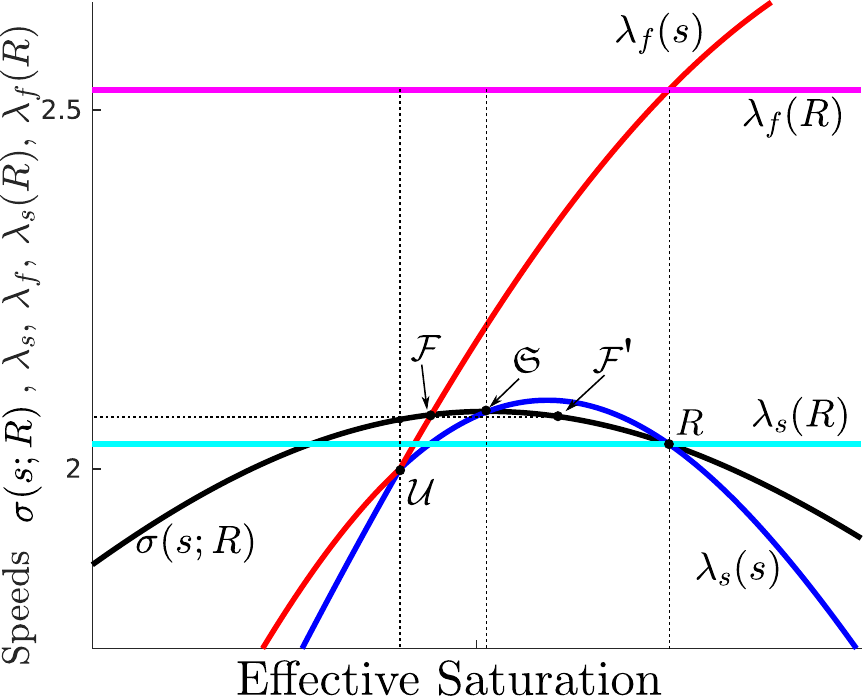}}
	\caption{Speed diagram for $\Hm(R)$, $R\in(\um,D_0]$. The horizontal axis corresponds to a parametrization of effective saturation, and the vertical axis shows speeds. The blue (resp. red) line is the characteristic speed $\lambda_s$ (resp. $\lambda_f$), while the black line is the shock speed $\sigma$. Horizontal cyan and magenta lines correspond to the constant values $\lambda_s(R)$, and $\lambda_f(R)$. States in the same dashed black line satisfy the triple shock rule; see \protect\cite{Azevedo2014}. Notice that $\mathfrak{S}$ is a Bethe-Wendroff state (\protect\cite{V.2015,Wendroff1972}) and for each state $M\in[\mathcal{F},\mathfrak{S}]$, there is a state in $U\in[\mathfrak{S},\mathcal{F}']$ that satisfies $\sigma(M;R)=\sigma(U;R)$. }
	\label{fig:ComparateSpeedD0}
\end{figure}

\begin{proof}
	Consider the Hugoniot curve $\Hm(R)$ for a right state $R$ in the segment $(\um,D_0)\subset[G,D]$. The Hugoniot curve for states along the invariant line is computed explicitly in \cite{L.2016}. Since the only possibility for the existence of an undercompressive shock between states $L$ and $R$ is that both lie in $[G,D]$, we only need to consider the primary branch $[G,D]$ of $\Hm(R)$. In Fig.~\ref{fig:ComparateSpeedD0}, we compare the characteristic speeds $\lambda_s(R)$,$\lambda_f(R)$, $\lambda_s(s)$ and $\lambda_f(s)$ with the shock speed $\sigma(R;s)$, as $s$ varies along the Hugoniot branch $[G,D]$, and identify the Bethe-Wendroff point (\protect\cite{V.2015,Wendroff1972}) $\mathfrak{S}$ and the undercompressive segment $[\mathcal{F},\mathfrak{S}]$. According to Fig.~\ref{fig:ComparateSpeedD0} there is a state $\mathcal{F}'\in[\mathfrak{S},R]$ such that $\sigma(\mathcal{F};R)=\sigma(\mathcal{F}';R)$ and $\mathcal{F}'$ is a repeller. Notice that for any state $M\in[\mathcal{F},\mathfrak{S}]$, there is a state $U\in[\mathfrak{S},\mathcal{F}']$ such that $\sigma(M;R)=\sigma(U;R)$. Since $U$ is between $M$ and $R$, no orbit is connecting $M$ to $R$ along the invariant line $[G,D]$. Therefore, there is no admissible undercompressive shock from $M$ to $R$ when $R\in(\um,D_0]$. 
\end{proof}

\begin{remark}\label{rem:transitionalBY1}
The undercompressive boundary associated with the boundary for loss of compatibility of characteristic shocks (see \cite{Lozano2018}) (shown in Figures \ref{fig:Transiplane1e8}-\ref{fig:WCNonLocalLA} as black dashed lines) is computed by varying $M\in[\text{X}^{G},D]$ and finding the states $U\in[\mathcal{F}(M),\mathfrak{S}(M)]$ with $\sigma(\mathfrak{S}(M);M)\leq\sigma(U;M)\leq \sigma(\mathcal{F}(M);M)$. In Fig.~\ref{fig:WCNonLocalLA}(b), we show the surface of undercompressive shocks for the case $(\nu_G^{-})^2/\nu_G >8$, in which this boundary does not appear because the state $\text{X}^{G}$ ceases to exist within the saturation triangle, see Lemma \ref{lemm:MixedContact}.   
\end{remark}

\section{Capillarity induced undercompressive shock surface}\label{section:GeneralMatrix}
In this section, we discuss how to construct numerically the surface of undercompressive shocks for the case $\bmm(U)\neq I$. We consider the general case where $\bmm(U)$ is defined as in \eqref{system6}-\eqref{PUprima}, which is associated with proper modeling of the physical diffusive effects caused by capillary pressures \cite{Abreu2014, Abreu2006,V.2002}. We verify that the surface of undercompressive shocks for this case has the same topological structure found in the case of $\bmm(U)= I.$ Finally, we perform numerical simulations comparing the solutions for identity and non-identity matrices.
  
\subsection {Saddle-saddle connections}\label{section:saddleconn}
When the viscosity matrix is not the identity, the secondary bifurcation locus is no longer made of invariant lines of the parabolic system \eqref{eq:sistem2}, and the admissible undercompressive shocks occur in regions away from these lines.

Schaeffer and Shearer in \cite{Schaeffer1987} studied the approximation of Corey flux function by a homogeneous quadratic system in a neighborhood of the umbilic point. For this quadratic approximation, \cite{L.1990,azevedo1995multiple} exhibits undercompressive shocks away from the bifurcation lines with two peculiarities: the left and right states lie within cones with vertices at the umbilic point, and the orbits joining these states occur along straight lines. In \cite{Dan2006}, Marchesin and Mailybaev showed that the undercompressive shocks are structurally stable under perturbations of the viscosity matrix.  

Recall that this work adopts the viscous profile criterion for shock admissibility. As a result, the admissibility of Riemann solutions that involve shocks is affected by perturbations of the viscosity matrix. 
Since these shocks are formed by saddles of \eqref{V9} that have a connecting orbit, we utilized the results in \cite{L.1990,azevedo1995multiple} for quadratic systems and look for connections near the umbilic point. 

Given a fixed state $U^-\in\Omega$, it will belong to the surface of undercompressive shocks $\mathcal{T}$ if we are able to find another state $U^+\in\Omega$  such that the conditions (1) and (2) are satisfied 
\begin{enumerate}
	\item[(1)] both $U^-$ and $U^+$ are saddle points for the associated traveling wave solution of the ODE system \eqref{V9}; 
\item[(2)] there is an orbit connecting $U^-$ to $U^+$. 
\end{enumerate}  
The Hugoniot locus for $U^-$ includes a curve arc composed of $U^+$ states, forming a crossing discontinuity with $U^-$. The point $U^-$ and points $U^+$ in the arc satisfy condition (1) (see Definition \ref{def:typeEqui}). This arc can be parameterized locally by the shock speed $\sigma$ so that we can refer to these points as $U^+(\sigma)$. To find $\sigma _0$ such that $U^+(\sigma _0)$ also satisfies (2) above, we vary $\sigma$ and then study the behavior of two orbits such that $U^-$ is the $\alpha$-limit and $U^+$ is the $\omega-$limit of the corresponding dynamical system \eqref{V9}; see \cite{Guckenheimer1986,Volpert1994}. If such $\sigma _0$ exists, these two orbits get closer as $\sigma$ approaches $\sigma_0$, and they coincide when $\sigma = \sigma _0$. 

Figure \ref{fig:sattrian2} shows how these orbits behave in practice.
Notice how the unstable manifold, colored violet, is located at the left side of the yellow stable manifold in Fig.~\ref{fig:sattrian2}(a), where $\sigma<\sigma_0$. As we increase $\sigma$, the distance between the two orbits decreases so that in Fig.~\ref{fig:sattrian2}(b), we see that the violet and yellow orbits coincide. If we keep increasing $\sigma$, the two orbits diverge, as in Fig.~\ref{fig:sattrian2}(c). In this case, the unstable manifold is located at the right of the yellow stable manifold. As a result, we can apply a binary search/bisection algorithm to approximate $\sigma _0$ by using the states in Figs.\ref{fig:sattrian2}(a) and Fig.~\ref{fig:sattrian2}(c). The distance between the violet and yellow orbits is measured by looking at their intersections with a conveniently chosen transversal line segment. When the procedure is successful, {\it i.e.,} when we find $(U^-,U^+,\sigma_0)\in \mathcal{T}$ such that there exists a connection between the saddle points $U^-$ and $U^+$, we take a perturbation state $U^{-}_*$ and repeat the process. We do this until we find the undercompressive regions in the state space; see \ref{Alg:1} for details. 
\begin{figure}[ht]
	\centering
	\subfigure[$\sigma<\sigma_0$.]{\includegraphics[scale=0.15]{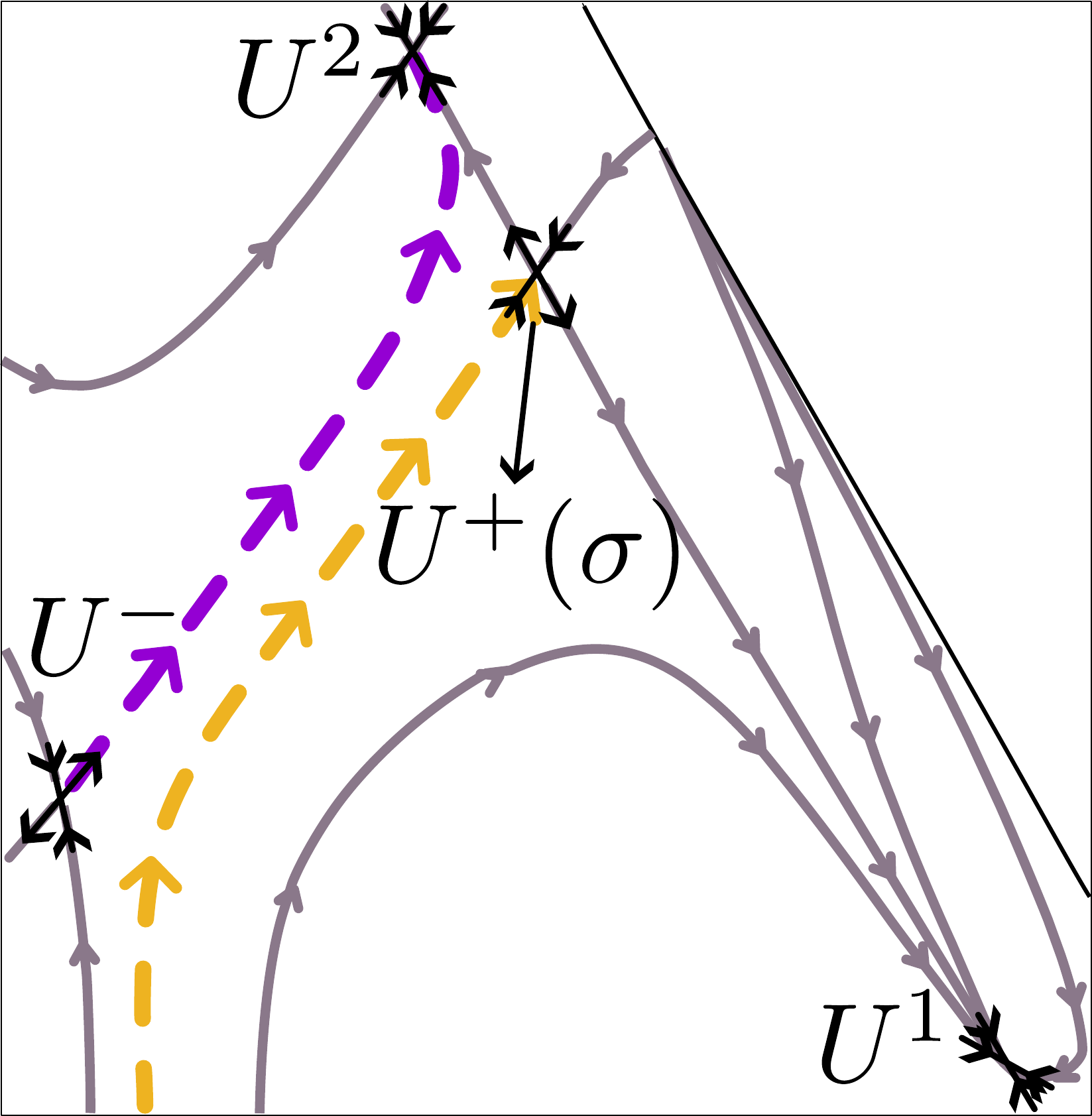}}\hspace{1.2mm}
	\subfigure[ $\sigma=\sigma_0$.]{\includegraphics[scale=0.15]{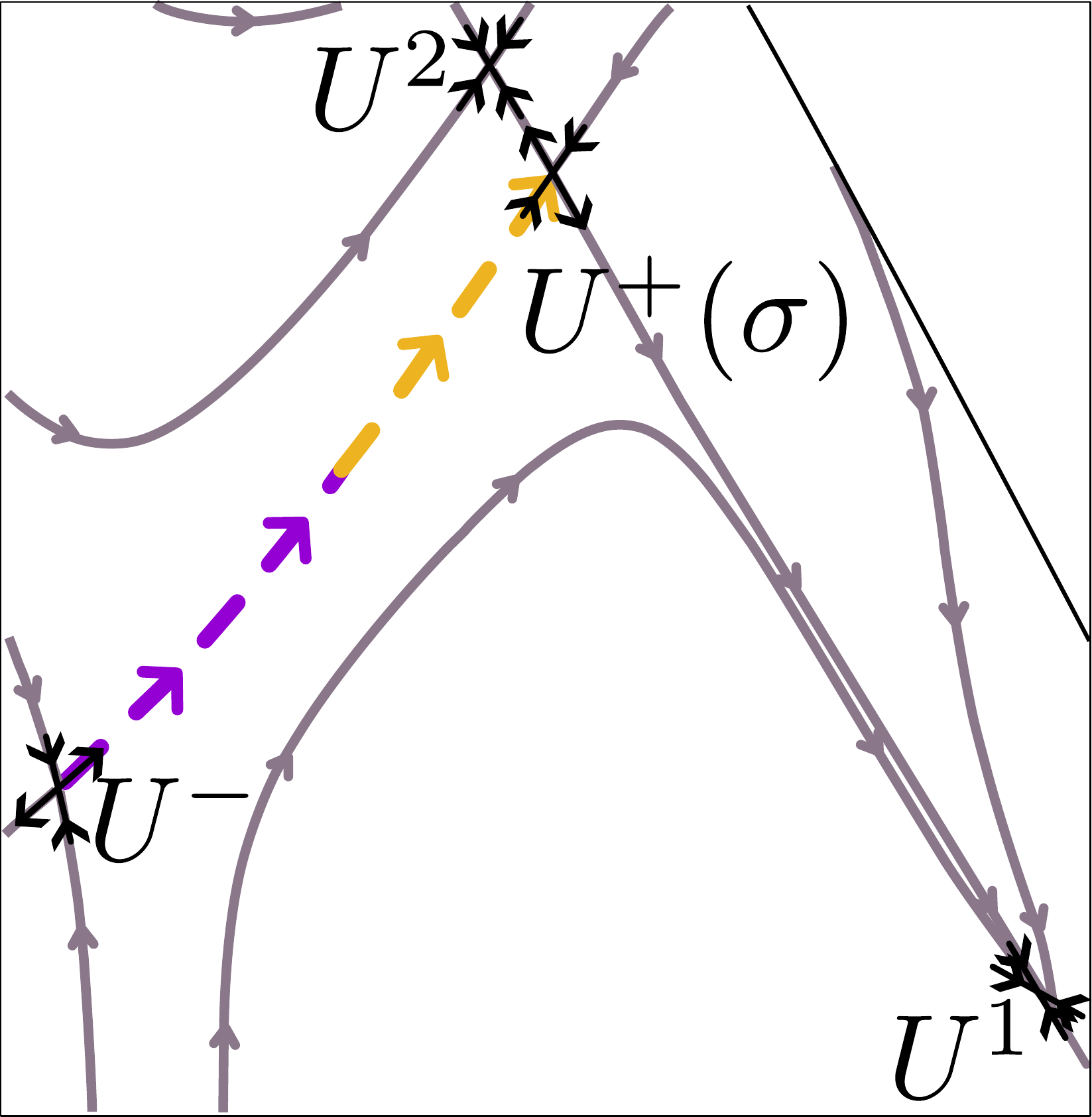}} \hspace{1.2mm}
	\subfigure[ $\sigma>\sigma_0$.]{\includegraphics[scale=0.15]{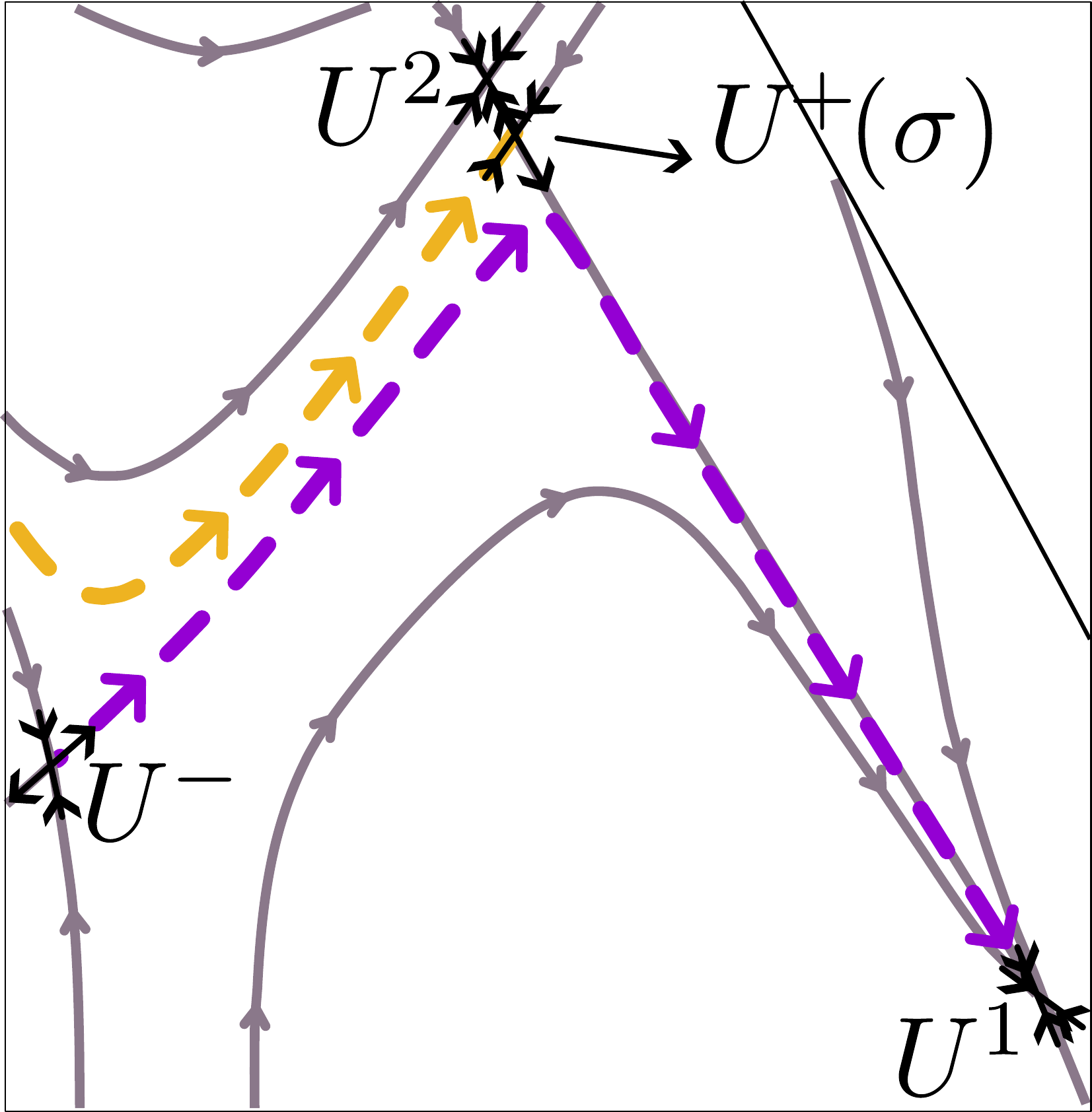}}
	\caption{Explanation to find connections between saddles. $U^-$ is a fixed state close to the umbilic point, and $U^+(\sigma) \in \mathcal{H}(U^-).$ The states $U^1$ and $U^2$ are other equilibria of the dynamical system \eqref{V9}, which also belong to $\mathcal{H}(U^-)$. The violet orbit is the unstable manifold for $U^-$, and the yellow orbit is the stable manifold for $U^+(\sigma)$ for the dynamical system associated with $U^-$ and $\sigma$. The gray orbits are other stable and unstable manifolds of the dynamical system. Here, we consider $\um\in II_O$.}
	\label{fig:sattrian2}
\end{figure}

\subsection{Saddle saddle-node connection.}\label{section:SnodeSconection}

As we described in the previous section, the surface of undercompressive shocks is constructed by finding an initial connection $(U^-,U^+,\sigma_0)$ and then by attempting to find connections for perturbed states $U^{-}_*$ of $U^{-}$. As long as this procedure is successful, we can repeat it until we cover the whole surface. When this procedure fails, there may exist intermediate states $U_b^-\in\Omega$ and $U_b^+\in\Omega$ matching one of the following cases: 
\begin{itemize}
    \item slow characteristic boundary: $U^-_b$ is a saddle-repeller and $U^+_b$ is a saddle (left-char.$\,s$-shock);
    \item fast characteristic boundary: $U^-_b$ is a saddle, and $U^+_b$ is a saddle-attractor (right-char.$\,f$-shock);
    \item undercompressive characteristics boundary: $U^-_b$ is a saddle-attractor and $U^+_b$ is a saddle (left-char.$\,u$-shock);
    \item genuine undercompressive boundary  : $U^-_b$ and $U^+_b$ are saddles, and $U^+_b\in\partial\Omega$;
    \item Other types of boundaries where the ones above meet, such as the umbilic point $\um$; the points $\mathscr{A}_1\in\{\mathscr{D}_1,\mathscr{E}_1,\mathscr{B}_1\}$ and $\mathfrak{A}_1\in\{\mathfrak{D}_1,\mathfrak{E}_1,\mathfrak{B}_1\}$, such that $\mathscr{A}_1\in\partial\Omega$ and $\mathfrak{A}_1$ lies in the $s$-right-extension of $\partial\Omega$; the points $\mathscr{A}_2\in\{\mathscr{D}_2,\mathscr{E}_2,\mathscr{B}_2\}$ and $\mathfrak{A}_2\in\{\mathfrak{D}_2,\mathfrak{E}_2,\mathfrak{B}_2\}$ such that $\mathscr{A}_2\in\partial\Omega$ and $\mathfrak{A}_2$ lies in the $f$-right-extension of $\partial\Omega$; and the points $\mathfrak{Y}^{\Gamma}$ and ${\mathscr{Y}}^{\Gamma}$, ${\Gamma}\in\{G,W,O\}$ that belong to the double fast contact locus. 
\end{itemize}
In what follows, we discuss only the procedure used to calculate saddle-node-to-saddle connections for states that live in slow characteristic boundaries. Algorithms for calculating the other types of boundaries are analogous.

We begin with the triple $(U^-, U^+, \sigma_{0})$ and a state $V^-\in \Omega$, so that $U^-$ and $V^-$ are close. A saddle-to-saddle connection exists between $U^-\in\mathcal{D}_{\mathcal{T}}$ and    $U^+\in\mathcal{D}'_{\mathcal{T}}$ ({\it i.e.}, \textit{inside} the undercompressive region). The state $V^-$ lacks a saddle-to-saddle connection ({\it i.e.}, \textit{outside} the undercompressive region); see Fig.~\ref{fig:sattrian3}(a). For the state $U^-$, we choose $\widetilde{\sigma}=\widetilde{\sigma}(U^-;\widetilde{U^+})$ so that $\widetilde{\sigma} = \lambda_s(U^-)$, $\widetilde{U^+}\in \Hm(U^-)$, $\widetilde{U^+}$ is a saddle, and $U^-$ becomes a saddle-repeller; see Fig.~\ref{fig:sattrian3}(b). Notice that $U^+$ and $\widetilde{U^+}$ belong to $\mathcal{H}(U^-)$.  We do the same for $V^-$ and select $\sigma_V=\sigma_V(V^-;V^+)$ to make $\sigma_V = \lambda_s(V^-)$, $V^+\in \Hm(V^-)$, $V^+$ a saddle, and $V^-$ becomes a saddle-repeller; see Fig.~\ref{fig:sattrian3}(c).  Notice that there may be more than one $\widetilde{U^+}$ and $V^+$ states satisfying these restrictions. However, since we assume $U^-$ and $V^-$ are located near the slow characteristic boundary, we select the ones closest to $U^+$. The described procedure defines $\widetilde{U^+}$ and $V^+$ and can be applied to determine opposite states to every intermediate state in the segment (red line) connecting $U^-$ to $V^-$; see Fig~\ref{fig:sattrian3}(c)-(d). One of these intermediate states is $B^-$, which yields a saddle-repeller-to-saddle connection with the state $B^+$ that satisfy $\lambda_s(B^-)=\sigma_B$, with $\sigma_B=\sigma_B(B^-;B^+)$. For this specific state, the unstable manifold of $B^-$ should coincide with the stable manifold of $B+$. This means that the triple $(B^-,B^+,\sigma_B)$ is in the slow characteristic boundary of $\mathcal{T}$.
Figure \ref{fig:sattrian3} illustrates the schematic form of the procedure to find the slow characteristic boundary. 

\begin{figure}[ht]
	\centering
	\subfigure[$U^-$, $U^+$ inside and $V^-$ outside of the undercompressive region. ]{\includegraphics[width=0.45\textwidth]{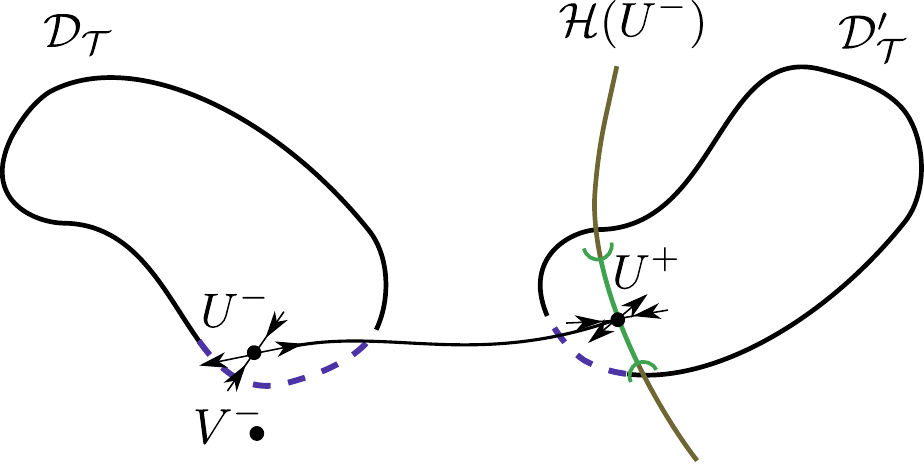}}\hspace{3.5mm}
	\subfigure[$U^-$ is a saddle-node and $\widetilde{U}^+\in\mathcal{H}(U^-)$ is a saddle. There is not an orbit connecting  them.]{\includegraphics[width=0.45\textwidth]{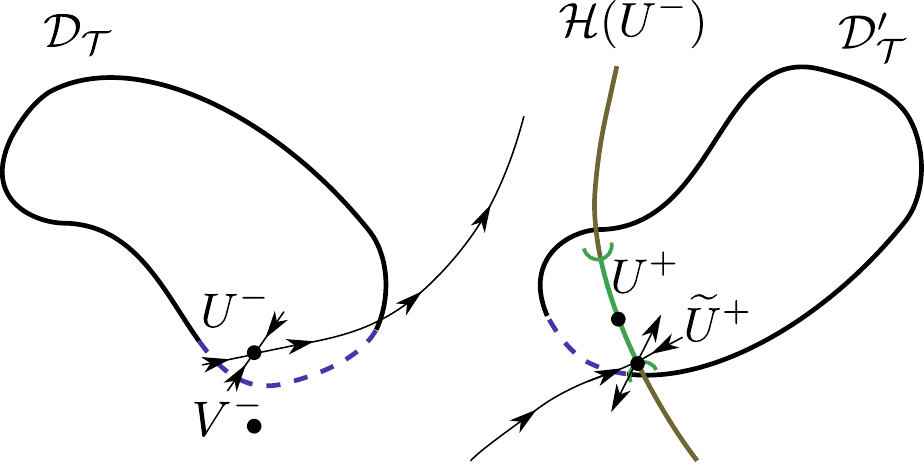}}
	\subfigure[$V^-$ is a saddle-node and $V^+\in\mathcal{H}(V^-)$ is a saddle. There is not an orbit connecting them.]{\includegraphics[width=0.45\textwidth]{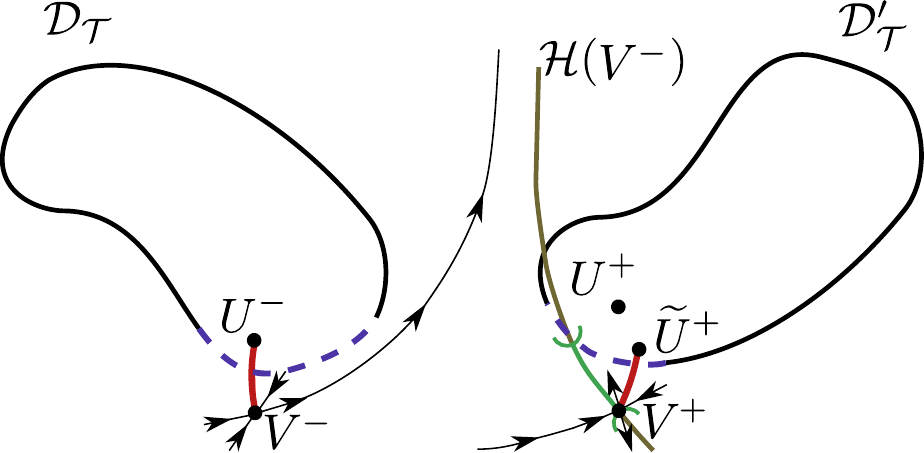}}\hspace{3.5mm}
 \subfigure[$(B^-,B^+,\sigma_{B})$ belongs to  the slow characteristic boundary of $\mathcal{T}$. ]{\includegraphics[width=0.45\textwidth]{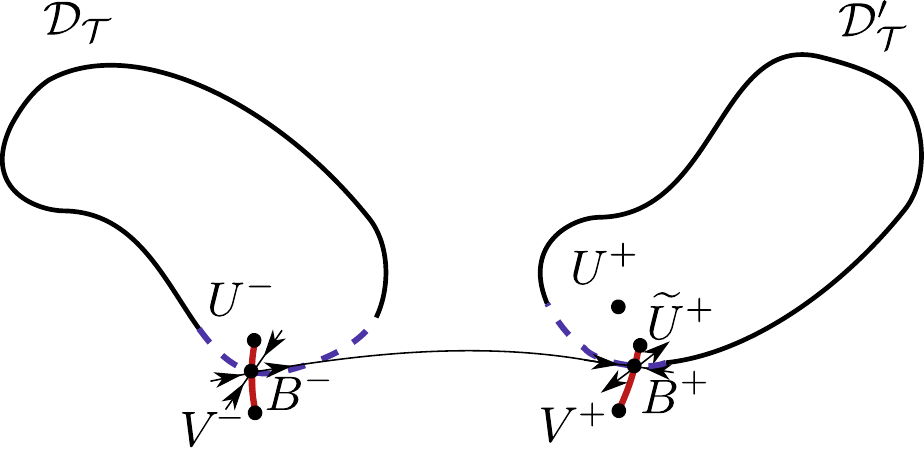}}
	\caption{Schematic procedure to find connections between saddle-node and saddle.
 }
	\label{fig:sattrian3}
\end{figure}

To approximate $B^-$, we proceed as in Section \ref{section:saddleconn} and get the intersection points of the stable and unstable manifolds with a conveniently chosen transversal line segment, deriving a vector that tells us the distance between the manifolds and their position relative to each other. We can use this information to start a binary search/bisection algorithm that approximates the desired points with a given accuracy target. We refer to \ref{Alg:2} for details.

When the above procedure fails, we can apply an alternative binary search algorithm that provides a tolerance-controlled approximation point to the undercompressive boundary point. The idea behind this alternative procedure is to make the successful saddle-to-saddle connection approach the unsuccessful one, and to do so, we consecutively test the middle state $\frac{1}{2}\left(U^-+U^{-}_*\right)$ for saddle-to-saddle connections, updating $U^-$ and $U^{-}_*$ in the process. We refer to \ref{Alg:3} for details.

\subsection{Undercompressive boundaries}

In the previous two sections, we have described how to calculate individual points of the interior and boundary of the undercompressive surface. The algorithms discussed in these two sections allow us to calculate saddle-to-saddle connections and saddle-to-saddle-node connections individually. Therefore, a natural way to calculate the boundary of the undercompressive surface is first to cover a portion of the state space with grid points, second to test them for saddle-to-saddle connections, and third to handle cases where two consecutive points in the grid satisfy the property that one has a saddle-to-saddle connection while the other one does not. In this final step, we apply one of the two algorithms described in Section \ref{section:SnodeSconection} to approximate the boundary point between these two consecutive states. This contour methodology requires using an $n$-dimensional grid to compute an $(n-1)$-dimensional space, which becomes computationally expensive as the number of grid points increases. Fortunately, the information obtained to construct one boundary point can be successively reused to calculate additional boundary points at a predefined target distance. This approach enables us to refine the undercompressive boundaries by controlling the target distance, effectively covering the $(n-1)$-dimensional space with an $(n-1)$-dimensional grid. Further details on this procedure can be found in \ref{ApendiceA}.

When calculating the undercompressive boundaries with $\bmm(U) \neq I$, we obtain the same boundaries as those discussed for $\bmm(U) = I$ (see Section \ref{section:TransSurface}). The key difference is that, in the first case, the undercompressive shocks form a two-dimensional projection region within the saturation triangle, whereas in the latter case, they are restricted to a one-dimensional space. Figure \ref{fig:windmill} illustrates the surface of undercompressive shocks $\mathcal{T}$ for the umbilic point of type $II$.

\begin{figure}[ht!]
	\centering
	\subfigure[Projection of the surface of undercompressive shocks in the saturation triangle.]{\includegraphics[width=0.415\textwidth]{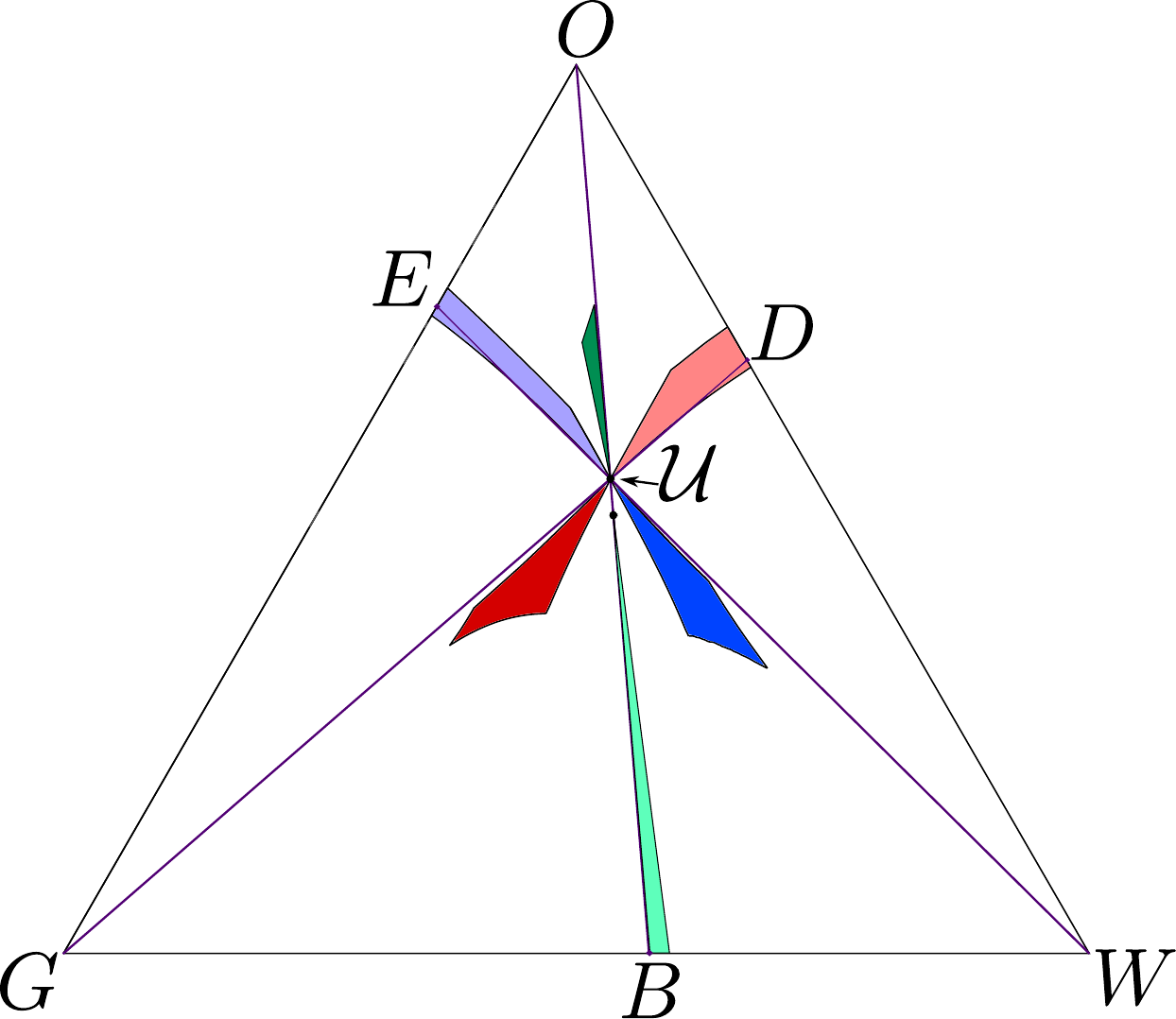}}  \hspace{0.50mm}
	\subfigure[Surfaces of undercompressive shocks associated with segments ${[G,D]}, {[W,E]}$ and ${[O,B]}$  in the 3-dimensional space.]{\includegraphics[width=0.55\textwidth]{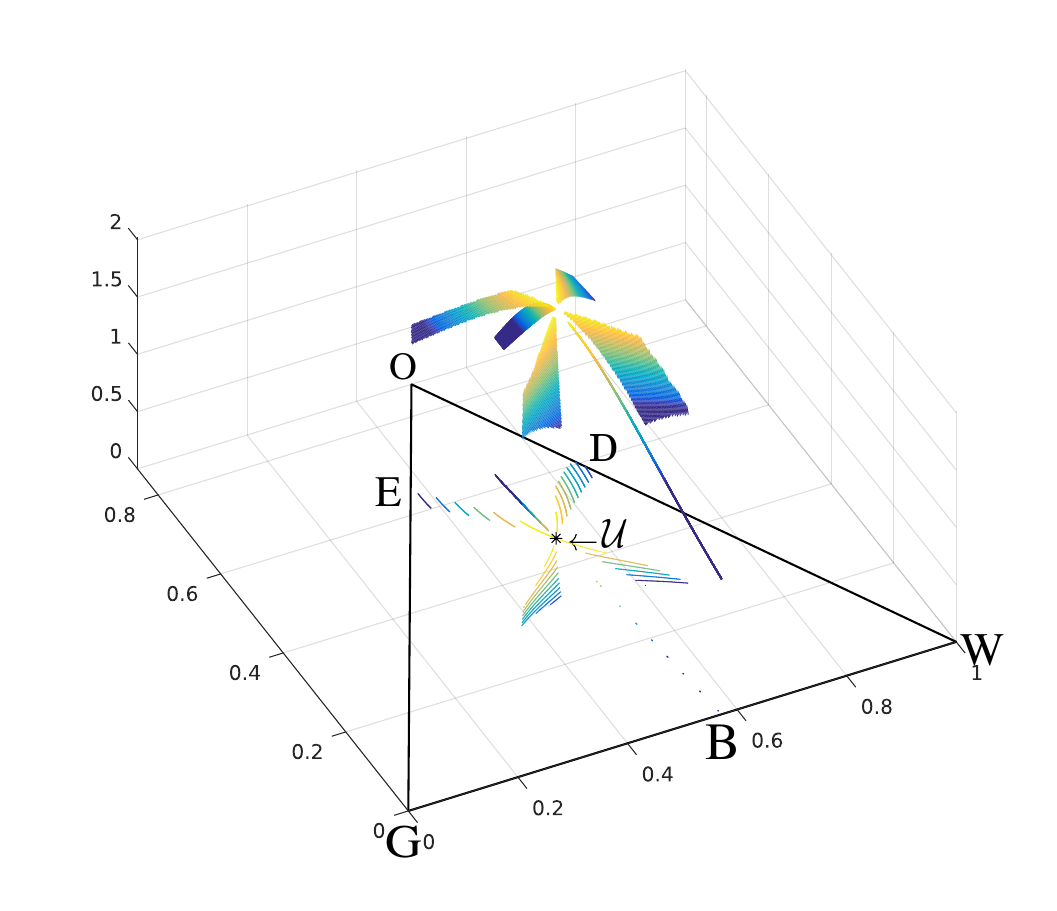}} 
	\caption{Surfaces of undercompressive shocks for umbilic point of type $II_O$ and $\bmm(U)\neq I$. (a) Dark red (dark blue, dark green) represents the domain $\mathcal{D}_\mathcal{T}$ and the light red (light blue, light green) represents the codomain $\mathcal{D}'_\mathcal{T}$. The size of the pair of green regions was modified so that it could be visible. (b) The $z$ axis represents the shock speed.}\label{fig:windmill}
\end{figure}

Now, we define the boundaries of the undercompressive shock surface for $\bmm(U) \neq I$. Similar to the case where $\bmm(U) = I$, we can associate a surface with the straight lines $[G, D]$, $[W, E]$, and $[O, B]$. Based on extensive numerical experiments, we conclude that the type and number of limits remain consistent with the case of the identity matrix. The cases can be enumerated as follows:
\begin{enumerate}
	\item {\bf Umbilic Point of Type $I$:} When the umbilic point is of type $I$, the three surfaces of the undercompressive shocks have four boundaries, analogous to those described in Definition \ref{Def:TranBoun}.
	\item {\bf Umbilic Point in Region $II_\Gamma$ (for $\Gamma \in \{G, W, O\}$):} The surface of undercompressive shocks associated with the segment $[\Gamma, \mathbb{B}]$, where $\mathbb{B} \in \{D, E, B\}$, has three boundaries and a gap between $\um$ and $\mathbb{B}_0 \in \{D_0, E_0, B_0\}$. This is similar to the case described in Lemma \ref{lem:transnu3bondariesmenor1} (note that in this case, $\nu_\Gamma < 1$).
	\item {\bf Umbilic Point in Region $II_\Gamma$, with $\nu_{\mathcal{A}} > 8$:} For $\Gamma \in \{G, W, O\}$, the surface of undercompressive shocks associated with the segment from vertex $\mathcal{A}$, where $\mathcal{A} \in \{G, W, O\}$, $\mathcal{A} \neq \Gamma$, and $\nu_{\mathcal{A}} > 8$, has three boundaries, similar to the case described in Lemma \ref{lem:transnu3bondariesmayor8}.  
	\item {\bf Umbilic Point in Region $II_\Gamma$, with $\nu_{\mathcal{A}} \leq 8$:} For $\Gamma \in \{G, W, O\}$, the surface of undercompressive shocks associated with the segment from vertex $\mathcal{A}$, where $\mathcal{A} \in \{G, W, O\}$, $\mathcal{A} \neq \Gamma$, and $\nu_{\mathcal{A}} \leq 8$, can have three or four boundaries. This depends on whether the points in the double contact, $\mathfrak{Y}^\Gamma$, and ${\mathscr{Y}}^\Gamma$, exist within $\Omega$. The states $\mathfrak{Y}^\Gamma$ and ${\mathscr{Y}}^\Gamma$ serve the same role as $\mathcal{Y}^\Gamma$ and $\text{Y}^\Gamma$ mentioned in Remark \ref{rem:DoubleContact}. See Fig.~\ref{fig:windmillA} for an illustration. This scenario is analogous to the cases described in Lemmas \ref{lem:transnu4bondaries} and \ref{lem:transnu3bondariesmayor8}.   
\end{enumerate}   
Notice that, as mentioned in Remark \ref{rem:twowings}, the surface of undercompressive shocks $\mathcal{T}$, shown in Fig.~\ref{fig:windmill}, reveals six regions when projected into $\Omega$ (or $\Omega \times \mathbb{R}$). The dark regions (red, blue, and dark green in Fig.~\ref{fig:windmill}(a)) represent the domain $\mathcal{D}_\mathcal{T}$, while the light regions (red, blue, and light green in Fig.~\ref{fig:windmill}(a)) represent the codomain $\mathcal{D}'_\mathcal{T}$. This color scheme enables precise localization of each pair of regions ($\mathcal{D}_\mathcal{T}$ and $\mathcal{D}'_\mathcal{T}$) with the same color, which are connected by $u$-shock waves. Furthermore, it facilitates the correlation between these region pairs ($\mathcal{D}_\mathcal{T}$ and $\mathcal{D}'_\mathcal{T}$) and the segments $[G, D]$, $[W, E]$, and $[O, B]$. 
Figure~\ref{fig:windmillA} illustrates $\mathcal{D}_\mathcal{T}$ and $\mathcal{D}'_\mathcal{T}$ of the surface of undercompressive shocks associated with $[G, D]$, identifying the limiting states $\mathfrak{D}_1$, ${\mathscr{D}}_1$, $\mathfrak{D}_2$, ${\mathscr{D}}_2$, $\mathfrak{Y}^{G}$, and ${\mathscr{Y}}^{G}$. Additionally, the states $D_1$, $D_2$, $\mathcal{Y}^{G}$, and $\text{Y}^{G}$, which define the surface in the case $\bmm(U) = I$, are also shown. The boundaries of $\mathcal{D}_\mathcal{T}$ and $\mathcal{D}'_\mathcal{T}$ associated with $[G, D]$ projected in $\Omega$ are as follows:
\begin{figure}[ht]
	\centering
	\includegraphics[width=0.45\textwidth]{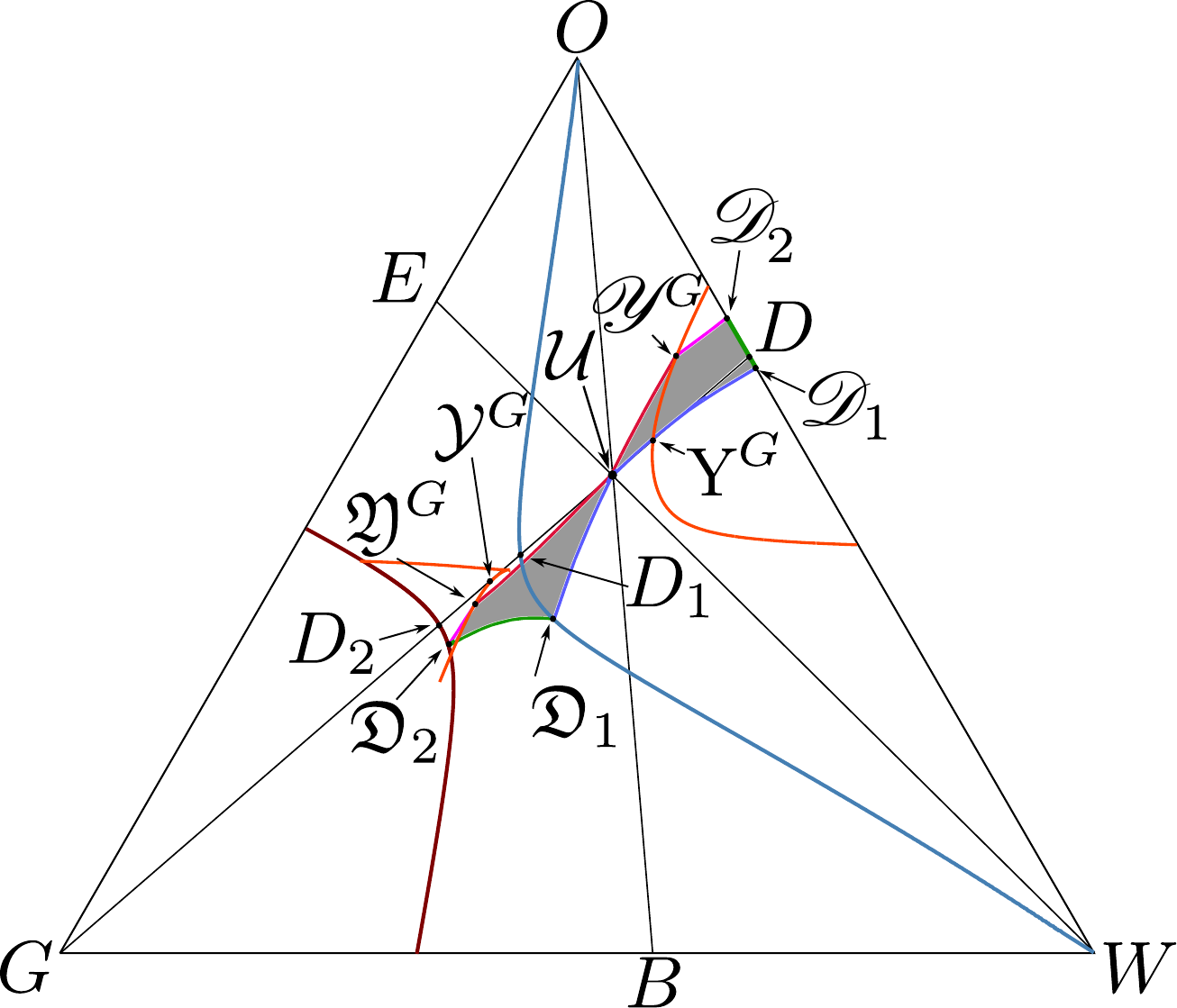}
	\caption{Projection of the surface of undercompressive shocks on $\Omega$ (shaded region). The states $D_1$, $D_2$, $\mathcal{Y}^{G}$, and $\text{Y}^{G}$ lie on $[G,D]$. These states define the surface of undercompressive shocks when $\bmm(U)=I$, see Section \ref{section:reduced}. The curve $[O, D_1,\mathfrak{D}_1, W]$ is the $s$-extension of $[O, W]$, the curve that crosses $[D_2,\mathfrak{D}_2]$ is the $f$-extension of $[O, W]$, and the curve segments $[{\mathscr{Y}}^{G}, \text{Y}^{G}]$ and $[\mathfrak{Y}^G, \mathcal{Y}^{G}]$ are part of the fast double contact locus (orange curves). The states ${\mathscr{D}_1}$ and ${\mathscr{D}_2}$ lie on $[O,W]$ and are associated with $\mathfrak{D}_1$ and $\mathfrak{D}_2$ such that $\sigma(\mathfrak{D}_1,{\mathscr{D}_1})=\lambda_s(\mathfrak{D}_1) $ and $\sigma(\mathfrak{D}_2,{\mathscr{D}_2})=\lambda_f(\mathfrak{D}_2) $. }	
 \label{fig:windmillA}
\end{figure}
\begin{itemize}
		\item A \textit{slow characteristic boundary} $ (SCB),$ defined by the set of states $ U^- \in [\mathfrak{D}_1, \um] $, each with a corresponding state $ U^+ \in [\um, \mathscr{D}_1] $, such that there is an admissible left-characteristic $ s $-shock between $ U^- $ and $ U^+ $, with $ \sigma(U^-; U^+) = \lambda_s(U^-) $. The state $ \mathscr{D}_1 $ lies on $ [W, O] $, and its corresponding state $ \mathfrak{D}_1 $ lies in the $ s $-right-extension of $ [W, O] $.
		\item A \textit{fast characteristic boundary} $ (FCB),$ defined by the set of states $ U^- \in [\mathfrak{Y}^G, \um] $, each with a corresponding state $ U^+ \in [\um, {\mathscr{Y}}^G] $, such that there is an admissible right-characteristic $ f $-shock between $ U^- $ and $ U^+ $, with $ \sigma(U^-; U^+) = \lambda_f(U^+) $. The states $ \mathfrak{Y}^G $ and $ {\mathscr{Y}}^G $ lie on the double fast contact locus.
  		\item An \textit{undercompressive characteristic boundary} $ (UCB),$ defined by the set of states $ U^- \in [\mathfrak{Y}^G, \mathfrak{D}_2] $, each with a corresponding state $ U^+ \in [{\mathscr{Y}}^G, {\mathscr{D}}_2] $, such that there is an admissible left-characteristic $ u $-shock between $ U^- $ and $ U^+ $, with $ \sigma(U^-; U^+) = \lambda_f(U^-) $. The states $ \mathscr{D}_2 $ and $ \mathfrak{D}_2 $ correspond to $ [W, O] $ and its $ f $-right-extension, respectively.
		\item A \textit{genuine undercompressive boundary} $ (GUB) $ defined by the set of states $ U^- \in (\mathfrak{D}_1, \mathfrak{D}_2) $, each with a corresponding state $ U^+ \in (\mathscr{D}_1, \mathscr{D}_2) $, such that there is an admissible undercompressive shock between $ U^- $ and $ U^+ $.
\end{itemize} 
In Fig.~\ref{fig:windmillB}, we compare the projection of the surface of undercompressive shocks on the plane $[G,D]\times \mathbb{R}^+$ for the case $\bmm(U)\neq I$ with the case $\bmm(U)= I$.    
Notice that the structure of these boundaries is the same as those defined for the identity matrix; see Definition \ref{Def:TranBoun}.

\begin{figure}[ht]
	\centering
	\subfigure[Case $\bmm(U)=I$.]{\includegraphics[width=0.45\textwidth]{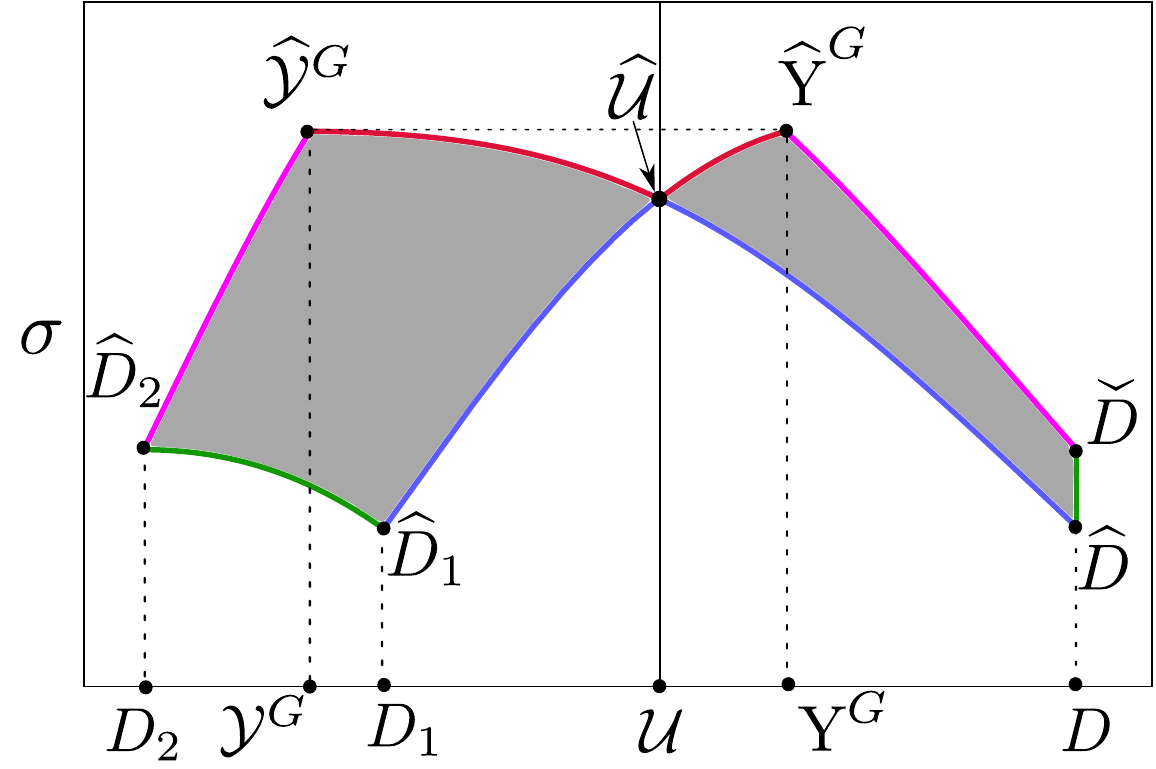}}  \hspace{0.5mm}
	\subfigure[Case $\bmm(U)\neq I$.]{\includegraphics[width=0.45\textwidth]{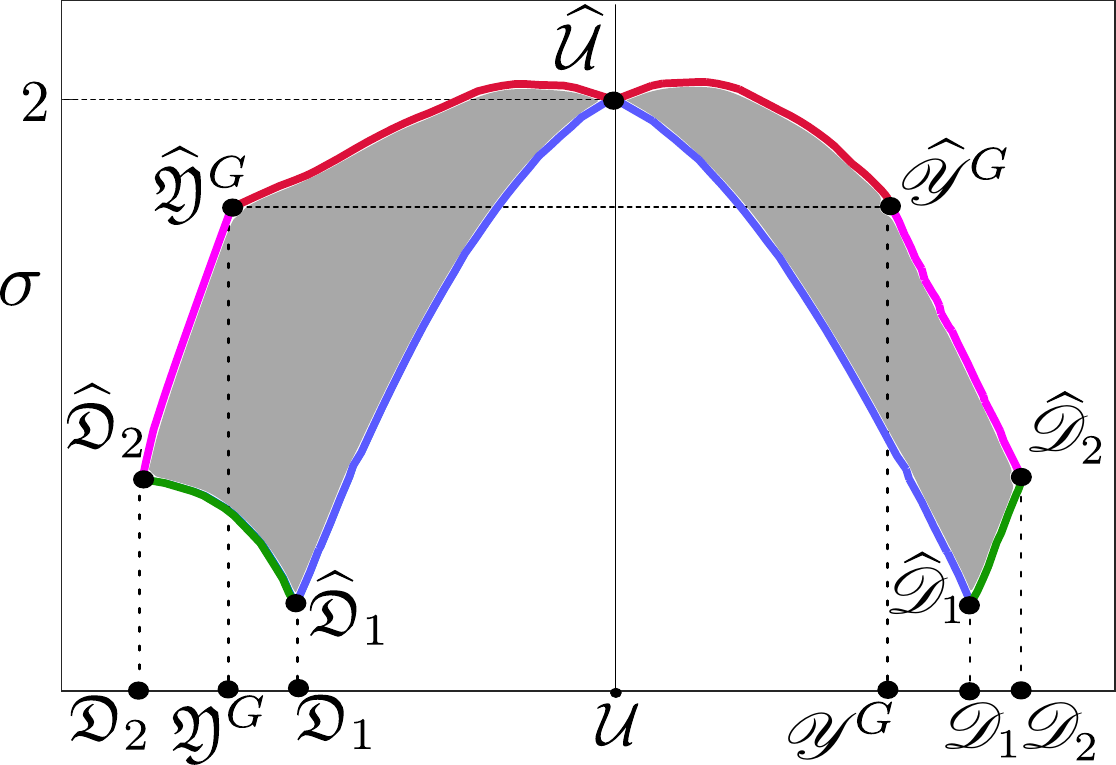}} 
	\caption{Projection of the surface of undercompressive shocks on the plane $[G,D]\times \mathbb{R}^+$ for identity and non-identity matrices. (a) The states $D_1,D_2,\mathcal{Y}^{G}$, and $\text{Y}^{G}$ lie on $[G,D]$. These states define the surface of undercompressive shocks when $\bmm(U)=I$, see Section \ref{section:TransSurface}. (b) The states $\mathfrak{D_1},\mathfrak{D_2},\widehat{\mathfrak{D_1}}, \widehat{\mathfrak{D_2}},\mathfrak{Y}_2^G$, and $\widehat{\mathfrak{Y}}_2^G$ lie inside the saturation triangle. The blue curves are slow characteristic boundaries $(SCB)$, red curves are fast characteristic boundaries $(FCB)$, pink curves are undercompressive characteristic boundaries $(UCB)$, and green curves are genuine undercompressive boundaries $(GUB)$.    }	\label{fig:windmillB}
\end{figure}
\subsection{Simulations}
In this section, we present numerical simulations illustrating the effect in the solution when the matrix $\bmm(U)\neq I$ is used. We chose the nonlinear Crank-Nicolson implicit finite-difference scheme, using Newton\textquotesingle s method to perform numerical simulations for the system \eqref{eq:sistem2}; see \cite{lambert2020mathematics}. This scheme is second-order accurate in space and time.  We consider umbilic point type $II_O$ with $\mu_w = 1.0, \mu_o = 2.0$, and $\mu_g=0.75$. We also consider the viscosity matrix defined in \eqref{system6}-\eqref{PUprima} with $c_{ow} = c_{og}=1$.

\begin{figure}[ht]
	\centering
	\subfigure[Case $\bmm(U)=I$. ]{\includegraphics[width=0.475\textwidth]{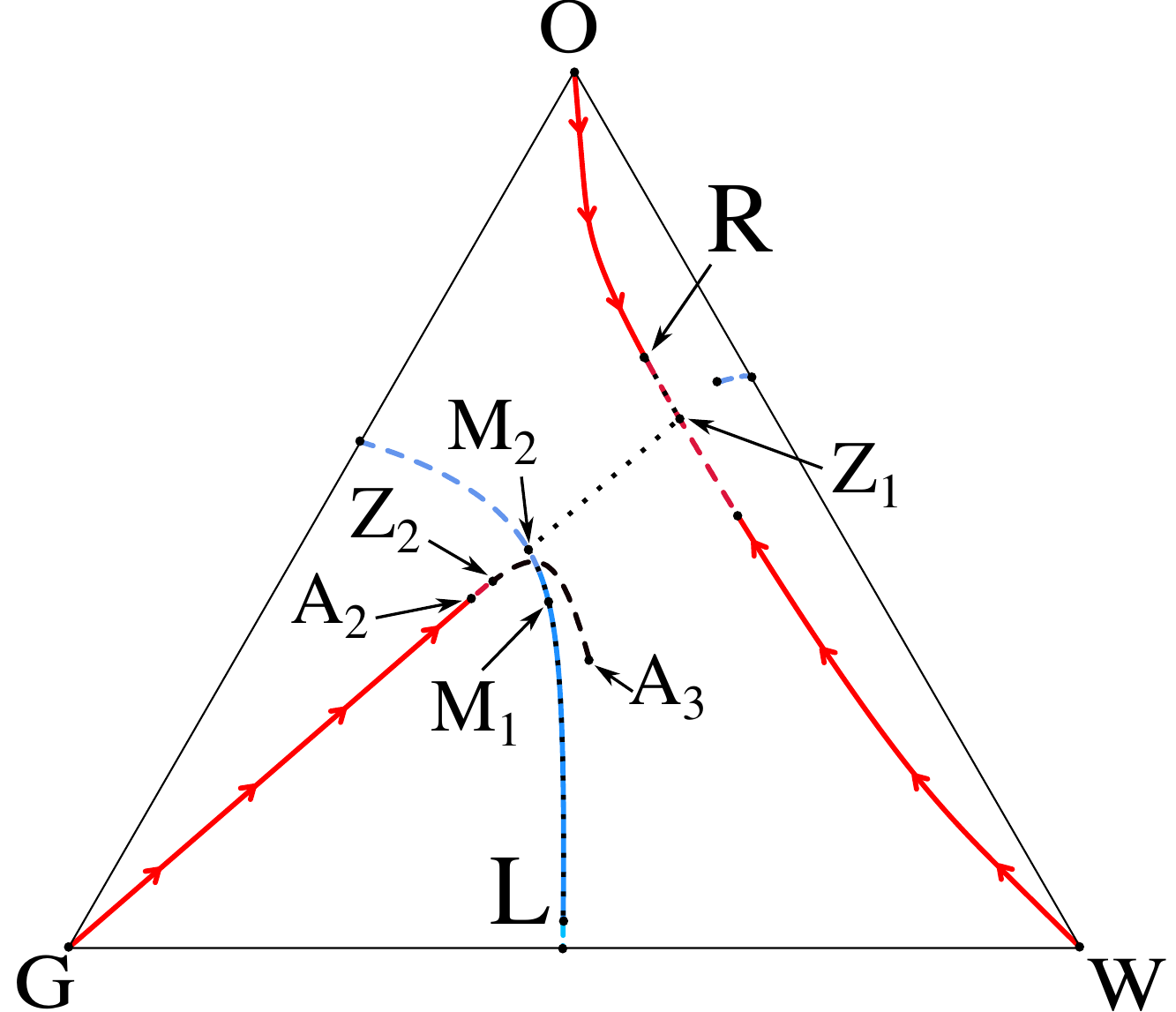}}
     \hspace{1pt}
     \subfigure[Case $\bmm(U)\neq I$.]{\includegraphics[width=0.475\textwidth]{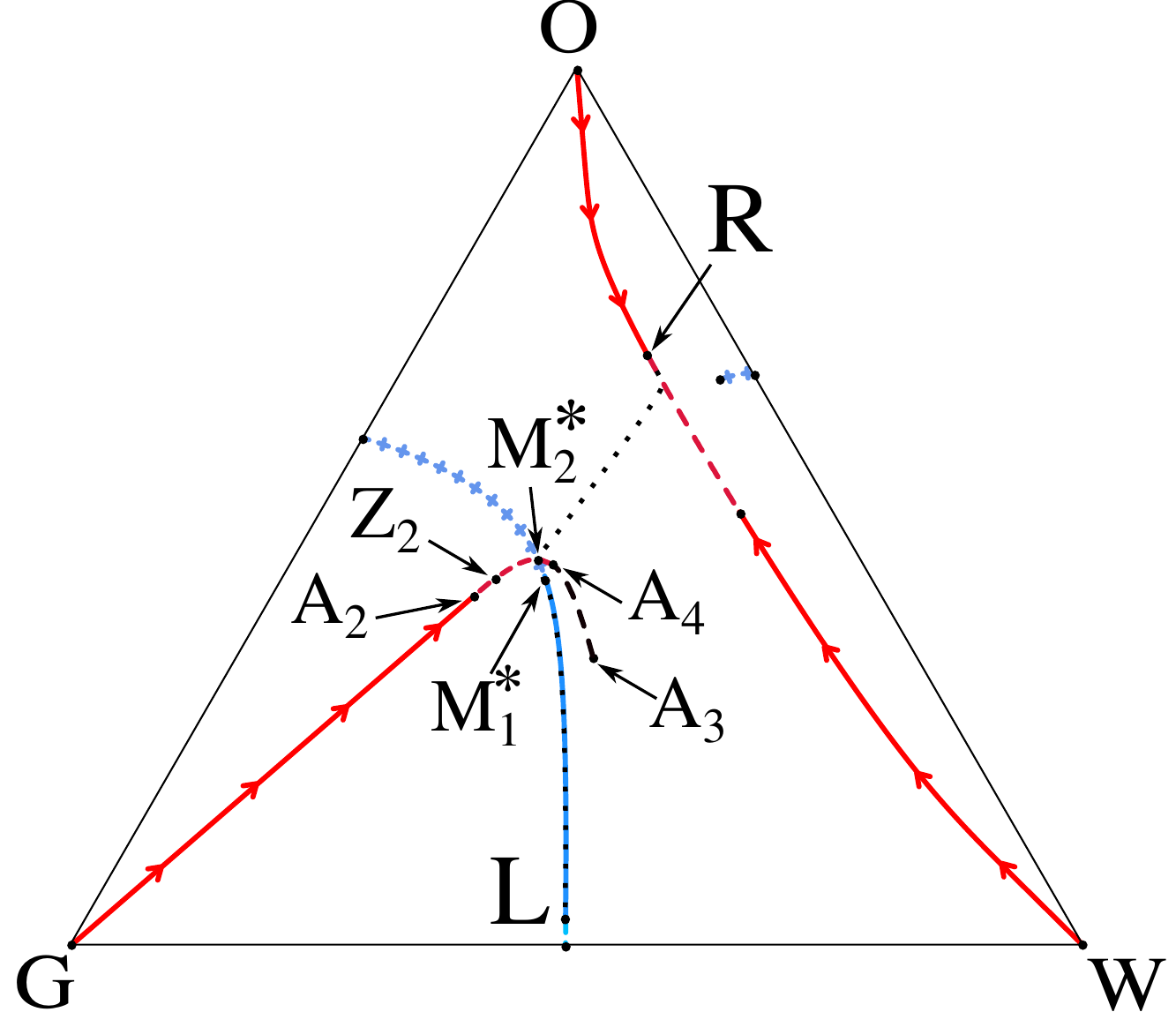} }
\caption{Solution of the Riemann problem for left state $L=(0.475708,0.02608)^T$ and right state $R=(0.235578, 0.670876)^T$. We present the solution using the wave curve method \protect\cite{V.2010}  (analytic solution) and the numerical solution (simulation using RCD \protect\cite{lambert2020mathematics}). The red curves represent $\wm_f^-(R)$; the blue curves, $\wm_s^+(L)$, see Section \protect\ref{sec:wavecurve}. The black dotted curve is the numerical simulation using RCD. (a) The dashed curve $[A_2,A_3]$ is a nonlocal fast shock segment of $\Hm(R)$, but for $\bmm(U)=I$. The segment $[A_2, Z_2]$ has viscous profile. (b)  The dashed curve $[A_2,A_3]$ is a nonlocal fast shock segment of $\Hm(R)$, but for $\bmm(U)\neq I$, the segment $[A_2,A_4]$ has viscous profile.   }
	\label{fig:SolSaturationTriangle}
\end{figure}

We first consider the case where $\bmm(U) = I$. Figure \ref{fig:SolSaturationTriangle}(a) displays the backward fast wave curve $\wm_f^-(R)$ and the forward slow wave curve $\wm_s^+(L)$ with $L = (0.475708, 0.02608)^T$ and $R = (0.235578, 0.670876)^T$; see Section \ref{sec:wavecurve}. Based on \cite{Lozano2018, Lozano2020pro}, the solution of the Riemann problem comprises the sequence:  
$$L  \xrightarrow{R_s} M_1 \xrightarrow{\CS_s} M_2 \xrightarrow{S_u} Z_1 \xrightarrow{S_f} R,$$
which consists of an $s$-rarefaction wave from $L$ to $M_1$, followed by an $s$-shock connecting $M_1$ and $M_2$, followed by a $u$-shock from $M_2$ to $Z_1$, and finally, an $f$-shock connecting $Z_1$ and $R$.  
This solution contains three wave groups, including an undercompressive shock between the states $M_2$ and $Z_1$ on the invariant segment $[G, D]$. Figure \ref{fig:SolSaturationTriangle}(a) also shows the nonlocal shock segment $[A_2, A_3] \subset \Hm(R)$. This curve intersects the invariant segment $[G, D]$ at $Z_2$, such that the shock segment $[A_2, Z_2] \subset [A_2, A_3]$ is admissible, whereas the segment $(Z_2, A_3]$ is not.  

Figure \ref{fig:SimulationTran}(a) illustrates the profiles for saturations $S_w$ and $S_o$ obtained from the numerical simulation. The simulation uses $\Delta x = 0.01$, $\Delta t = 0.01$, a final time of $T_f = 100$, and a small amount of artificial diffusion, $\epsilon = 0.005$. Note that the three wave groups and the two constant states of this solution are well-represented in the simulation results.

Next, we consider the case where the viscosity matrix $\bmm(U) \neq I$. Figure \ref{fig:SolSaturationTriangle}(b) shows the backward fast wave curve $\wm_f^-(R)$ and the forward slow wave curve $\wm_s^+(L)$ for the previously defined states $R$ and $L$. In this case, the admissibility of the nonlocal fast shock segment $[A_2, A_3] \subset \Hm(R)$ changes. Specifically, the shock segment $[A_2, Z_2, A_4] \subset [A_2, A_3]$ becomes admissible when $\bmm(U) \neq I$. Furthermore, there is an intersection between $\wm_s^+(L)$ and $\wm_f^-(R)$ at the point $M_2^*$, which satisfies the following:
\begin{itemize}
	\item The state $M_2^*$ lies on the $s$-composite segment of $\wm_s^+(L)$. Thus, there exists a state $M_1^*$ in the $s$-rarefaction segment of $\wm_s^+(L)$ such that $\sigma_s = \sigma(M_1^*; M_2^*) = \lambda_s(M_1^*)$.  
	\item The state $M_2^*$ lies on the $f$-shock segment of $\wm_f^-(R)$ with $\sigma_f = \sigma(M_2^*; R)$, where this shock has a viscous profile, and $\sigma_s < \sigma_f$.  
\end{itemize} 
The solution to the Riemann problem in this case is given by (see \cite{Lozano2018} for details):  
$$
L \xrightarrow{R_s} M_1^* \xrightarrow{\CS_s} M_2^* \xrightarrow{S_f} R,
$$  
which consists of an $s$-rarefaction wave from $L$ to $M_1^*$, followed by an $s$-shock connecting $M_1^*$ and $M_2^*$, and finally, an $f$-shock from $M_2^*$ to $R$.  

Therefore, when we change the viscosity matrix, the solution only uses an intermediate state, $M_2^*$, and the undercompressive shock utilized when $\bmm(U)=I$ vanishes. Figure \ref{fig:SimulationTran}(b) shows the saturation profiles $S_w$ and $S_o$ obtained from the numerical simulation for this case. We consider $ \Delta x = 0.01 $, $ \Delta t = 0.01 $, final time $ T_f = 100 $, and the matrix given by the capillary effects.

\begin{figure}
	\centering
	\subfigure[Solution profile case $\bmm(U)=I.$]{\includegraphics[width=0.4\textwidth]{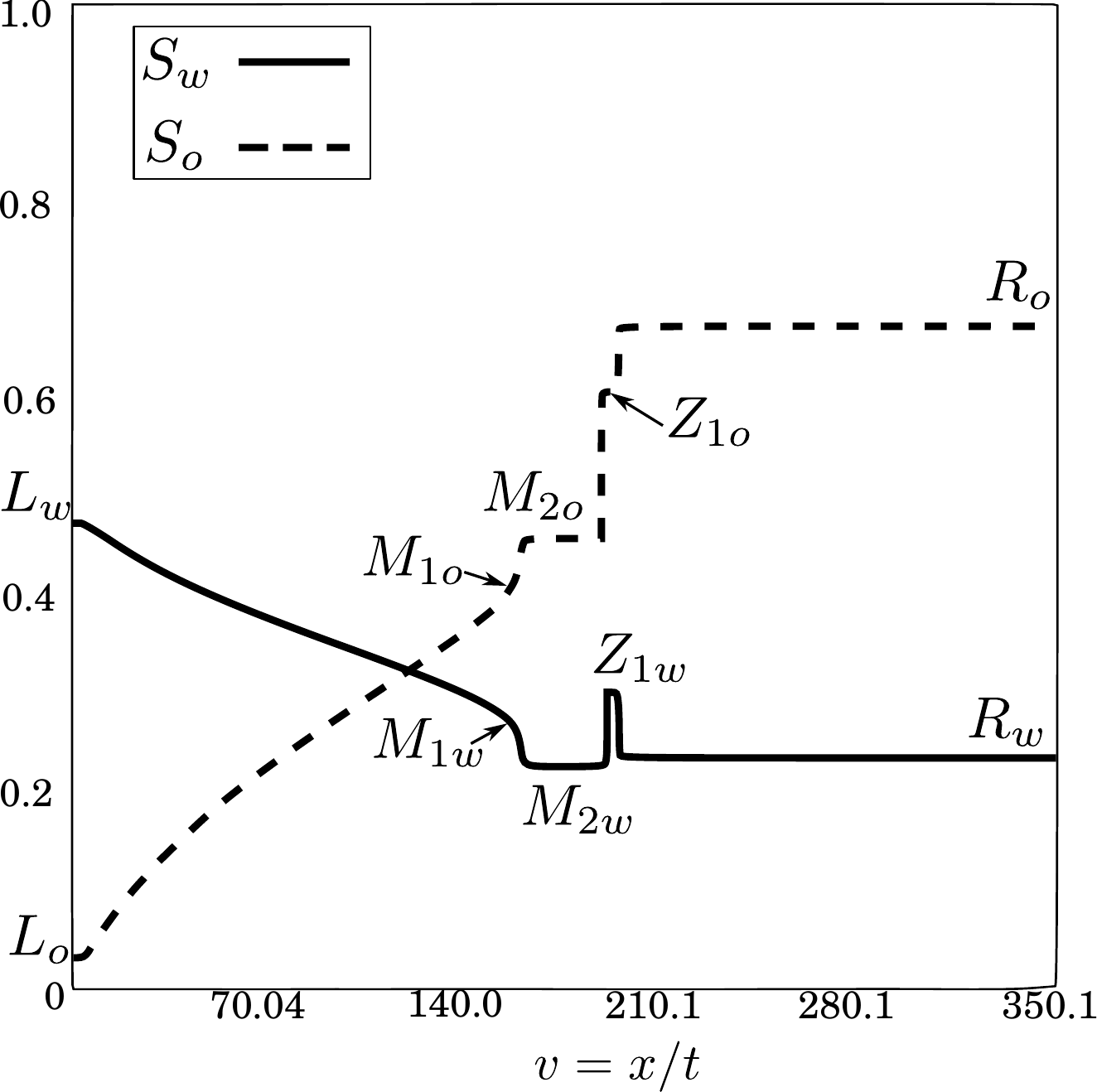}}  \hspace{0.65mm}
	\subfigure[Solution profile case $\bmm(U)\neq I.$]{\includegraphics[width=0.4\textwidth]{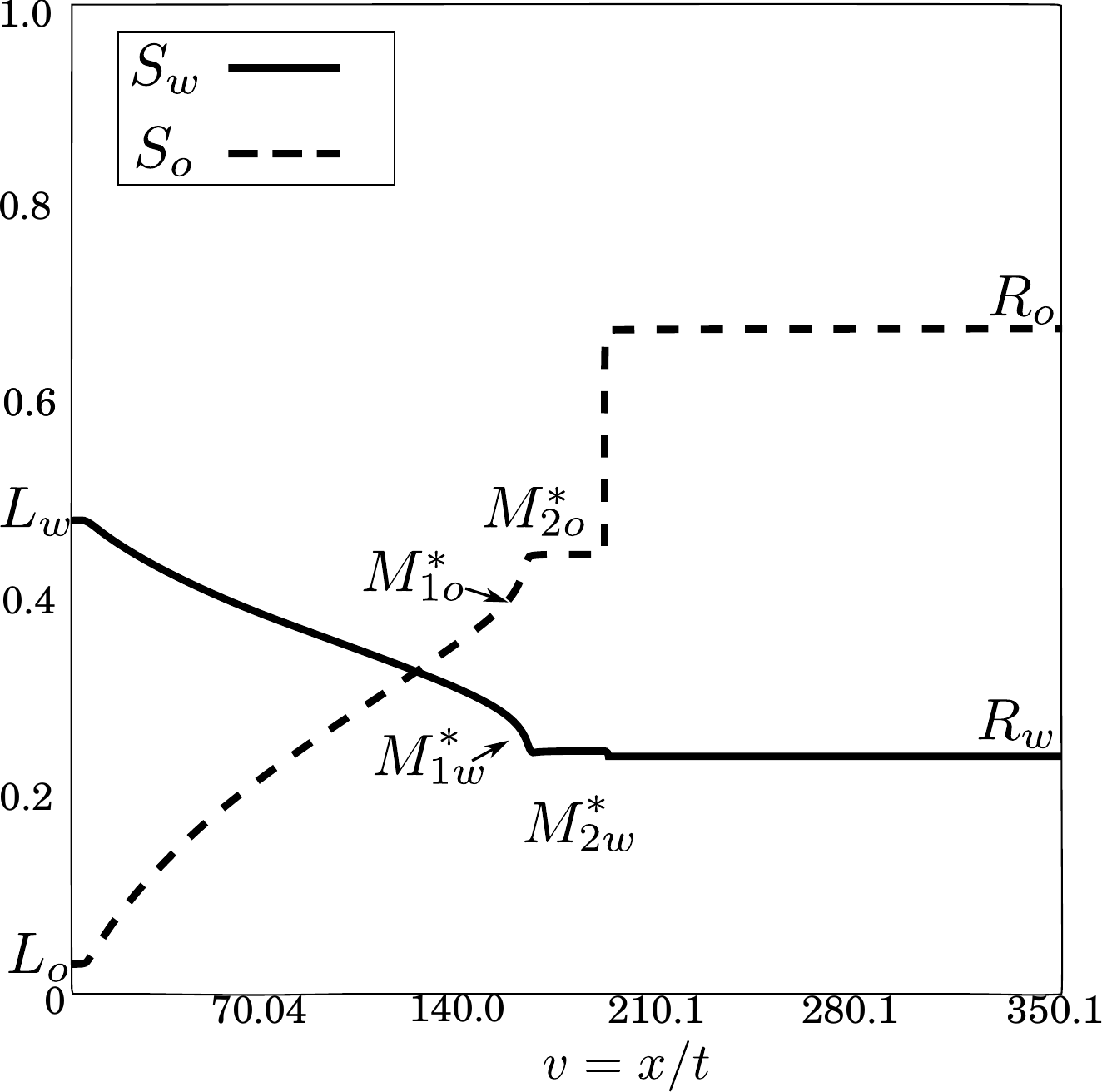}} 
	\caption{Numerical solution of Riemann problem for left state $L=(0.475708,0.02608)^T$ and right state $R=(0.235578, 0.670876)^T$. (a) For $\bmm(U)=I,$ the Riemann solution comprises a $s$-composite from $L$ to $M_2$ with intermediate state $U$, followed by a $u$-shock from $M_2$ to $Z_1$, followed by a $f$-shock from $Z_1$ to $R$. (b) For  $\bmm(U)\neq I$, the solution only comprises a slow wave group followed by a fast shock wave.  The horizontal axis is given by $v = x/t$. The solution profiles are shown at $t=100$. }  	
 \label{fig:SimulationTran}
\end{figure}

\section{Conclusions}
\label{sec:_conclus} 
In this paper, we have constructed the undercompressive shock surface in the expanded 3-dimensional state space of saturations and speed. This was done for two types of diffusion matrices. The first type is the identity matrix; the second arises from capillary pressure effects. In the case of the identity matrix, we have demonstrated analytically the construction as a ruled surface that lies over suitable planes perpendicular to the state space and intersects it along the invariant lines. In the case of the general diffusion matrix, we present the numerical procedure for constructing it. We found that the structure of the undercompressive surfaces is similar for these two diffusion matrices. Although projections $\pi^-$ and $\pi^+$ are singular when $\bmm(U)=I$ (are not invertible) and, therefore, there is no undercompressive mapping, the algorithm we present in constructing the undercompressive surface can still be applied to identify the left and right states that constitute a $u$-shock, which is essential in solving a Riemann problem. 

The surface of undercompressive shock allows us to characterize the set of solutions to the Riemann problem containing three wave groups, with the middle one being transitional because it contains an undercompressive shock wave. Understanding this phenomenon was made more accessible by working in the extended state space. In particular, this characterization facilitates the identification of cases where speed compatibility between classic and transitional wave groups can fail.

Finally, we showed how the choice of the capillarity matrix influences the solutions of Riemann problems. This influence encompasses not only the emergence and disappearance of undercompressive shock waves within a solution but also the alteration of the admissibility of specific shock segments previously restricted by the criterium of viscous profile.

\section*{Acknowledgments}

The authors are grateful to Prof. Aparecido de Souza and Prof. Frederico Furtado for fruitful discussions and comments that improved the manuscript.

 L. L. gratefully acknowledge support from Shell Brasil through the project ``Avançando na modelagem matemática e computacional para apoiar a implementação da tecnologia `Foam-assisted WAG' em reservatórios do Pré-sal'' (ANP 23518-4)  at UFJF and the strategic importance of the support given by ANP through the R\&D levy regulation. 
 
D. M. and L. L. were partly supported by CNPq grant 405366/2021-3.

D. M. was partly supported by CNPq grant 306566/2019-2 and FAPERJ grants E-26/210.738/2014, E-26/202.764/2017, E-26/201.159/2021.

\appendix

\section{Algorithms}
\label{ApendiceA}
In this section, we present the Algorithms implemented in Software ELI to construct the surface of undercompressive shocks for the case $\bmm(U)\neq I$. We are restricting this discussion to the two-dimensional case. 

To ease the exposition of the algorithms described in this section, we need to make a few definitions: 
\begin{itemize}
    \item $\overline{s}$ refers to a line conveniently placed in the phase space. Once the Sotomayor line $\overline{s}$ is chosen, it remains fixed (at least for as long as we are calculating one individual saddle-to-saddle or saddle-to-saddle-node connection);
    \item the subscripts $[\cdot]^-$ and $[\cdot]^+$ are used to differentiate two different states in a Hugoniot locus, both equilibria of \eqref{V9};
    \item given two states $U^-$ and $U^+$, we define the vector $\textbf{d}(U^-, U^+)$ to satisfy the following conditions:
\begin{itemize}
    \item $\textbf{d}(U^-, U^+)$ lives in the line $\overline{s}$;
    \item the starting point of the vector $\textbf{d}(U^-, U^+)$ coincides with a point of a stable manifold for $U^-$ (or else unstable manifold);
    \item the ending point of the vector $\textbf{d}(U^-, U^+)$ coincides with a point of an unstable manifold for $U^+$ (or else stable manifold);
\end{itemize}
\end{itemize}
We may simply write $\textbf{d}$ instead of $\textbf{d}(U^-, U^+)$ if the states $U^-$ and $U^+$ are obvious. Notice that, depending on the choice of $\overline{s}$, the calculation of $\textbf{d}(U^-, U^+)$ may fail due to several reasons, such as orbits being asymptote to the line or orbits ending in another equilibrium before they intersect the line. As such, even though the line $\overline{s}$ must remain fixed while we calculate a single saddle-to-saddle or saddle-to-saddle-node connection, it may vary when we calculate the undercompressive region as a whole.

\subsection{saddle-to-saddle Connection}

In order to start our algorithm for calculating saddle-to-saddle connections, we need to do a search over the Hugoniot locus for $U^-$. This search provides us with two two points $U^+_p$ (previous) and $U^+_n$ (next), such that what is shown in Figure\label{fig:Algo_Saddle_Saddle}(a) happens: the difference vectors $\textbf{d}_p:=\textbf{d}(U^-, U^+_p)$ and $\textbf{d}_n:=\textbf{d}(U^-, U^+_n)$ point to opposite directions. Once this initial search is done, the idea of the algorithm is to consecutively update $U^+_p$ and $U^+_n$ in such a way that the difference vectors $\textbf{d}_p$ and $\textbf{d}_n$ become nearly zero. We do that by selecting an intermediate state $U^+_{c}$ (current), such that:
\begin{itemize}
    \item $U^+_{c}$ is also a point of the Hugoniot locus of $U^-$;
    \item $U^+_{c}$ is a saddle;
    \item $U^+_{c}$ is such that $\sigma _c :=\sigma (U^-, U^+_{c})$ is between $\sigma _p :=\sigma (U^-, U^+_{p})$ and $\sigma _n := \sigma (U^-, U^+_{n})$. In this work we choose $U^+_{c}$ such that $\sigma _c = (\sigma _p + \sigma _n)/2$.
\end{itemize}
and replacing one of $U^+_{p}$ or $U^+_{n}$ by $U^+_{c}$. Figure\label{fig:Algo_Saddle_Saddle}(b) shows one possible such state, alongside the difference vector calculated using it. Notice that $U^+_{c}$ provides us with a difference vector $\textbf{d}_c$ that is smaller than $\textbf{d}_p$ and $\textbf{d}_n$. This only happens because $\textbf{d}_p$ and $\textbf{d}_n$ point to opposite directions ($\textbf{d}_p \cdot \textbf{d}_n$ < 0), and is the reason why the algorithm proposed next converges.

\begin{figure}[ht]
	\centering
	\subfigure[Example of the arrangement necessary to start Algorithm \ref{Alg:1}. $U^+_{p}$ and $U^+_{n}$ are such that their difference vectors $\textbf{d}_p$ and $\textbf{d}_n$ point to opposite directions ($\textbf{d}_p \cdot \textbf{d}_n$ < 0). The yellow orbits are invariant manifolds of \eqref{V9} for $U^-$ and $U^+_{p}$, while the purple orbits are invariant manifolds of \eqref{V9} for $U^-$ and $U^+_{n}$. All of the manifolds are cut short when they cross the Sotomayor line $\overline{s}$.]{\includegraphics[scale=0.27]{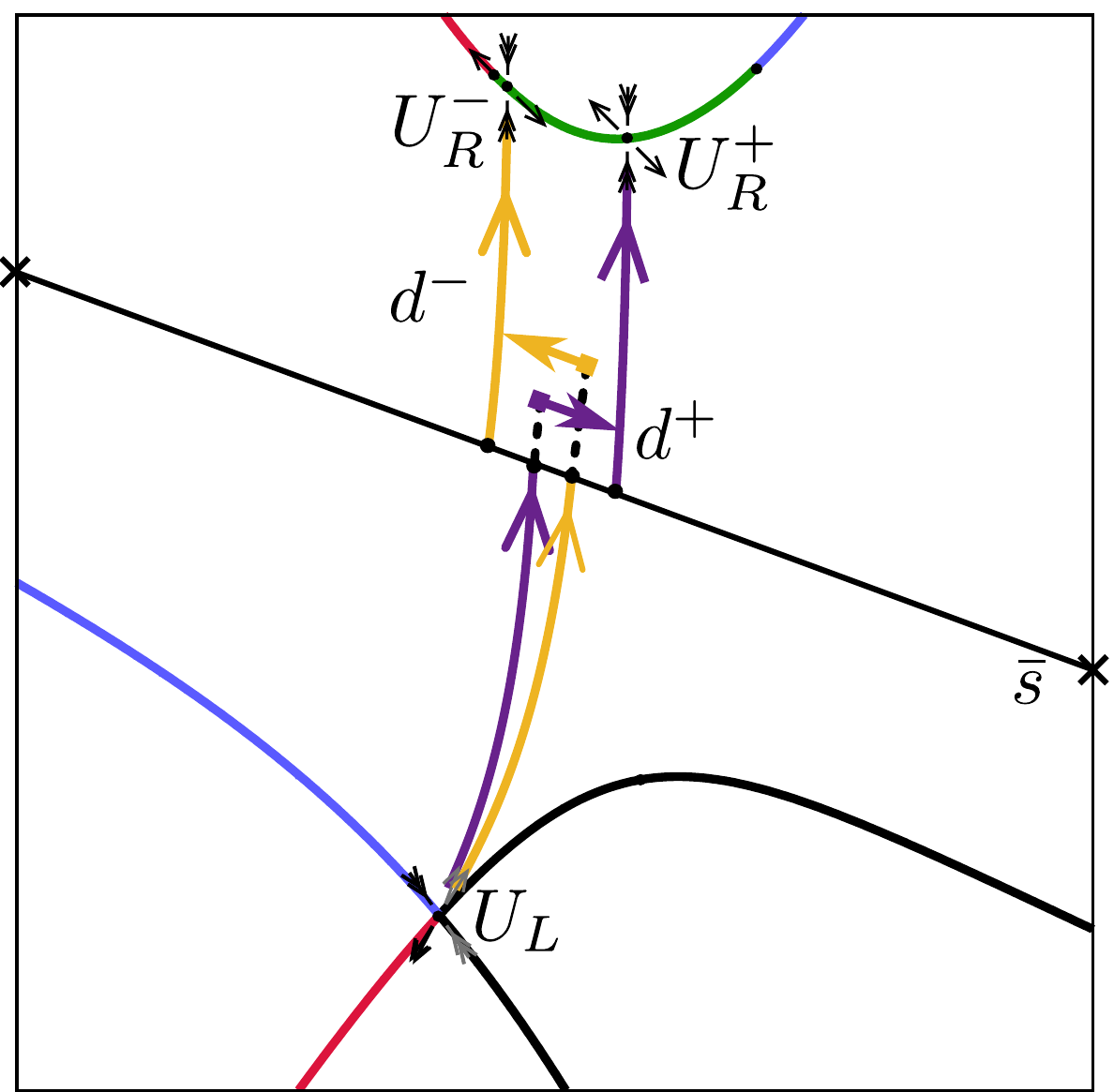}}  \hspace{0.65mm}
	\subfigure[Example of a choice of $U^+_{c}$ that allows Algorithm \ref{Alg:1} to iterate towards calculating a point that provides a saddle-to-saddle connection to $U^-$. Notice that the point is chosen so that the difference vector $\textbf{d}_c$ is smaller than $\textbf{d}_p$ and $\textbf{d}_n$. The yellow orbits are invariant manifolds of \eqref{V9} for $U^-$ and $U^+_{c}$.]{\includegraphics[scale=0.27]{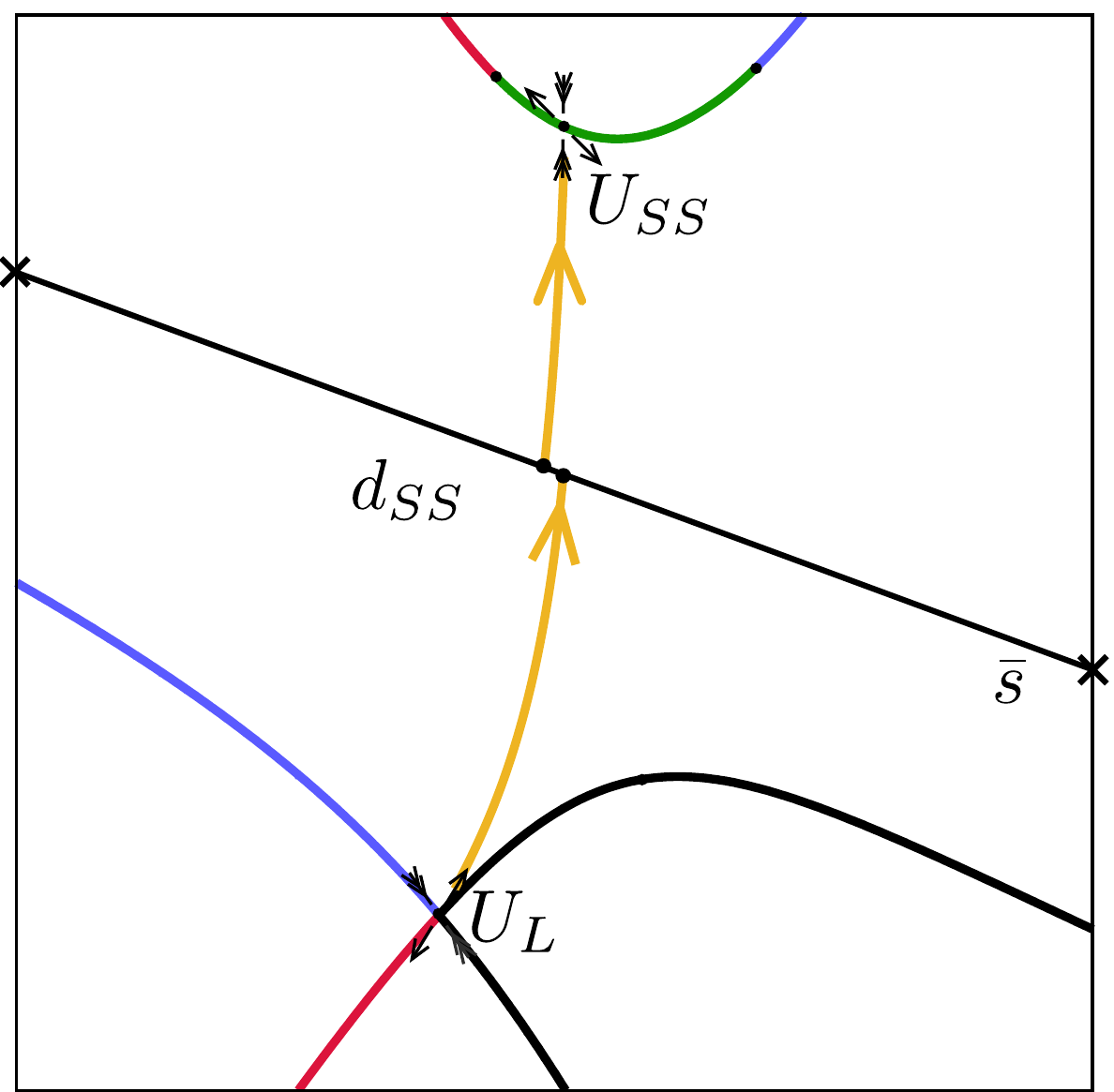}} 

	\caption{Example of an iteration of Algorithm \ref{Alg:1}, used to find saddle-to-saddle connections.}	\label{fig:Algo_Saddle_Saddle}
\end{figure}

Algorithm \ref{Alg:1} outlines the procedure used to calculate saddle-to-saddle connections. Steps 1 and 4 are the most computationally expensive, as they involve calculating all invariant manifolds for a given state and identifying those that intersect the Sotomayor line, \(\overline{s}\). Step 3 is the second most computationally intensive step, as it requires searching for a point on the Hugoniot locus of \(U^-\) where \(\sigma = (\sigma_p + \sigma_n)/2\).

\begin{algorithm}[H]
\caption{Saddle-to-saddle connection in two dimensions}
\begin{algorithmic}[1]\label{Alg:1}
\REQUIRE Saddles $U^{+}_p$, $U^{+}_n$ and $U^-$, Sotomayor line $\overline{s}$, $\delta $
\ENSURE An approximation $U^+_{\delta}$ to the point that has a saddle-to-saddle connection between itself and $U^-$
\STATE Calculate $\textbf{d}_{p}$, $\textbf{d}_{n}$
\WHILE{$\max \left\lbrace ||\textbf{d}_{p}||, || \textbf{d}_{n}|| \right\rbrace > \delta$}
	\STATE Select $U^+_c$ to be a saddle between $U^{+}_{p}$ and $U^{+}_{n}$
	\STATE Calculate $\textbf{d}_{c}$
	\IF {$\textbf{d}_{p} \cdot \textbf{d}_{c}>0$}
		\STATE $U^{+}_{p} \leftarrow  U^+_{c}$
		\STATE $\textbf{d}_{p} \leftarrow \textbf{d}_{c} $
	\ELSE
		\STATE $U^{+}_{n} \leftarrow  U^+_{c}$
		\STATE $\textbf{d}_{n} \leftarrow \textbf{d}_{c}$
	\ENDIF
\ENDWHILE
\STATE $U^+_{\delta} \leftarrow U^+_{c}$
\end{algorithmic}
\end{algorithm}

\subsection{Undercompressive Boundary point}
\begin{figure}[ht]
	\centering
	\subfigure[Zoom in the phase state]{\includegraphics[scale=0.27]{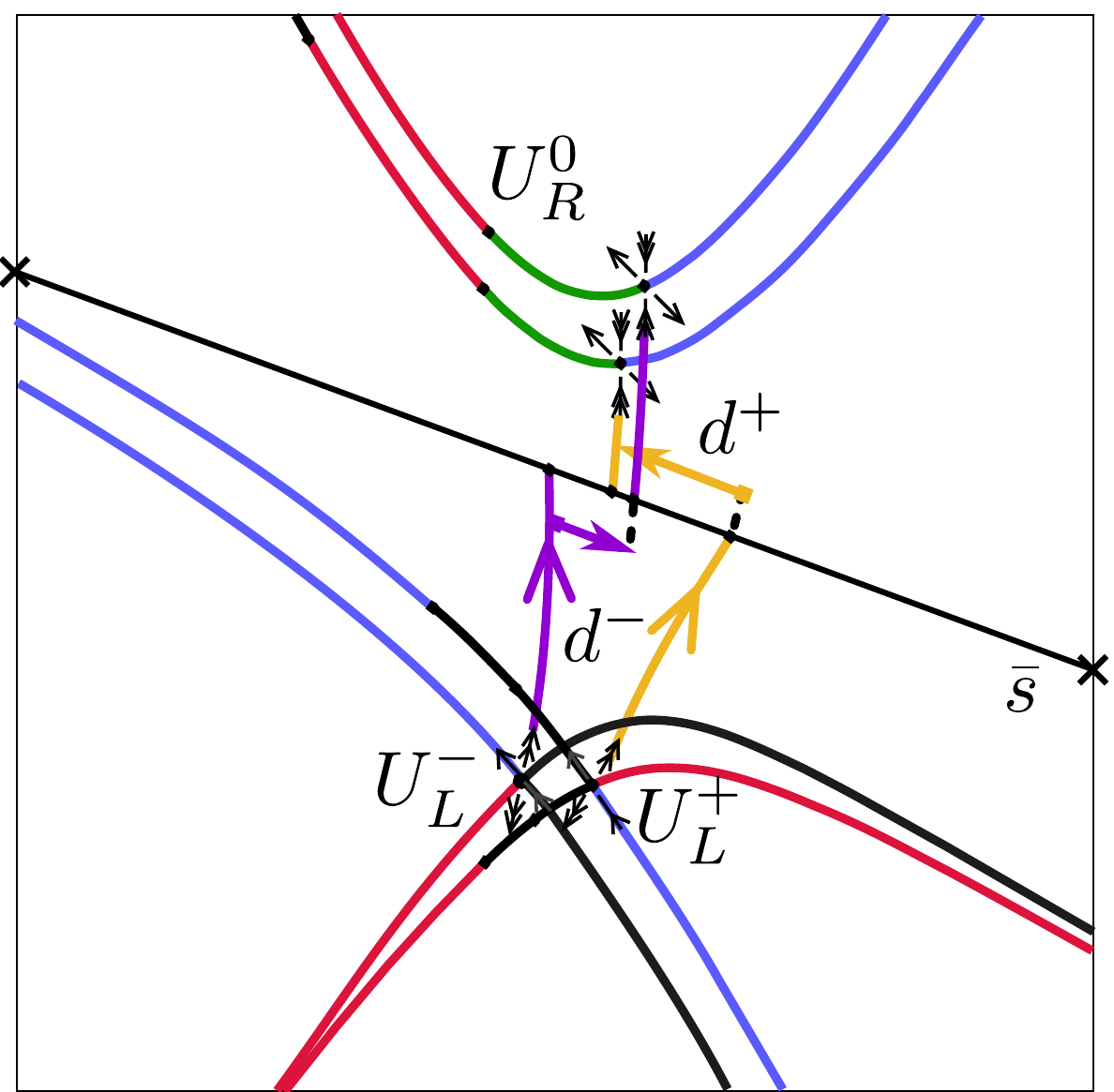}}  \hspace{0.55mm}
	\subfigure[Zoom in the phase state]{\includegraphics[scale=0.27]{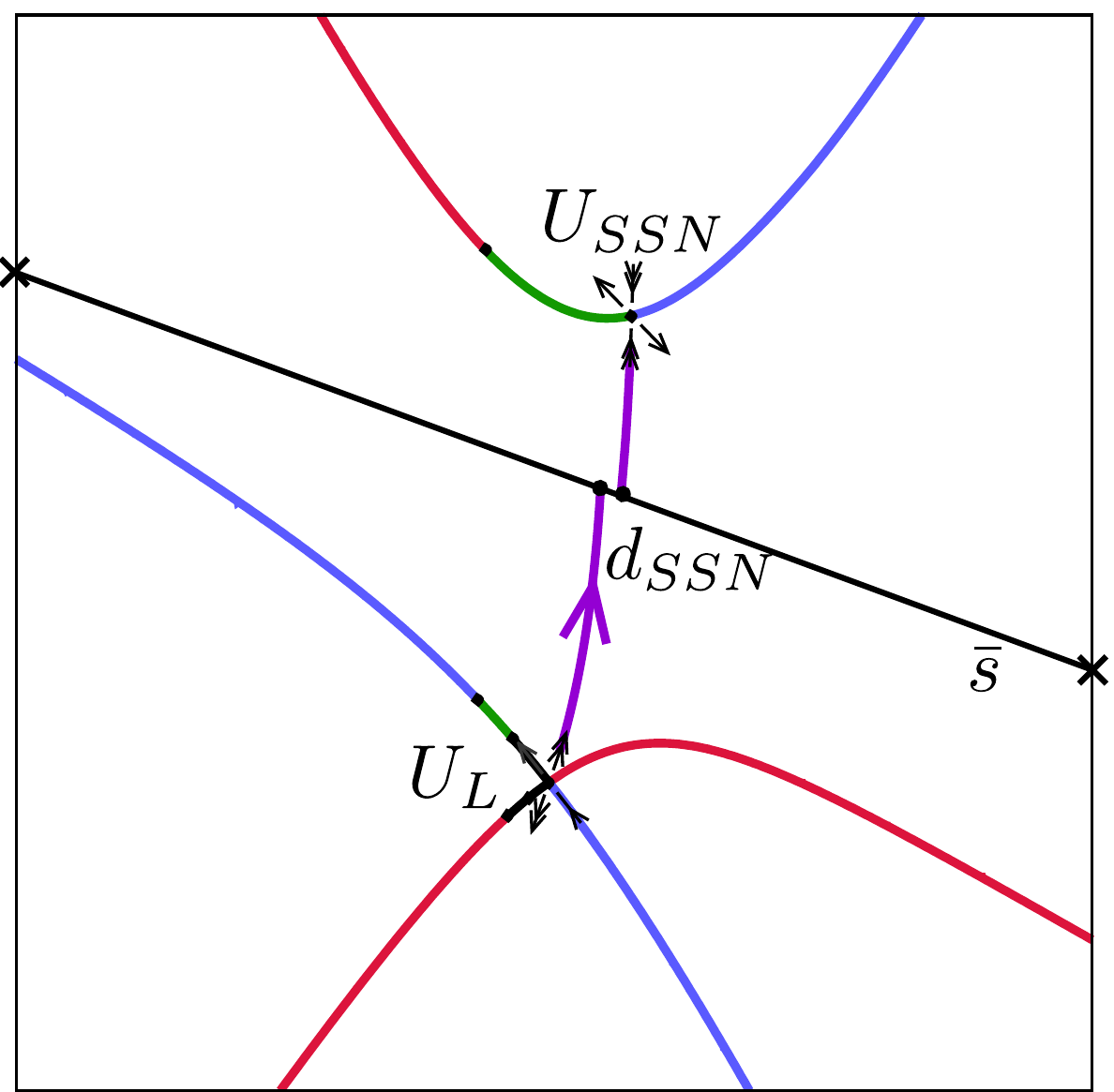}} 

	\caption{ }	\label{fig:Algo_Saddle_Saddle_node}
\end{figure}
\begin{algorithm}[H] \label{Alg:2}
\caption{Undercompressive Boundary point}
\begin{algorithmic}[1]
\REQUIRE Left Sates $U^{-}_L$ providing SS connection and $U^{+}_L$ not providing SS connection, transversal line segment $\overline{s}$, $\delta $
\ENSURE Left State $U_{L}$ at the boundary of the undercompressive region according to $\delta $
\WHILE{$|| U^{+}_{L}-U^{-}_L || > \delta$}
	\STATE $U_{L} \leftarrow  U^{-}_L+\frac{1}{2}(U^{+}_{L}-U^{-}_L)$
	\IF {$U_{L}$ has a saddle-to-saddle connection}
		\STATE $U^{-}_L \leftarrow U_{L}$
	\ELSE
		\STATE $U^{+}_L \leftarrow U_{L}$
	\ENDIF
\ENDWHILE
\STATE $U_{L} \leftarrow  U^{-}_L+\frac{1}{2}(U^{+}_{L}-U^{-}_L)$
\end{algorithmic}
\end{algorithm}
\subsection{saddle-to-saddle-Node Connection }
\begin{algorithm}[H]\label{Alg:3}
\caption{saddle-to-saddle-Node Connection  in two dimensions}
\begin{algorithmic}[1]
\REQUIRE $U^{-}_L$, $U^{+}_L$, saddle-node $U^0_R$, transversal line segment $\overline{s}$, $\delta $
\ENSURE  $U_{L}$ and $U_{SSN}$ providing a SSN conn. according to $\delta $

\STATE $U_{L} \leftarrow  U^{-}_L+\frac{1}{2}(U^{+}_{L}-U^{-}_L)$
\STATE Select $U_{SSN}$ based on $U_{L}$ and $U^0_R$
\STATE Calculate $\textbf{d}^{-}$, $\textbf{d}^{+}$ (difference vectors)
\WHILE{$\max \left\lbrace || \textbf{d}^{-} ||, || \textbf{d}^{+} || \right\rbrace > \delta$}
	\STATE $U_{L} \leftarrow  U^{-}_L+\frac{1}{2}(U^{+}_{L}-U^{-}_L)$
	\STATE Select $U_{SSN}$ based on $U_{L}$ and $U^0_R$
	\STATE Calculate $\textbf{d}_{SSN}$ (difference vector)
	\IF {$\textbf{d}^{-} \cdot \textbf{d}_{SSN}>0$}
		\STATE $U^{-}_{L} \leftarrow  U_{L}$
		\STATE $\textbf{d}^{-} \leftarrow \textbf{d}_{SSN} $
	\ELSE
		\STATE $U^{+}_{L} \leftarrow  U_{L}$
		\STATE $\textbf{d}^{+} \leftarrow \textbf{d}_{SS} $
	\ENDIF
\ENDWHILE
\end{algorithmic}
\end{algorithm}






\end{document}